\documentclass[12pt]{article}%
\pdfoutput=1
\usepackage{amsfonts}
\usepackage{amsmath}
\usepackage{amssymb}
\usepackage{graphicx}%
\providecommand{\U}[1]{\protect\rule{.1in}{.1in}}
\newtheorem{theorem}{Theorem}

\newtheorem{corollary}[theorem]{Corollary}

\newtheorem{lemma}[theorem]{Lemma}

\newtheorem{proposition}[theorem]{Proposition}

\newenvironment{proof}[1][Proof]{\noindent\textbf{#1.} }{\ \rule{0.5em}{0.5em}}
\begin{document}

\title{The Whittaker Plancherel Theorem}
\author{Nolan R. Wallach}
\maketitle

\begin{abstract}
The purpose of this article is to give an exposition of a proof of the
distributional form of the Whittaker Plancherel Theorem. The proof is an
application of Harish-Chandra's Plancherel Theorem for real reductive groups
and its exposition can be used as an introduction to Harish-Chandra's Plancherel Theorem.
The paper follows the basic method in the author's original approach in his
second volume on real reductive groups. An error in the calculation of the
Whittaker Transform of a Harish-Chandra wave packet is fixed using a result of
Rapha\"{e}l Beuzart-Plessis.

\end{abstract}

\section{Introduction}

This paper has its roots in the final chapter of my Real Reductive Groups II.
Its goal was to show how to use the Harish-Chandra Plancherel theorem for real
reductive groups to decompose the unitarily induced representation from a
generic one dimensional character of a maximal unipotent subgroup (the
Whittaker Plancherel Theorem). This work has been widely cited by researchers
in the theory of automorphic forms but, unfortunately, there are several
serious gaps in the main arguments. The first was indicated in \cite{BanK}.
Which made me aware that the theory of cusp forms in the Whittaker case is
much more subtle than in the Harish-Chandra theory. After I fixed that error
in the book\ the first referee of the corresponding paper found a new gap
which was promulgated from the original. The referee of the second version of
the manuscript found a related gap which was also implicit in the original
version in \cite{RRGII} as did the referee of the third version. This version
gives a complete proof of the theorem based on the basic ideas in my original
proof augmented by a powerful convergence theorem of Beuzart-Plessis
\cite{raphael}.

The two volumes on Real Reductive Groups were written over a period of 15 years
starting in 1977, with the first volume appearing in 1987 and the second in
1992 the work on Whittaker vectors was done in the early 1980's at the same
time that Harish-Chandra was doing his work. Recently V.S.Varadarajan and R.
Gangolli have edited Volume V \cite{CollectedV}\ of Harish-Chandra's collected
work and have, in a heroic piece of work, organized a collection of his
unpublished notes of his work on the Whittaker Plancherel Theorem as part III
of the volume. As it turns out neither of us completed the proof of the
theorem although we did independently find the same statements down to th
same definition of the Whittaker Schwartz space. The key missing ingredient
has to do with the completeness of the space of wave packets.

As indicated in the beginning of this introduction one of the main goals of
this paper is to fix the gaps and errors in Chapter 15 of \cite{RRGII}. I will
refer to results in that chapter that are correct when the defective results
are replaced by those in this paper without additional proof. In this paper I
will follow the tradition of Harish-Chandra and call the inversion formula the
Plancherel Formula (see Theorem \ref{plancherel} in this paper for
Harish-Chandra's statement). In Corollary \ref{12-version} I show how the
theorem implies an $L^{2}$ version which could be interpreted as an analog of
the Parseval Formula.

Here is a description of the \ flow of this paper. The purpose of Section 2 is
to set up the notation used in the rest of the paper. Section 3 gives a
description of Harish-Chandra's work on what he calls cusp forms for the
group, in the context of the Harish-Chandra Schwartz space, and their
relationship with the discrete series. Section 4 begins the development of the
Whittaker version of the Schwartz space. Theorem \ref{15.3.Replace}, combined
with Corollary \ref{goodenough} in that section replaces the error found by
\cite{BanK}. Section 5 gives a proof of the Whittaker version of
Harish-Chandra's theorem asserting that an eigenfunction for the bi-invariant
differential operators that is in the Schwartz space is a cusp form. Section 6
gives an exposition of the Harish-Chandra Plancherel theorem. The reader
should be warned that in \cite{RRGII} the Harish-Chandra Plancherel density,
$\mu(\omega,i\nu)$, needs several normalizing constants. In this paper we will
be \textquotedblleft sloppy\textquotedblright\ about normalizing constants and
measures. These aspects were handled in \cite{RRGII}. The density we write as
$\mu(\omega,\nu)$ and it contains all of these normalizations. Theorem
\ref{HCimplication} is the most important part of the Harish-Chandra
Plancherel Theorem for the purposes of the proof of the Whittaker version of
the theorem. Section 7 \ involves a major deviation from \cite{RRGII} giving a
decomposition of the Whittaker Schwartz space in terms of Theorem
\ref{HCimplication}. Section 8 gives a description of the holomorphic
continuation of the Jacquet integrals. It also proves key tempered estimates
for this analytic continuation. Section 9 contains a critical result on the
Fourier transform of wave packets leading our proof of the relationship
between the discrete series for $G$ and the Whittaker discrete series. The
paper next approaches the continuous spectrum. This is where the result in
\cite{raphael} plays an important role. Section \ref{Final} contains our
version of the Whittaker Plancherel Formula. This version is new to the paper
and is based on a preliminary result in Section \ref{Firstform} that is also
new to this paper.

I heartily thank the readers of the first, second and third versions of this
paper for their detailed discussion of key flaws. Many aspects of the second
volume of \cite{RRGII} were \textquotedblleft written in a
vacuum\textquotedblright. \cite{RRGI} contained many complicated results, but
almost all of them could be given an exposition in a graduate course or a
technical seminar and, in fact, they were. \cite{RRGII} was a different story.
A year long graduate course could not handle the shear technical difficulty of
either the Whittaker Plancherel theorem or the Harish-Chandra Plancherel
Theorem. Thus the added filter was not possible for these highly technical,
difficult theorems. The readers played the role of an exceptionally qualified,
helpful and critical audience. Also, I placed arXiv:2206.14773 on the arXiv
with the hope that someone would come forth with a general proof of the
convergence theorem that was critical to my way of calculating the Jacquet
integral of a wave packet. I thank Rapha\"{e}l Beuzart-Plessis for responding
with a reference to his proof and thus closing a long painful story. I also
thank the referee of the present version of this paper for an immense and very
helpful piece of work. In addition, I thank Herv\'{e}\ Jacquet for his many
clarifying suggestions, perceptive criticism and for his kindness.

\section{Notation}
We will be studying the class of real reductive groups $G$ as in \cite{BoWa}
and \cite{RRGI},\cite{RRGII}. Recall that, there is an algebraic subgroup $G_{1}$
of $GL(n,\mathbb{R})$ (for some $n$) invariant under transpose such that $G$
is a finite covering of an open subgroup of $G_{1}$. Let $p:G\rightarrow
G_{1}$ be the covering map.

Let $K$ be a maximal compact subgroup of $G$ (which we can take as
$p^{-1}(G_{1}\cap O(2n))$) and let $\theta$ be the corresponding Cartan
involution. Let $G=KA_{o}N_{o}$ be an Iwasawa decomposition of $G$. Let
$P_{o}$ be the normalizer of $N_{o}$ in $G$. The minimal parabolic subgroups
are the $G$ conjugates of $P_{o}$. Since $G=KP_{o}$ they are all $K$
conjugate. A parabolic subgroup of $G$ is a subgroup that is its own
normalizer containing a minimal parabolic subgroup. Let $P$ be a parabolic
subgroup. We set $M_{P}=P\cap\theta P$ and $N_{P}$ equal to the unipotent
radical of $P$. Then $P=M_{P}N_{P}$ is called a ($K$--) standard Levi
decomposition. Let $C$ denote the center of $M_{P}$. If $\mathfrak{a}%
_{P}=\{x\in Lie(C)|\theta x=-x\}$ then we set $A_{P}=\exp(\mathfrak{a}_{P})$.
Note that
\[
\exp:\mathfrak{a}_{P}\rightarrow A_{P}%
\]
is an isomorphism of Lie groups ($\mathfrak{a}_{P}$ a Lie group under
addition). Define (as usual) $\log:A_{P}\rightarrow\mathfrak{a}_{P}$ to be the
inverse of $\exp$. Set $^{o}M_{P}=\cap_{\mu}\ker\mu$, the intersection is over
all Lie group homomorphisms
\[
\mu:M_{P}\rightarrow\mathbb{R}_{>0},
\]
then $M_{P}={}^{o}M_{P}A_{P}$ and $P=^{o}M_{P}A_{P}N_{P}$ with unique
decomposition. Note that $^{o}M_{P}$\ is in the class of groups that we are
studying for all parabolic subgroups including $G$. We will also use the
notation $^{o}G$ for $^{o}M_{G}$. The above decomposition is usually called a
Langlands decomposition. We will use the notation $\overline{P}=\theta P$ the
standard opposite parabolic to $P$. We note that $\overline{P}={}^{o}%
M_{P}A_{P}\theta\left(  N_{P}\right)  $ and we write $\theta\left(
N_{P}\right)  =\overline{N}_{P}$ . \ Applying this standard material to $P_{o}
$ and using the notation $X_{P_{o}}=X_{o}$ for $X=M,A,N,\mathfrak{a}$ \ we
have $P_{o}={}^{o}M_{o}A_{o}N_{o}$. Also the Iwasawa decomposition implies
that $G=N_{o}A_{o}K$ with unique decomposition. So if $g\in G$ then
$g=n(g)a(g)k(g)$ with $n(g)\in N_{o},a(g)\in A_{o},k(g)\in K$. We will call a
parabolic subgroup standard if it contains $P_{o}$.

As usual, $U(Lie(G))$ denotes the universal enveloping algebra of $Lie(G)$. We
look upon $U(Lie(G))$ as the left invariant differential operators on $G$. We
use the standard notation $Z_{G}(Lie(G))$ for $Ad(G)$--invariants in
$U(Lie(G))$ and $Z(Lie(G))$ for the center of $U(Lie(G))$. \ Note that since
$G$ has a finite number of connected components $Z(Lie(G))$ is a finitely
generated $Z_{G}(Lie(G))$ module under multiplication.

Let $n=\dim N_{o}$ and let $\mu(g)=\wedge^{n}Ad(g)$and let
\[
V=\mathrm{Span}\mu(g)\wedge^{n}Lie(N_{o})\subset\wedge^{n}Lie(G)
\]
then $(\mu,V)$ is a real representation of $G$. On $Lie(G)$ we put an inner
product on $Lie(G)$ such that if $\theta X=-X$ then $adX$ is self adjoint and
such that $K$ acts orthogonally. One can show that this inner product can be
taken to be of the form $-B(\theta X,Y)$ with $B$ a symmetric bilinear form on
$Lie(G)$ such that
\[
B(Ad(g)X,Ad(g)Y)=B(X,Y).
\]
On $Lie(G)$ (and any subspace thereof we will use the norm $\left\Vert
X\right\Vert =(-B(\theta X,X))^{\frac{1}{2}}$. The inner product naturally
induces an inner product on $\wedge^{n}Lie(G)$. Let $\left\Vert ...\right\Vert
$ denote the operator norm on $\mathrm{End}(\wedge^{n}Lie(G)) $. Let
$G=A_{G}{}^{o}G$ be the $K$--standard Langlands decomposition of $G$. Define
for $g\in G,g=ag_{1}$ with $a\in A_{G},g_{1}\in{}^{o}G$, $\left\Vert
g\right\Vert =e^{\left\Vert \log a\right\Vert }\left\Vert \mu(g_{1}%
)_{|V}\right\Vert ^{\frac{1}{2}}$

Then $\left\Vert ...\right\Vert $ satisfies the following four properties

1. $\left\Vert xy\right\Vert \leq\left\Vert x\right\Vert \left\Vert
y\right\Vert ,x,y\in G.$

2. $\left\Vert x\right\Vert =\left\Vert x^{-1}\right\Vert ,x\in G.$

3. The sets $B_{r}=\{x\in G|\left\Vert x\right\Vert \leq r\}$ for $0<r<\infty$
are compact.

4. $\left\Vert k_{1}xk_{2}\right\Vert =\left\Vert x\right\Vert $ for
$k_{1},k_{2}\in K,x\in G$.

Any function that satisfies these properties will be called a norm.

Using these properties one has

5. If $(\pi,H)$ is a Banach representation of $G$ with operator norm
$\left\Vert ...\right\Vert _{H}$ then there exist $C>0$ and $r>0$ such that
$\left\Vert \pi(x)\right\Vert _{H}\leq C\left\Vert x\right\Vert ^{r},x\in G$
(see {\cite{RRGI}} 2.A.2. p. 71).

The norm we just constructed will be called the standard norm.

We note that for $1.$ implies that $\left\Vert e\right\Vert \geq1$($e$ the
identity element in $G$) and so
\[
1\leq\left\Vert e\right\Vert =\left\Vert gg^{-1}\right\Vert \leq\left\Vert
g\right\Vert ^{2}%
\]
thus $\left\Vert g\right\Vert \geq1$. Also note

\begin{lemma}
\label{basic-inequality}There exists a constant $C>0$ such that if $a\in
A_{o}$ then
\[
C\left\Vert \log a\right\Vert \leq\log\left\Vert a\right\Vert \leq\left\Vert
\log a\right\Vert .
\]

\end{lemma}

If $P$ is a parabolic subgroup of $G$ and if $\mathfrak{a}_{P}$ is above then
we define (as usual) $\rho_{P}(H)=\frac{1}{2}\mathrm{tr}(ad(H)_{|Lie(N_{P})})$
for $H\in\mathfrak{a}_{P}$. We will set $\rho=\rho_{P_{o}}$. If $\lambda
\in\mathfrak{a}_{o}^{\ast}$ then we set $a^{\lambda}=e^{\lambda(\log a)}$.

Let $\Phi(P,A_{P})$ denote the weights of $\mathfrak{a}_{p}$ acting on
$Lie(N_{P})$. Set
\[
\mathfrak{a}_{P}^{+}=\{H\in\mathfrak{a}_{P}|\alpha(H)>0,\alpha\in\Phi
(P,A_{P})\}.
\]

\begin{lemma}
\label{rho-estimate}Assume that $G$ has compact center then if $\left\Vert
...\right\Vert $ is the standard norm on $G$ and $H\in\overline{\mathfrak{a}%
_{o}^{+}}$ then $\left\Vert \exp H\right\Vert =e^{\rho_{o}(H)}$ with
\[
\rho_{o}(H)=\frac{1}{2}\mathrm{tr}ad(H)_{|Lie(N_{o})}%
\]

\end{lemma}

Let $P$ be a parabolic subgroup of $G$ and let $(\sigma,H^{\sigma})$ be a
Hilbert representation of $^{o}M_{P}$ such that the underlying ($Lie(^{o}%
M_{P}),{}^{o}M_{P}\cap K)$ is admissible and finitely generated and
$(\sigma,H^{\sigma})$ is unitary as a representation of $^{o}M_{P}\cap K$ and
let $\nu\in\left(  \mathfrak{a}_{P}^{\ast}\right)  _{\mathbb{C}}$. We form the
smooth induced representation $I_{P,\sigma,\nu}^{\infty}$ as follows: The
underlying Fr\'{e}chet space is
\[
I_{\sigma}^{\infty}=\{f:K\rightarrow H_{\sigma}^{\infty}|f\text{ is }%
C^{\infty}\text{ and }f(mk)=\sigma(m)f(k),m\in{}^{o}M_{P}\cap K,k\in K\}.
\]
If $f\in I_{\sigma}^{\infty}$ then we define
\[
\left(  \pi_{P,\sigma,\nu}(g)f\right)  (k)={}_{P}f_{\nu}(kg)
\]
with
\[
_{P}f_{\nu}(namk)=a^{\nu+\rho_{P}}\sigma(m)f(k)
\]
if $n\in N_{P},a\in A_{P},m\in{}^{o}M_{P}$ and $k\in K$, $\rho_{P}(h)=\frac
{1}{2}tradh_{|_{Lie(N_{p})}}$and
\[
\left\langle f_{1},f_{2}\right\rangle =\int_{K}\left\langle f_{1}%
(k),f_{2}(k)\right\rangle _{\sigma}dk.
\]
The Hilbert space $I_{\sigma}$ is the completion of $I_{\sigma}^{\infty}$ with
respect to this inner product and the operators $\pi_{P,\sigma,\nu}(g)$ extend
to yield a Hilbert representation of $G$.

A fundamental inequality of Harish-Chandra (c.f. \cite{RRGI}, Lemma 4.A.2.3
(2) with the misprint $\left\Vert \log n\right\Vert $ replaced by
$\log\left\Vert n\right\Vert $, see also Lemma \ref{HCinequality} which is
part (1) of the Lemma) implies

\begin{lemma}
\label{HCLemma}There exist $C_{1},C_{2}>0$ such that
\[
C_{1}(1+\log\left\Vert n\right\Vert )\leq(1+\rho_{P_{o}}(\log a_{\bar{P}_{o}%
}(n))\leq C_{2}(1+\log\left\Vert n\right\Vert )
\]
for all $n\in N_{o}$.
\end{lemma}

\section{The Harish-Chandra Schwartz Space}

We keep the notation of the previous section. The purpose of this one is to
give a tour of Harish-Chandra's most profound results involving the discrete
series. We follow the discussion in \cite{RRGI} Chapter 7.

First we recall the definition of the Harish-Chandra Schwartz space. If $f\in
C^{\infty}(G),\,k\in\mathbb{R}_{\mathbb{\geq}}=[0,\infty),x,y\in U(Lie(G))$
then set
\[
p_{k,x,y}(f)=\sup_{g\in G}\left(  (1+\log\left\Vert g\right\Vert )^{k}%
\Xi(g)^{-1}\left\vert R_{y}L_{x}f(g)\right\vert \right)  .
\]
Here $R_{y}$ is the right regular action so $R_{y}f=yf$ as a left invariant
differential operator and $L_{x}$ is the left regular action (so a right
invariant differential operator) and $\Xi$ is Harish-Chandra's basic spherical
function which in particular is $K$ bi-invariant. The key facts we need about
$\Xi$ are in \cite{RRGI}, Section 4.5.

1. If $g\in G,$write $g=n(g)a(g)k(g)$ with $n(g)\in N_{o},a(g)\in A_{o}$ and
$k(g)\in K$. Then $\Xi(g)=\int_{K}a(kg)^{\rho_{o}}dk$ (the integration is over
normalized invariant measure on $K$).

2. There exist $C,d>0$ such that if $h\in\overline{\mathfrak{a}_{o}^{+}}$
then
\[
e^{-\rho(h)}\leq\Xi(\exp(h))\leq Ce^{-\rho(h)}(1+\left\Vert h\right\Vert
)^{d}.
\]

3. $\Xi(g)>0$ and there exists $r>0$ such that ($dg$ is a Haar measure on $G)
$
\[
\int_{G}(1+\log\left\Vert g\right\Vert )^{-r}\Xi(g)^{2}dg<\infty.
\]

4. If $P$ is a parabolic subgroup of $G$ and $N=N_{P}$ then there exists $r$
such that ($dn$ is a Haar measure on $N)$
\[
\int_{N}(1+\log\left\Vert n\right\Vert )^{-r}\Xi(n)dn<\infty.
\]

5. Since it is a $K$ bi-invariant spherical function it satisfies
\[
\int_{K}\Xi(xky)dk=\Xi(x)\Xi(y)
\]
(indeed this is one of the definitions of zonal spherical function).

6. $\Xi(xy)=\int_{K}a(kx^{-1})^{\rho}a(ky)^{\rho}dk$ for $x,y\in G$ (c.f.
\cite{RRGI} (2) p. 147).

7. Assume that $G$ has compact center and $\left\Vert ...\right\Vert $ is the
standard norm (see Section 2). Then Lemma \ref{rho-estimate} implies that
there exist $C_{1},C_{2},d>0$ such that
\[
C_{1}\left\Vert g\right\Vert ^{-1}\leq\Xi(g)\leq C_{2}\left\Vert g\right\Vert
^{-1}(1+\log\left\Vert g\right\Vert )^{d}%
\]
since $G=K\exp(\overline{\mathfrak{a}_{o}^{+}})K$.

Using 6. one sees easily that if $Y$ is a compact subset of $G$ then there
exist positive constants $M_{Y},L_{Y}$ such that if $y\in Y$ then
\[
L_{Y}\Xi(x)\leq\Xi(xy)\leq M_{Y}\Xi(x).
\]
Indeed, we can choose $L_{Y},M_{Y}$ such that $0<L_{Y}<a(ky)^{\rho}\leq M_{F}
$ for $y\in Y$ since $KY$ is compact and $a(x)^{\rho}>0$. Thus proves
\[
L_{Y}\int_{K}a(kx^{-1})^{\rho}dk\leq\Xi(xy)\leq M_{Y}\int_{K}a(kx^{-1})^{\rho
}dk=M_{Y}\Xi(x).
\]

We will say that $f\in C^{\infty}(G)$ satisfies the weak inequality if there
exist $C,d$ such that
\[
\left\vert f(g)\right\vert \leq C\Xi(g)(1+\log\left\Vert g\right\Vert )^{d}.
\]

The Harish-Chandra Schwartz space, $\mathcal{C}(G)$, is the subspace of
$C^{\infty}(G)$ consisting of those functions $f$ such that
\[
p_{k,x,y}(f)<\infty
\]
for all $x,y\in U(Lie(G))$ and all $k\geq0$ endowed with the topology given by
the semi-norms $p_{k,x,y}$. With this topology $\mathcal{C}(G)$ is a
Fr\'{e}chet space and an algebra under convolution.

In this context the basis of Harish-Chandra's \textquotedblleft philosophy of
cusp forms\textquotedblright\ is encapsulated in the following sequence of
results. Let $P$ be a standard parabolic subgroup of $G$. We define for $f\in
C(G)$
\[
f^{P}(m)=a_{P}(m)^{-\rho_{P}}\int_{N_{P}}f(nm)dn
\]
for $m\in M_{P}$ such that the integral converges. Until further notice all of
the coming results in this section are due to Harish-Chandra. We will give
references to \cite{RRGI}.

\begin{theorem}
\label{cusp-to-cusp}(c.f. \cite{RRGI}Theorem 7.2.1) If $f\in\mathcal{C}(G)$
then the integral defining $f^{P}$ converges absolutely and uniformly in
compacta of $M$ and defines an element of $\mathcal{C}(M)$. Furthermore, the
map $f\mapsto f^{P}$ is continuous from $\mathcal{C}(G)$ to $\mathcal{C}(M)$.
\end{theorem}

This follows from the following result of Harish-Chandra (c.f. \cite{RRGI} p.
233.). Set $\Xi_{^{0}M_{P}}$ equal to the analogue of $\Xi$ for $^{0}M_{P} $
and $K\cap M_{P}$ then

\begin{proposition}
\label{basic}If $u,v>0$ are given then there exists $r>0$ and $C_{u,v}$ such
that if $m\in{}^{0}M_{P},a\in A_{P}$
\[
a^{-\rho_{P}}\int_{N_{P}}\Xi(nam)(1+\log\left\Vert nam\right\Vert )^{-r}dn\leq
\]
\[
\Xi_{^{0}M_{P}}(m)(1+\log\left\Vert m\right\Vert )^{-u}(1+\log\left\Vert
a\right\Vert )^{-v}.
\]

\end{proposition}

If $f^{P}=0$ for all proper parabolic subgroups of $G$ then we call $f$ a cusp
form. Let $Z(Lie(G))$ denote the center of $U(Lie(G))$ then

\begin{theorem}
\label{Z-finite} (c.f. \cite{RRGI},Theorem 7.7.2) If $f\in\mathcal{C}(G)$ is
such that $\dim Z(Lie(G))f<\infty$ then $f$ is a cusp form.
\end{theorem}

Let $\hat{K}$ be the set of equivalence classes of irreducible continuous
representations of $K$. Let $\gamma\in\hat{K}$ and let $\chi_{\gamma}$ and
$d(\gamma)$ be respectively the character and dimension of any representative
of $\gamma$. Defining for $f\in C(G)$
\[
E_{\gamma}(f)(g)=f_{\gamma}(g)=d(\gamma)\int_{K}f(gk)\chi_{\gamma}(k^{-1})dk
\]
then under the right regular action of $K$ on $C(G)$ the function $f_{\gamma}
$ transforms as a representative of $\gamma$. \ The following result is easily
derivable from \cite{RRGI} Corollary 3.4.7 and is used in the proof of the
preceding theorem.

\begin{theorem}
\label{K-Fourier}If $f\in\mathcal{C}(G)$ then $f_{\gamma}\in\mathcal{C}(G) $
and the series
\[
\sum_{\gamma\in\hat{K}}f_{\gamma}%
\]
converges to $f$ in $\mathcal{C}(G)$. Furthermore, if $G$ has compact center
and $\dim Z(Lie(G))f<\infty$ then the $(Lie(G),K)$ modules
\[
U(Lie(G)_{\mathbb{C}})Span_{\mathbb{C}}R_{K}f_{\gamma}%
\]
are admissible and are finite direct sum of the underlying $(Lie(G),K)$
modules of irreducible square integrable representations.
\end{theorem}

The last part of the above theorem needs an explanation. An irreducible
unitary representation of $G$, $(\pi,H)$, is said to be square integrable if
the matrix coefficients
\[
g\mapsto\left\langle \pi(g)v,w\right\rangle
\]
are square integrable for all $v,w\in H$. In fact, all one needs is one
non-zero square integrable matrix coefficient.

\begin{theorem}
(c.f. Theorem 5.5.4 \cite{RRGI}) Assume that $G$ has compact center. If
$(\pi,H)$ is an irreducible square integrable representation of $G$ and if
$v,w$ are $K$--finite elements of $H$ then the corresponding matrix
coefficient, $f(g)=\left\langle \pi(g)v,w\right\rangle $, is in $\mathcal{C}%
(G)$ and $\dim Z_{G}(Lie(G))f=1$.
\end{theorem}

Using the Casselman-Wallach (CW) theorem in form of (Theorem 11.8.2
\cite{RRGII}) we have the following strengthening of the above theorem.

\begin{corollary}
If $(\pi,H)$ is an irreducible square integrable representation of $G$ and if
$v,w\in H^{\infty}$then the function $g\longmapsto\left\langle \pi
(g)v,w\right\rangle $ is in $\mathcal{C}(G)$.
\end{corollary}

\begin{proof}
We recall the space $\mathcal{S}(G)$ (see 7.1.2 \cite{RRGI}) the space of all
elements of $C^{\infty}(G)$ such that $\sup_{g\in G}\left\Vert g\right\Vert
^{d}\left\vert R_{x}f(g)\right\vert <\infty$ for all $x\in U(Lie(G))$ (thought
of as a left invariant differential operator) and all $d. $ Then
$\mathcal{S}(G)$ is a convolution algebra. If $(\pi,H)$ is a Hilbert
representation of $G$ then we can define an action of $\mathcal{S}(G) $ on
$H^{\infty}$ by
\[
\int_{G}f(g)\pi(g)vdg=\pi(f)v\text{.}%
\]
Theorem 11.8.2 in \cite{RRGII} implies that if $(\pi,H)$ is irreducible and
admissible then $\mathcal{S}(G)$ acts algebraically irreducibly on $H^{\infty
}. $ Thus in particular, $\pi(\mathcal{S}(G))u=H^{\infty}$. If $u\in H_{K}$
and $u\neq0$. \ Now assume that $(\pi,H)$ is as in the previous theorem. If
$v,w\in H^{\infty}$ and $u\in H_{K}$ is such that $u\neq0$ then exist
$f_{1},f_{2}\in\mathcal{S}(G)$ such that $\pi(f_{1})u=v,\pi(f_{2})u=w.$ So if
$c_{v,w}(g)=\left\langle \pi(g)v,w\right\rangle $ then
\[
c_{v},_{w}=\check{f}_{2}\ast c_{u,u}\ast f_{1}%
\]
with $\check{f}_{2}(g)=\overline{f_{2}(g^{-1})}$. Since $\mathcal{S}%
(G)\subset\mathcal{C}(G)$, and $\mathcal{C}(G)$ is closed under convolution
$c_{v,w}\in\mathcal{C}(G)$.
\end{proof}

A simple limiting argument implies that if $\mathcal{C}_{\pi}(G)$ is the
closure in $\mathcal{C}(G)$ of the span of the $c_{v,w}$ with $v,w$
$K$--finite elements of $H$ then

\begin{corollary}
With the notation above, the space $\mathcal{C}_{\pi}(G)$ is contained in the
space of cusp forms.
\end{corollary}

We are now ready to close the circle and describe one of Harish-Chandra's
deepest results. We set $\mathcal{E}_{2}(G)$ equal to the set of irreducible
square integrable representations of $G$ (obviously, in light of Schur's
Lemma, $\mathcal{E}_{2}(G)=\emptyset$ if the center of $G$ is not compact). If
$\omega\in\mathcal{E}_{2}(G)$ and if $(\pi,H)$ is a representative of $\omega$
then we set $\mathcal{C}_{\omega}(G)=\mathcal{C}_{\pi}(G)$.

\begin{theorem}
Assume that $G$ has compact center. Then the space of cusp forms on $G$ is the
topological direct sum
\[
\bigoplus_{\omega\in\mathcal{E}_{2}(G)}\mathcal{C}_{\omega}(G).
\]

\end{theorem}

This follows from Harish-Chandra's difficult converse to Theorem
\ref{Z-finite}.

\begin{theorem}
\label{cuspmainG}(c.f. \cite{RRGI} Theorem 7.7.6) Assume that the center of
$G$ is compact. If $f$ is a cusp form on $G$ that is $K$--finite then $\dim
Z(Lie(G))f<\infty$.
\end{theorem}

\section{The Schwartz space adapted to Whittaker models}

In this section we study a parallel theory to that of the previous section for
so called Whittaker functions. We retain the notation of the preceding
section. Let $\chi:N_{o}\rightarrow S^{1}$ be a unitary character. We say that
$\chi$ is generic if its differential is non-zero on every $A_{o}$ weight
space in $Lie(N_{o})/[Lie(N_{o}),Lie(N_{o})]$. We set $C^{\infty}%
(N_{o}\backslash G;\chi)$ equal to the space of smooth functions on $G$ such
that $f(ng)=\chi(n)f(g)$ for $n\in N_{o}$ and $g\in G$.

We define a unitary representation $L^{2}(N_{o}\backslash G;\chi)$ as
\ follows: We fix an invariant measure on $N_{o}$ and take the corresponding
right invariant measure on $N_{o}\backslash G$, $d\bar{g}$. The Hilbert space
is the space of all measurable (with respect to Haar measure) functions on $G$
such that

1. $f(ng)=\chi(n)f(g)$ for $n\in N_{o}$ and $g\in G$.

2.$\left\Vert f\right\Vert ^{2}=\int_{N_{o}\backslash G}|f(\bar{g})|^{2}%
d\bar{g}=\int_{A_{o}\times K}|f(ak)|^{2}a^{-2\rho}da<\infty$($\rho=\rho_{o}$,
up to normalization of $da$ on $A$).

If $f\in L^{2}(N_{o}\backslash G;\chi)$ then we define $\pi_{\chi
}(g)f(x)=f(xg)$ for $x,g\in G$. \ This defines a unitary representation of $G
$.

The right-most formula in 2. suggests the generalization of the Harish-Chandra
Schwartz space in this context. We define
\[
p_{d,x}(f)=\sup_{g\in G}(a(g)^{-\rho}(1+\left\Vert \log(a(g))\right\Vert
)^{d}|xf(g)|
\]
for $d\in\mathbb{R}$ and $x\in U(Lie(G))$. Then $\mathcal{C(}N_{o}\backslash
G;\chi)$ is the space of all $f\in C^{\infty}(N_{o}\backslash G;\chi)$ such
that $p_{d,x}(f)<\infty$ for all choices of $d$ and $x.$ We endow
$\mathcal{C(}N_{o}\backslash G;\chi)$ with the topology defined by the
semi-norms $p_{d,x}$. This defines the Fr\'{e}chet space given in \cite{RRGII}
Chapter 15. In light of Lemma \ref{basic-inequality} the following
semi--norms
\[
q_{d,x}(f)=\sup_{g\in G}(a(g)^{-\rho}(1+\log\left\Vert a(g)\right\Vert
)^{d}|xf(g)|
\]
are equivalent to the $p_{d,x}$ we will use one or the other depending on
convenience. \ We will also need another similar space. Here we use the
(Langlands) decomposition
\[
G=A_{G}{}^{o}G
\]
With $A_{G}$ a connected subgroup of \thinspace$A_{o}$ and $^{o}G$ has compact
center. If $d,r\geq0,x\in U(Lie(G))$ and $g=ag_{o}$ with $a\in A_{G} $ and
$g_{o}\in{}^{o}G$ then we define for $f\in C^{\infty}(N_{o}\backslash
G;\chi)$
\[
q_{d,k,x}(f)=\sup_{g\in{}^{o}G,a\in A_{G}}a(g)^{-\rho}(1+\log\left\Vert
a(g_{o})\right\Vert )^{k}(1+\log\left\Vert a\right\Vert )^{-d}\left\vert
xf(ag_{o})\right\vert .
\]
We set for each $d\geq0$, $\mathcal{B}_{d}(N_{o}\backslash G;\chi)$ equal to
the space of all $f\in C^{\infty}(N_{o}\backslash G;\chi)$ such that
$q_{d,k,x}(f)<\infty$ for all $k\geq0$ and $x\in U(Lie(G))$. \ Then
\[
\mathcal{B}(N_{o}\backslash G;\chi)=\lim_{\rightarrow}\mathcal{B}_{d}%
(N_{o}\backslash G;\chi)
\]
is a LF space. We note that if $G$ has compact center then $\mathcal{B}%
(N_{o}\backslash G;\chi)=\mathcal{C}(N_{o}\backslash G;\chi)$.

We have

\begin{lemma}
\label{Ltwo}$\mathcal{C(}N_{o}\backslash G;\chi)\subset L^{2}\mathcal{(}%
N_{o}\backslash G;\chi)$.
\end{lemma}

\begin{proof}
If $f\in\mathcal{C(}N_{o}\backslash G;\chi)$ then then for all $d>0$
\[
\left\vert f(ak\right\vert \leq C_{d}a^{\rho}(1+\left\Vert \log a\right\Vert
)^{-d}%
\]
Thus
\[
\int_{A_{o}\times K}|f(ak)|^{2}a^{-2\rho}dadk\leq C_{d}\int_{A_{o}\times
K}a^{2\rho}(1+\left\Vert \log a\right\Vert )^{-d}a^{-2\rho}dadk.
\]
Take $d$ so large that
\[
\int_{A_{o}}(1+\left\Vert \log a\right\Vert )^{-d}da<\infty.
\]

\end{proof}

\begin{proposition}
\label{Fourier}If $f\in\mathcal{C}(G)$ then the the integral
\[
f_{\chi}(g)=\int_{N_{o}}\chi(n)^{-1}f(ng)dn
\]
converges absolutely and uniformly on compacta in $g$ to an element of
$\mathcal{C(}N_{o}\backslash G;\chi)$. Furthermore, the map defined by
$T_{\chi}(f)=f_{\chi}$ is a continuous map from $\mathcal{C}(G)$ to
$\mathcal{C(}N_{o}\backslash G;\chi)$.
\end{proposition}

\begin{proof}
If $n\in N_{o}$, $a\in A_{o}$ and $k\in K$ then
\[
|f(nak)|\leq p_{d,1,1}(f)\Xi(na)(1+\log\left\Vert na\right\Vert )^{-d}.
\]
We note that there is an orthonormal basis of $V$ such that relative to that
basis $\mu(n)=I+X$ with $X$ upper triangular with zeroes on the main diagonal
and $\mu(a)$ is diagonal. Thus
\[
\mu(na)=\mu(a)+X\mu(a)
\]
and so if $a\in A_{o}\cap{}^{o}G$ then $\left\Vert an\right\Vert
\geq\left\Vert a\right\Vert $. If $a_{1}\in A_{G}$ and $a_{2}\in A_{o}\cap
{}^{o}G$ then $\left\Vert a_{1}a_{2}n\right\Vert =\left\Vert a_{1}\right\Vert
\left\Vert a_{2}n\right\Vert \geq\left\Vert a_{1}\right\Vert \left\Vert
a_{2}\right\Vert =\left\Vert a_{1}a_{2}\right\Vert $ so $\left\Vert
an\right\Vert \geq\left\Vert a\right\Vert $ for $a\in A_{o}$. Also $\left\Vert
n\right\Vert =\left\Vert a^{-1}an\right\Vert \leq\left\Vert a^{-1}\right\Vert
\left\Vert an\right\Vert =\left\Vert a\right\Vert \left\Vert an\right\Vert
1\leq\left\Vert an\right\Vert ^{2}.$ This implies that
\[
\left(  1+\log\left\Vert an\right\Vert \right)  \geq\left(  1+\log\left\Vert
a\right\Vert \right)
\]
and
\[
\left(  1+\log\left\Vert an\right\Vert \right)  \geq\frac{\left(
1+\log\left\Vert n\right\Vert \right)  }{2}.
\]
So if $d>0$
\[
\left(  1+\log\left\Vert an\right\Vert \right)  ^{-d}\leq2^{\frac{d}{2}%
}\left(  1+\log\left\Vert a\right\Vert \right)  ^{-\frac{d}{2}}\left(
1+\log\left\Vert n\right\Vert \right)  ^{-\frac{d}{2}}.
\]
This and Proposition \ref{basic} imply that for all $d>0$ and $n_{1}\in N_{o}
$, $a\in A_{o}$ and $k\in K$
\[
a^{-\rho}|f_{\chi}(n_{1}ak)|\leq a^{-\rho}\left\vert \int_{N_{o}}\chi
(n)^{-1}f(nn_{1}ak)dn\right\vert \leq a^{-\rho}\int_{N_{o}}|f(n_{1}nak)|dn
\]%
\[
=a^{-\rho}\int_{N_{o}}|f(nak)|dn\leq
\]%
\[
p_{d,1,1}(f)a^{-\rho}\int_{N_{o}}\Xi(na)(1+\log\left\Vert na\right\Vert
)^{-d}dn\leq
\]%
\[
cp_{d,1,1}(f)2^{\frac{d}{2}}(1+\log\left\Vert a\right\Vert )^{-\frac{d}{2}}.
\]
with $c$ independent of $d$ for $d$ sufficiently large. This implies that for
large $d$
\[
q_{\frac{d}{2},1}(f_{\chi})\leq2^{\frac{d}{2}}p_{d,1,1}(f).
\]
If $x\in U(Lie(G))$ then $\left(  xf\right)  _{\chi}=x(f_{\chi})$. Thus
\[
q_{\frac{d}{2},x}(f_{\chi})\leq2^{\frac{d}{2}}p_{d,1,x}(f).
\]

\end{proof}

We now come to a result that an is the analogue of Lemma 15.3.2 in
\cite{RRGII}. The latter lemma has an error in its proof (pointed out in
\cite{BanK} who also show that it cannot be true as stated).

Let $P$ be a standard parabolic subgroup (i.e. $P_{o}\subset P$) $P={}%
^{o}M_{P}A_{P}N_{P}$ (as in Section 2) and recall that $\overline{N}%
_{P}=\theta(N_{P})$. If $f\in\mathcal{C(}N_{o}\backslash G;\chi)$ and if
$m\in{}^{o}M_{P}$ and $a\in A_{P}$ then our candidate for the generalization
is
\[
f^{P}(ma)=a^{\rho}\int_{\overline{N}_{P}}f(\bar{n}ma)d\bar{n}.
\]
We will prove in Theorem \ref{15.3.Replace} that the integral converges
absolutely. We need the following simple lemma and to recall a key equality of
Harish-Chandra in the proof of the replacement to the defective lemma.

\begin{lemma}
\label{inequality}Let $V$ be a real inner product space. Let $V=V_{1}\oplus
V_{2}$ an orthogonal direct sum. Let $p_{1}$ \ denote the orthogonal
projection of $V$ onto $V_{1}$. Assume that we have $u\in V_{1},v\in V_{2}$
and $w\in V$ such that there exists $C>0$ such that
\[
1+\left\Vert p_{1}w\right\Vert \geq C(1+\left\Vert w\right\Vert ).
\]
Then
\[
1+\left\Vert u+v+w\right\Vert \geq\frac{C(1+\left\Vert w\right\Vert
)}{(1+\left\Vert u\right\Vert )}%
\]
and
\[
\left(  1+\left\Vert u+v+w\right\Vert \right)  ^{2}\geq\frac{C(1+\left\Vert
v\right\Vert )}{(1+\left\Vert u\right\Vert )}.
\]

\end{lemma}

\begin{proof}
Since $p_{1}v=0$,
\[
1+\left\Vert u+v+w\right\Vert \geq1+\left\Vert u+p_{1}(w)\right\Vert .
\]
Note that if $x,y\in V$ then
\[
1+\left\Vert x+y\right\Vert \geq\frac{1+\left\Vert x\right\Vert }{1+\left\Vert
y\right\Vert }.
\]
Indeed,
\[
1+\left\Vert x\right\Vert =1+\left\Vert x+y-y\right\Vert \leq
\]%
\[
1+\left\Vert x+y\right\Vert +\left\Vert y\right\Vert \leq\left(  1+\left\Vert
x+y\right\Vert \right)  (1+\left\Vert y\right\Vert ).
\]
Applying this to the content of the lemma we have
\[
1+\left\Vert u+v+w\right\Vert \geq1+\left\Vert u+p_{1}(w)\right\Vert \geq
\]%
\[
\frac{1+\left\Vert p_{1}w\right\Vert }{1+\left\Vert u\right\Vert }\geq
\frac{C(1+\left\Vert w\right\Vert )}{1+\left\Vert u\right\Vert }.
\]
This proves the first inequality. Also note that
\[
\left(  1+\left\Vert u+v+w\right\Vert \right)  (1+\left\Vert w\right\Vert
)\geq(1+\left\Vert u+v\right\Vert )\geq(1+\left\Vert v\right\Vert ).
\]
So the first inequality implies that
\[
\left(  1+\left\Vert u+v+w\right\Vert \right)  (1+\left\Vert w\right\Vert
)\leq C^{-1}\left(  1+\left\Vert u+v+w\right\Vert \right)  ^{2}(1+\left\Vert
u\right\Vert )
\]
proving the second inequality.
\end{proof}

The next lemma is due to Harish-Chandra (c.f. \cite{RRGI} Lemmas 3.A.2.3 and
4.A.2.3 (1) ) in it and in the rest of the section we will be using the
notation $a(g)$ for $a_{P_{o}}(g)$.

\begin{lemma}
\label{HCinequality}Let $P$ be a standard parabolic subgroup and let $\bar
{n}\in\bar{N}_{P}$ then
\[
\log a(\bar{n})\in\{\sum c_{\alpha}H_{\alpha}|c_{\alpha}\geq0\}
\]
(here if $\alpha\in\Phi(P_{o},A_{o})$ then $H_{\alpha}\in\mathfrak{a}_{o}$ is
defined by $B(H_{\alpha},h)=\alpha(h)$ for $h\in\mathfrak{a}_{o}$) and there
is a constant $C_{1}>0$ such that
\[
1-\rho_{P}(\log a(\bar{n}))\geq C_{1}(1+\left\Vert \log a(\bar{n})\right\Vert
)
\]
and%
\[
1-\rho_{P}(\log a(\bar{n}))\geq C_{1}(1+\log\left\Vert \bar{n}\right\Vert ).
\]

\end{lemma}

\begin{theorem}
\label{15.3.Replace}If $f\in\mathcal{B(}N_{o}\backslash G;\chi)$ then the
integral defining $f^{P}(ma)$ converges absolutely and uniformly for $m\in
{}^{o}M_{P}$ and $a\in A_{P}$ in compacta. Furthermore, the map $f\longmapsto
f^{P}$ is continuous from $\mathcal{B(}N_{o}\backslash G;\chi)$ to
$\mathcal{B(}N_{o}\cap M_{P}\backslash M_{P};\chi_{|N_{o}\cap M})$.
\end{theorem}

\begin{proof}
Since the estimates on $A_{G}$ have no effect on the the integrals, we may
assume that $G$ has compact center and we need to prove the results for
$f\in\mathcal{C(}N_{o}\backslash G;\chi)$. If $l\geq0,\bar{n}\in\bar{N}%
_{P},m\in{}^{o}M_{P},a\in A_{P}$ then
\[
\left\vert f(\bar{n}ma)\right\vert \leq a(\bar{n}ma)^{\rho}(1+\left\Vert \log
a(\bar{n}ma)\right\Vert )^{-l}q_{l,1}(f).
\]
The Iwasawa decomposition for $m$ relative to $K_{1}=K\cap M_{P},A_{1}%
=A_{o}\cap{}^{o}M_{P},N_{1}=N_{o}\cap M_{P}$ is $m=n_{1}a_{1}k_{1}$. Thus
since $A_{P}$ is in the center of $M_{P}$ we have
\[
\bar{n}ma=\bar{n}an_{1}a_{1}k_{1}%
\]
so if we set $b=an_{1}a_{1}$ then
\[
a(\bar{n}ma)=a(an_{1}a_{1}b^{-1}\bar{n}b)=aa_{1}a(b^{-1}\bar{n}b).
\]
We will now apply the previous two lemmas to get some inequalities. Set
${}^{\ast}\mathfrak{a}=Lie(A_{o}\cap{}^{o}M_{P})$. We note that
\[
\mathfrak{a}_{o}=\mathfrak{a}_{P}\oplus{}^{\ast}\mathfrak{a}%
\]
orthogonal direct sum let $p_{1}$ be the projection of $\mathfrak{a}_{o}$ onto
$\mathfrak{a}_{P}$ corresponding to this direct sum decomposition. Noting that
if $h\in\mathfrak{a}_{o}$ then $\rho_{P}(h)=\rho_{P}(p_{1}h)$. Thus
$\left\vert \rho_{P}(h)\right\vert \leq C_{2}\left\Vert p_{1}h\right\Vert $
and we can assume that $C_{2}>1$. Noting that if $\bar{n}\in\bar{N}_{P}$ that
$\rho_{P}(\log a(\bar{n}))\leq0$ for $\bar{n}\in\bar{N}_{P}$ then the second
inequality in Lemma \ref{HCinequality} implies that
\[
C_{2}(1+\left\Vert p_{1}\log a(\bar{n})\right\Vert )\geq1-\rho_{P}(\log
(\bar{n})\geq C_{1}(1+\left\Vert \log a(\bar{n})\right\Vert ).
\]
If we set $u=\log a,v=\log a_{1}$ and $w=\log a(b^{-1}\bar{n}b)$ the
hypotheses of Lemma \ref{inequality} are satisfied with $C=C_{2}^{-1}C_{1}$ if
$\bar{n}\in\bar{N}_{P}$. So we have
\[
\log a(\bar{n}ma)=u+v+w
\]
so Lemma \ref{inequality} implies that
\[
1+\left\Vert \log a(\bar{n}ma)\right\Vert \geq\frac{C(1+\left\Vert \log
a(b^{-1}\bar{n}b)\right\Vert )}{(1+\left\Vert \log a\right\Vert )}%
\]
and
\[
\left(  1+\left\Vert \log a(\bar{n}ma)\right\Vert \right)  ^{2}\geq\frac
{C_{1}(1+\left\Vert \log a_{1}\right\Vert )}{(1+\left\Vert \log a\right\Vert
)}.
\]
Hence, writing $p=2k+d$
\[
a(\bar{n}ma)^{\rho}(1+\log\left\Vert a(\bar{n}ma)\right\Vert )^{-l}=
\]%
\[
(aa_{1})^{\rho}a(b^{-1}\bar{n}b)^{\rho}(1+\left\Vert \log(aa_{1}a(b^{-1}%
\bar{n}b))\right\Vert )^{-l}\leq
\]%
\[
C_{1}^{-\left(  k+d\right)  }(aa_{1})^{\rho}a(b^{-1}\bar{n}b)^{\rho
}(1+\left\Vert \log a\right\Vert )^{d+k}\times
\]%
\[
(1+\left\Vert \log a_{1}\right\Vert )^{-k}(1+\left\Vert \log a(b^{-1}\bar
{n}b)\right\Vert )^{-d}=
\]%
\[
C_{1}^{-\left(  k+d\right)  }(aa(m))^{\rho}a(b^{-1}\bar{n}b)^{\rho
}(1+\left\Vert \log a\right\Vert )^{d+k}\times
\]%
\[
(1+\left\Vert \log a(m)\right\Vert )^{-k}(1+\left\Vert \log a(b^{-1}\bar
{n}b)\right\Vert )^{-d}%
\]
Obviously, if $\Omega$ is a compact subset of $M_{P}$ then there is a constant
$C_{\Omega,d}$ such that
\[
0\leq C_{1}^{-\left(  k+d\right)  }a^{\rho}(1+\left\Vert \log a\right\Vert
)^{d}a(m)^{\rho}(1+\left\Vert \log a(m)\right\Vert )^{-d}\leq C_{\Omega,d}%
\]
if $ma\in\Omega$ since $m\in{}^{o}M_{P}$. \ Thus
\[
\int_{\overline{N}_{P}}\left\vert f(\bar{n}ma)\right\vert d\bar{n}\leq
\]%
\[
C_{\Omega,d}q_{d,1}(f)\int_{\overline{N}_{P}}a(b^{-1}\bar{n}b)^{\rho}\left(
1+\left\Vert \log a(b^{-1}\bar{n}b)\right\Vert \right)  ^{-d}d\bar{n}=
\]%
\[
C_{\Omega,d}q_{d,1}(f)a^{-2\rho}\int_{\overline{N}_{P}}a(\bar{n})^{\rho
}\left(  1+\left\Vert \log a(\bar{n})\right\Vert \right)  ^{-d}d\bar{n}.
\]
If $d$ is sufficiently large this integral converges (c.f. \cite{RRGI} Theorem
4.5.4) to $C_{2}$. Now if we put together everything we have proved so far
Lemma 1 implies that
\[
\left\vert f^{P}(ma)\right\vert \leq C_{3}^{-d-k}q_{p,1}(f)a(m)^{\rho}%
(1+\log\left\Vert a(m)\right\Vert )^{-k}(1+\log\left\Vert a\right\Vert
)^{d+k}.
\]
So
\[
q_{d,k,1}\left(  f^{P}\right)  \leq C_{4}^{-d-k}q_{2k+d,1}(f).
\]
Replacing $f$ by $xf$ for $x\in U(LieM_{P})$ yields $q_{d,k,x}\left(
f^{P}\right)  \leq C_{1}^{-d-k}q_{d+2k,x}(f)$. Thus $f\longmapsto f^{P}$ is
continuous from $\mathcal{C(}N_{o}\backslash G;\chi)$ to $\mathcal{B(}%
N_{o}\cap M_{P}\backslash M_{P};\chi_{|N_{o}\cap M})$.
\end{proof}

The next results can be proved using the arguments in the proofs of the two
preceding theorems.

\begin{corollary}
\label{basic-estimate}Let $P$ be a standard parabolic subgroup of $G$ and set
$\overline{N}=\overline{N}_{P}=\theta N_{P}$. If $f\in\mathcal{C}(G)$ then the
integral
\[
\int_{N_{o}\times\bar{N}}\left\vert f(n_{o}\bar{n})\right\vert dn_{o}d\bar{n}%
\]
is convergent and defines a continuous seminorm on $\mathcal{C}(G)$.
\end{corollary}

\begin{corollary}
\label{new-try}Let $P$ be a standard parabolic subgroup of $G$ and set
$\overline{N}_{P}=\theta N_{P}$. If $f\in\mathcal{C}(G),m\in{}^{o}M_{P}$ and
$a\in A_{P}$ then then the integral
\[
\phi(ma)=\int_{N_{P}\times\bar{N}_{P}}f(n\bar{n}ma)dnd\bar{n}%
\]
converges absolutely and uniformly in compacta of $M$ and satisfies the
inequalities
\[
|L_{x}R_{y}\phi(ma)|\leq q_{x,y,d,k}(f)\Xi_{^{o}M_{P}}(m)(1+\log\left\Vert
m\right\Vert )^{-k}(1+\left\Vert \log a\right\Vert )^{d}%
\]
with $d,k>0,x,y\in U(\mathfrak{g})$ and $q_{x,y,d,k}$ is a continuous
semi-norm on $\mathcal{C}(G)$.
\end{corollary}

\begin{corollary}
\label{goodenough}Assume that $G$ has compact center. Let $P$ be a standard
parabolic subgroup of $G$. Let $H\in Lie(A_{P})$ be such that $\rho(H)\leq0 $
and let $m\in{}^{o}M_{P}$. If $f\in\mathcal{C(}N_{o}\backslash G;\chi) $ then
for each $d>0$ there exists $C_{d}$ such that
\[
\left\vert f^{P}(\exp(H)m)\right\vert \leq C_{d}(1-\rho(H))^{-d}a(m)^{\rho}.
\]

\end{corollary}

\begin{proof}
We begin the argument as we did in the proof Theorem \ref{15.3.Replace} with
the same notation. Let $a=\exp(H)$ and $b=an_{1}a_{1}$. Note that
\[
a(\bar{n}ma)^{\rho}(1+\left\Vert \log a(\bar{n}ma)\right\Vert )^{-d}=
\]%
\[
(aa_{1})^{\rho}a(b^{-1}\bar{n}b)^{\rho}(1+\left\Vert \log\left(
aa_{1}a(b^{-1}\bar{n}b)\right)  \right\Vert )^{-d}.
\]
Also,
\[
\rho_{P}(\log\left(  aa_{1}a(b^{-1}\bar{n}b)\right)  =\rho_{P}(\log
(a))+\rho_{P}(\log a(b^{-1}\bar{n}b))
\]
and there exists $0<C<1$ such that
\[
1+\left\Vert \log a(\bar{n}ma)\right\Vert \geq C(1-\rho_{P}(\log\left(
aa_{1}a(b^{-1}\bar{n}b)\right)  ).\text{ }%
\]
Thus
\[
1+\left\Vert \log a(\bar{n}ma)\right\Vert \geq C\left(  1-\rho_{P}%
(\log(a))-\rho_{P}(\log a(b^{-1}\bar{n}b))\right)
\]
recalling from Lemma \ref{HCinequality} $\rho_{P}(\log a(b^{-1}\bar{n}%
b))\leq0$ we have
\[
1+\left\Vert \log a(\bar{n}ma)\right\Vert \geq C(1-\rho_{P}(\log a))
\]
and
\[
1+\left\Vert \log a(\bar{n}ma)\right\Vert \geq C(1-\rho_{P}(\log a(b^{-1}%
\bar{n}b)).
\]
This implies that
\[
a(\bar{n}ma)^{\rho}(1+\left\Vert \log a(\bar{n}ma)\right\Vert )^{-k-d}\leq
\]%
\[
C_{1}^{-k-d}(aa_{1})^{\rho}a(b^{-1}\bar{n}b)^{\rho}(1-\rho_{P}(\log
a))^{-k}(1-\rho_{P}(\log a(b^{-1}\bar{n}b))^{-d}.
\]
Harish-Chandra has shown that if $d>0$ is sufficiently large then
\[
\int_{\bar{N}}a(\bar{n})^{\rho}(1-\rho_{P}(\log a(\bar{n}))^{-d}<\infty
\]
(c.f. Theorem 4.5.4 \cite{RRGI})We have
\[
\left\vert f^{P}(ma)\right\vert \leq a^{\rho}\int_{\bar{N}}\left\vert
f(\bar{n}ma)\right\vert d\bar{n}\leq
\]%
\[
C_{1}^{-k-d}(a_{1})^{\rho}a^{2\rho}(1-\rho_{P}(\log a))^{-k}\int_{\bar{N}%
}a(b^{-1}\bar{n}b)^{\rho}(1-\rho_{P}(\log a(b^{-1}\bar{n}b))^{-d}d\bar{n}.
\]
If $d$ is sufficiently large then the integral converges in light of the lemma
above and the argument in Theorem \ref{15.3.Replace}. The estimate follows
from the formula.
\end{proof}

Let $\alpha_{1},...,\alpha_{l}$ be the simple roots of the parabolic $P_{o}$
relative to $A_{o}$.

\begin{lemma}
\label{N-estimate}We assume that $G$ has compact center and that $\chi$ is
generic. If $f\in\mathcal{C}(N_{o}\backslash G;\chi)$, $h\in\mathfrak{a}_{o}$,
$m_{i}\in\mathbb{Z}_{\geq0}$ is given for $1\leq i\leq l$ and $d>0$ then there
is a continuous seminorm $q=q_{\{m_{i}\},d}$ on $\mathcal{C}(N_{o}\backslash
G;\chi)$ such that
\[
|f(\exp(h)k)|\leq e^{-\sum_{ii=1}^{l}m_{i}\alpha_{i}(h)}(1+\left\Vert
h\right\Vert )^{-d}e^{\rho_{o}(h)}q(f).
\]

\end{lemma}

\begin{proof}
Let $F=\{i|m_{i}>0\}$ and let $x_{1},...,x_{n}$ be a basis of $Lie(G)$. If
$X\in Lie(G)$ and if $k\in K$ then we can write $Ad(k)X=\sum a_{i}(k,X)x_{i}$.
Note that
\[
\left\vert a_{i}(k,X)\right\vert \leq C\left\Vert X\right\Vert
\]
for all $k\in K$. Now let $X_{i}$ be an element of the $\alpha_{i}$ root space
in $\mathfrak{n}_{o}$ such that $d\chi(X_{i})=z_{i}\neq0$. Then
\[
f(\exp(h)k)=z_{i}^{-1}L_{X_{i}}f(\exp(h)k)=z_{i}^{-1}\frac{d}{dt}_{|t=0}%
f(\exp(tX_{i})\exp(h)k)=
\]%
\[
z_{i}^{-1}\frac{d}{dt}_{|t=0}f(\exp(h)\exp(tAd(\exp(-h))X_{i})k)=
\]%
\[
z_{i}^{-1}\frac{d}{dt}_{|t=0}f(\exp(h)\exp(te^{-\alpha_{i}(h)}X_{i})k)=
\]%
\[
z_{i}^{-1}\frac{d}{dt}_{|t=0}f(\exp(h)\exp(te^{-\alpha_{i}(h)}Ad(k^{-1}%
)X_{i})k)=
\]%
\[
e^{-\alpha_{i}(h)}z_{i}^{-1}\sum a_{j}(k^{-1},X_{i})x_{j}f(\exp(h)k).
\]
Iterating this argument yields and expression
\[
f(\exp(h)k)=e^{-\sum_{i\in F}m_{i}\alpha_{i}(h)}Z(k)f(ak)
\]
with $Z$ a smooth function from $K$ to $L=U^{\sum_{i\in F}m_{i}}(Lie(G))$ with
$U^{j}(Lie(G))$ the standard filtration. If we choose a basis of $L$,
$y_{1},...,y_{r}$ then we have
\[
Z(k)=\sum b_{i}(k)y_{i}%
\]
with $b_{i}$ continuous functions on $K.$ Let $C_{j}=\max_{k\in K}\left\vert
b_{j}(k)\right\vert $. We have
\[
\left\vert f(\exp(h)k)\right\vert \leq e^{-\sum_{i\in F}m_{i}\alpha_{i}%
(h)}\sum_{j}C_{j}\left\vert y_{j}f(\exp(h)k)\right\vert \leq
\]%
\[
e^{-\sum_{i\in F}m_{i}\alpha_{i}(h)}(1+\left\Vert h\right\Vert )^{-d}%
e^{\rho_{o}(h)}\sum_{j}C_{j}q_{d,y_{j}}(f).
\]

\end{proof}

\begin{corollary}
\label{xi-estimate}Assume that $\chi$ is generic then if $f\in\mathcal{C}%
(N_{o}\backslash G;\chi)$ and $x\in U(\mathfrak{g})$ then there exist $d$ and
$m $ such that
\[
\left\vert xf(g)\right\vert \leq q_{x,d,m}(f)\Xi(a(g))a(g)^{-\sum_{i\in
F}r_{i}\alpha_{i}}(1+\log\left\Vert a(g)\right\Vert )^{-d}.
\]
with $r_{i}\in\mathbb{Z}_{\geq0}.$
\end{corollary}

\begin{proof}
If $a\in A_{o}$ then $a=\exp(h)$ with $h\in-s\overline{\mathfrak{a}_{o}^{+}}$
for some $s\in W(A_{o})$. Harish-Chandra's estimates (2. in the previous
section) imply that
\[
a^{s\rho_{o}}\leq\Xi(a)\leq Ca^{s\rho_{o}}(1+\log\left\Vert a\right\Vert
)^{d}.
\]
Now the desired inequality follows from $s\rho_{o}(h)=\rho_{o}(h)-\sum
_{i}m_{i}\alpha_{i}(h)$ with $m_{i}\geq0$. The previous lemma completes the proof.
\end{proof}

We have seen in assertion 7. section 3 that if $G$ has compact center then the
standard norm, $\left\Vert ...\right\Vert $, on $G$ such that there exist
$C_{1},C_{2},r>0$ such that
\[
C_{1}\left\Vert g\right\Vert ^{-1}\leq\Xi(g)\leq C_{2}\left\Vert g\right\Vert
^{-1}(1+\log\left\Vert g\right\Vert )^{r}.
\]

\begin{theorem}
Assume that $\chi$ is generic. Let $f\in\mathcal{C}(N_{o}\backslash G;\chi)$
be right $K$--finite and let $\varphi\in C_{c}^{\infty}(N_{o})$ then
$\psi(nak)=\varphi(n)f(ak)$ for $n\in N_{o},a\in A_{o}$ and $k\in K$ defines
an element of $\mathcal{C}(G)$.
\end{theorem}

To prove this result we will use the following lemma.

\begin{lemma}
Let $\psi\in C^{\infty}(G)$ be such that
\[
\psi(nak)=\sum_{i=1}^{a}\sum_{j=1}^{b}\varphi_{i}(n)f_{ij}(a)\gamma
_{j}(k),n\in N_{o},a\in A_{o},k\in K.
\]
Assume that $\varphi_{i}\in C_{c}^{\infty}(N_{o}),\gamma_{i}\in C^{\infty}(K)$
and $f_{ij}\in C^{\infty}(A_{o})$ satisfies
\[
\left\vert uf_{ij}(a)\right\vert \leq C_{u,r,d}a^{-\sum r_{i}\alpha_{i}%
}\left\Vert a\right\Vert ^{-1}(1+\log\left\Vert a\right\Vert )^{-d}%
\]
for all $u\in U(\mathfrak{a}_{o})$ and $r_{i}\in\mathbb{Z}_{\geq0}$. Then
\[
\left\vert \psi(g)\right\vert \leq C_{d}\left\Vert g\right\Vert ^{-1}%
(1+\log\left\Vert g\right\Vert )^{-d}%
\]
and if $X\in Lie(G)$ then both $R_{X}\psi$ and $L_{X}\psi$ have decompositions
as above relative to the Iwasawa decomposition.
\end{lemma}

\begin{proof}
Let $\Omega$ be the union of the supports of the $\varphi_{i}$. There exist
constants $c_{3}$ and $c_{4}$ such that if $n\in\Omega$ then
\[
\left\Vert n\right\Vert \leq c_{3},\left\vert \varphi_{i}(n)\right\vert \leq
c_{4},1\leq i\leq d.
\]
Thus if $a\in A_{o}$ then
\[
\left\Vert na\right\Vert \leq\left\Vert n\right\Vert \left\Vert a\right\Vert
\leq c_{3}\left\Vert a\right\Vert
\]
and
\[
\left\Vert a\right\Vert =\left\Vert n^{-1}na\right\Vert \leq\left\Vert
n^{-1}\right\Vert \left\Vert na\right\Vert \leq c_{3}\left\Vert na\right\Vert
.
\]
Thus, since $\left\Vert gk\right\Vert =\left\Vert g\right\Vert $ for $k\in K$,
for each $d$ there exists a constant, $B_{d\text{,}}$ such that
\[
\left\vert \psi(nak)\right\vert \leq c_{4}B_{d}\left\Vert a\right\Vert
^{-1}(1+\log\left\Vert a\right\Vert )^{-d}\leq
\]%
\[
c_{4}c_{3}B_{d}\left\Vert nak\right\Vert ^{-1}(1+\log\left\Vert nak\right\Vert
)^{-d}.
\]
We now prove that each of $L_{X}\psi(g)$ and $R_{X}\psi(g)$ satisfy the
hypotheses of the lemma. We first consider the right derivatives. If $X\in
Lie(G)$ then
\[
R_{X}\psi(nak)=\frac{d}{dt}_{t=0}\psi(na\exp(tAd(k)X)k)
\]
we choose a basis of $Lie(G)$, $X_{1},...,X_{n}$ with $X_{i}$ in the
$\beta_{i}$ root space for the action of $\mathfrak{a}_{o}$ on $Lie(N_{o})$ if
$i\leq r$, $X_{i}\in\mathfrak{a}_{o}$ for $r<i\leq r+l$ and $X_{i}\in Lie(K)$
for $i>r+l$. Then
\[
Ad(k)X=\sum_{i}c_{i}(k,X)X_{i}\text{.}%
\]
Hence
\[
R_{X}\psi(nak)=\sum_{i}c_{i}(k,X)\frac{d}{dt}_{t=0}\psi(na\exp(tX_{i})ak).
\]
It is enough to prove that each of the functions
\[
\frac{d}{dt}_{t=0}\psi(na\exp(tX_{p})k)
\]
satisfies the hypotheses. If $p\leq r$ then
\[
\frac{d}{dt}_{t=0}\psi(na\exp(tX_{p})k)=\frac{d}{dt}_{t=0}\psi(n\exp
(tAd(a)X_{p})k)
\]%
\[
=a^{\beta_{p}}\sum\frac{d}{dt}_{t=0}\varphi_{i}(n\exp(X_{p}))f_{ij}%
(a)\gamma_{j}(k).
\]
Which is of the right form since $\beta_{i}=\sum s_{ij}\alpha_{j}$ with
$s_{ij}\in\mathbb{Z}_{\geq0}$. If $r<p\leq r+l$ then
\[
\frac{d}{dt}_{t=0}\psi(na\exp(tX_{p})k)=\sum_{i,j}\gamma_{j}(k)\varphi
_{i}(n)\frac{d}{dt}_{t=0}f_{ij}(a\exp(tX_{p})).
\]
which is of the right form. Note that the decomposition is obvious if $p>r+l$.

We now take a basis $Y_{i}=X_{i}$, $i\leq r+l$ and $Y_{i}\ $in the $-\beta
_{i}$ root space with $\beta_{i}$ as above $i>r+l.$ We now look at the left
derivative. Noting that
\[
Ad(n)^{-1}X=\sum_{i}d_{i}(n,X)Y_{i}%
\]%
\[
L_{X}\psi(nak)=\sum d_{p}(n,X)\frac{d}{dt}_{t=0}\psi(n\exp(tY_{p})ak).
\]
If \ $p\leq r+l$ then it is obvious that
\[
\frac{d}{dt}_{t=0}\psi(n\exp(tY_{p})ak)
\]
has the desired decomposition. Suppose that $p>r+l$ then
\[
\frac{d}{dt}_{t=0}\psi(n\exp(tY_{i})ak)=a^{\beta_{i}}\frac{d}{dt}_{t=0}%
\psi(na\exp(tY_{i})k).
\]
Also $Y_{i}=-\theta Y_{i}+(Y_{i}+\theta Y_{i})$ and $-\theta Y_{i}$is in the
$\beta_{i}$ root space and $Z_{i}=Y_{i}+\theta Y_{i}$ $\in Lie(K)$. So
\[
\frac{d}{dt}_{t=0}\psi(n\exp(tY_{i})ak)=a^{2\beta_{i}}\frac{d}{dt}_{t=0}%
\psi(n\exp(-t\theta Y_{i})ak)+a^{\beta_{i}}\frac{d}{dt}_{t=0}\psi
(na\exp(tZ_{i})k).
\]
We leave the rest to the reader.
\end{proof}

\begin{proof}
(of the theorem)
\[
\psi(nak)=\varphi(n)f(ak)
\]
since $f$ is $K$--finite $f(gk)=\sum f_{i}(g)\gamma_{i}(k)$ with $\gamma
_{i}\in C^{\infty}(K)$ and $f_{i}\in\mathcal{C}(N_{o}\backslash G;\chi) $.
Corollary \ref{xi-estimate} \ implies that $\psi$ satisfies the hypotheses of
the previous lemma. An iteration of the lemma implies the theorem.
\end{proof}

\begin{theorem}
\label{surjective}Assume that $\chi$ is generic. Let $\mathcal{C}(G)_{K}$
(resp. $\mathcal{C}(N_{o}\backslash G;\chi)_{K}$) denote the right $K$ finite
elements of $\mathcal{C}(G)$ (resp. $\mathcal{C}(N_{o}\backslash G;\chi)$).
Let $T_{\chi}$be as in Theorem \ref{Fourier} then the map
\[
T_{\chi}:\mathcal{C}(G)_{K}\rightarrow\mathcal{C}(N_{o}\backslash G;\chi)_{K}%
\]
is surjective.
\end{theorem}

\begin{proof}
Let $f\in\mathcal{C}(N_{o}\backslash G;\chi)_{K}$ and let $\varphi\in
C_{c}^{\infty}(N_{o})$ be such that
\[
\int_{N_{o}}\chi(n)^{-1}\varphi(n)dn=1.
\]
Then define $\psi(nak)=\varphi(n)f(ak)$ or $n\in N_{o},a\in A_{o}$ and $k\in
K$. We have seen that $\psi\in\mathcal{C}(G)_{K}$. Also
\[
\psi_{\chi}(nak)=\chi(n)f(ak)=f(nak).
\]

\end{proof}

We will use the following result later in this paper.

\begin{lemma}
\label{compact-estimate}Let $\omega$ be a compact subset of $A_{o}$ and let
$P={}^{o}M_{P}A_{P}N_{P}$ be a standard parabolic subgroup with standard
Langlands decomposition. If $\mu\in\mathfrak{a}_{P}^{\ast}$ satisfies
$(\mu,\alpha)\geq0$ for all $\alpha\in\Phi(P,A_{P})$ then there exists a
constant $C_{\mu,\omega}>0$ such that if $\bar{n}am\in N_{o}\omega K$ with
$\bar{n}\in\bar{N}_{P},a\in A_{P},m\in{}^{o}M_{P}$ and $k\in K$ then $a^{\mu
}\geq C_{\mu,\omega}$.
\end{lemma}

\begin{proof}
The set of $\mu$ in the statement is the convex hull of the elements
$\lambda\in\mathfrak{a}^{\ast}$ obtained as follows:Let $F$ be a finite
dimensional representation of $G$ such that $^{o}M_{P}$ acts trivially on
$F^{N_{P}}$ and $a\in A_{P}$ acts by $a^{\lambda}$ on $F^{N_{P}}$. It is
therefore sufficient to prove the result for such $\lambda$ with $F$ the
corresponding representation. Put a $K$--invariant inner product on $F$ such
that the elements of $\mathfrak{a}_{o}$ act as Hermitian operators. Let
$\left\Vert ...\right\Vert $ be the corresponding norm on $F$. Let $v_{o}$ be
a unit vector in $F^{N_{o}}=F^{N_{P}}$. There exist constants $C_{1},C_{2}$
such that if $u\in\omega$ then
\[
0<C_{1}\leq\left\Vert uv_{o}\right\Vert \leq C_{2}<\infty.
\]
Thus, if $\bar{n}\in\bar{N},a\in A_{P},m\in{}^{o}M_{P}$, $n\in N,u\in
\omega,k\in K$ and if $\bar{n}am=nuk$ then
\[
\left\Vert \left(  \bar{n}am\right)  ^{-1}v_{o}\right\Vert =\left\Vert
u^{-1}v_{o}\right\Vert \leq C_{2}.
\]
Also,
\[
C_{2}\geq\left\Vert \left(  \bar{n}am\right)  ^{-1}v_{o}\right\Vert
=\left\Vert a^{-1}(am)\bar{n}(am)^{-1})v_{o}\right\Vert \geq\left\Vert
a^{-1}v_{o}\right\Vert =a^{-\lambda}.
\]
Here, we have used the fact that if $\bar{n}\in\bar{N}$ then $\bar{n}%
v_{o}=v_{o}+z$ with $\left\langle v_{o},Az\right\rangle =0$.
\end{proof}

\section{The space of Whittaker cusp forms}

We retain the notation of the previous sections.

If $f\in\mathcal{C(}N_{o}\backslash G;\chi)$ then we say that $f$ is a cusp
form if $\left(  R_{k}f\right)  ^{P}=0$ for all proper standard parabolic
subgroups $P$ and all $k\in K$. We leave the following lemma as an exercise

\begin{lemma}
If $f$ is a cusp form then $\left(  R_{g}f\right)  ^{P}=0$ for all proper
standard parabolic subgroups $P$ and all $g\in G$.
\end{lemma}

We set $^{o}\mathcal{C(}N_{o}\backslash G;\chi)$ equal to the space of cusp
forms in $\mathcal{C(}N_{o}\backslash G;\chi)$. Here is the analogue of
Theorem \ref{Z-finite} \ in this context. Due to the error in \cite{RRGII}
this result was also left without a proof.

\begin{theorem}
If $f\in\mathcal{C(}N_{o}\backslash G;\chi)$ and $\dim Z(Lie(G))f<\infty$ then
$f$ is a cusp form.
\end{theorem}

\begin{proof}
It is enough to prove that $f_{\gamma}^{P}=0$ for all standard parabolic
subgroups and all $\gamma\in\hat{K}$. So we will assume $f\ $\ is $K$-finite.
We may assume that the center of $G$ is compact. If $P$ is a parabolic
subgroup of $G$ we have $\bar{P}={}^{o}M_{P}A_{P}\bar{N}_{P}$. We note that
\[
V=U(Lie(G))Span_{\mathbb{C}}R_{K}f
\]
is a finitely generated, admissible $(Lie(G),K)$--module. We also note that if
$\phi\in Lie(\bar{N}_{P})V$ then $\phi^{P}=0$. Also $V/Lie(\bar{N}_{P})V $ is
an admissible, finitely generated $(Lie(^{o}M_{P}),{}^{o}M_{P}\cap K)$--module
(\cite{RRGI} Corollary 3.7.2) so $\dim Span_{\mathbb{C}}(A_{P}u)<\infty$ if
$u\in V/Lie(\bar{N}_{P})V$. This implies that $f^{P}(\exp(tH)m)$ is an
exponential polynomial in $t$ for each $m\in{}^{o}M_{P}$ and $H\in Lie(A_{P})$
i.e.
\[
f^{P}(\exp(tH)m)=\sum_{j}e^{\mu_{j}t}p_{j}(t,m)
\]
with $p_{j}(t,m)$ a polynomial in $t$ and $\mu_{j}\in\mathbb{C}$ distinct.
But
\[
f^{P}\in\mathcal{B(}N_{o}\cap M_{p}\backslash M_{P};\chi_{|N_{o}\cap M_{p}})
\]
which implies that there exists $r$ and $C$ such that for all $t$
\[
\left\vert f^{P}(\exp(tH)m)\right\vert \leq C(1+\left\vert t\right\vert
)^{r}.
\]
Hence $\mu_{j}=i\nu_{j}$ with $\nu_{j}\in\mathbb{R}$. \ Now Corollary
\ref{goodenough} implies that if $\rho(H)>0$ and $t>0$ and $d>0$ than there
exists $C_{m,d}$ such that
\[
\left\vert \sum_{j}e^{-i\nu_{j}t}p_{j}(-t,m)\right\vert \leq C_{m,d}%
(1+t)^{-d}.
\]
Thus $\sum_{j}e^{-i\nu_{j}t}p_{j}(t,m)=0$ for all $t$ which implies the theorem.
\end{proof}

\begin{theorem}
\label{Discrete+Cuspform}Assume that $\chi$ is generic and $G$ has compact
center. If $H\subset L^{2}(N_{o}\backslash G;\chi)$ is a closed, invariant,
irreducible subspace then $H^{\infty}\subset{}^{o}\mathcal{C}(N_{o}\backslash
G;\chi)$.
\end{theorem}

\begin{proof}
In \cite{RRGII} Theorem 15.3.5 a proof is given that
\[
H_{K}^{\infty}\subset{}^{o}\mathcal{C}(N_{o}\backslash G;\chi).
\]
The completion of $H_{K}^{\infty}$ in $L^{2}(N_{o}\backslash G;\chi)^{\infty}$
is a smooth Fr\'{e}chet representation of moderate growth as is the completion
in $\mathcal{C}(N_{o}\backslash G;\chi)$. The Casselman-Wallach theorem
\cite{RRGII},11.6.7 implies that the two completions are the same. Now the
previous theorem implies this theorem.
\end{proof}

In order to carry out the rest of the theory of Whittaker cusp forms we need
to recall Harish-Chandra's decomposition of the Schwartz space of $G$ and our
results on the analytic continuation of Jacquet integrals.

\section{The Harish-Chandra Plancherel Theorem.}

Throughout this section we will assume that $G$ has compact center. The
purpose of this section is to describe Harish-Chandra's decomposition of
$\mathcal{C}(G)$ relative to conjugacy classes of associate standard parabolic
subgroups of $G$. If $P_{1}$ and $P_{2}$ are parabolic subgroups of $G$ then
they are associate if there exists $g\in G$ with $gM_{P_{1}}g^{-1}=M_{P_{2}}$.
A parabolic subgroup, $P$, of $G$ is said to be cuspidal if $^{o}M_{P}$
contains a compact Cartan subgroup. Let $\mathcal{P}=\mathcal{P}(G)$ denote
the set of conjugacy classes of associate cuspidal parabolic subgroups. Then
one can show that up to conjugacy all Cartan subgroups of $G$ can be obtained
from elements of $\mathcal{P}$ as $[P]\longmapsto T_{P}A_{P}$ where $P$ is a
representative of $[P]$ and $T_{P}$ is a compact Cartan subgroup of $M_{P}$.
If we further, divide by $M_{P}$ conjugacy then the correspondence is
bijective. Thus we can either talk about the set of all conjugacy classes of
Cartan subgroups of $G$ or the set of associativity classes of cuspidal,
standard parabolic subgroups of $G$.

The distributional form Harish-Chandra Plancherel theorem is (we will explain
the notation after the statement).

\begin{theorem}
\label{plancherel}(c.f. \cite{RRGII} Theorem 13.4.1) Let $\delta$ be the Dirac
distribution at the identity element $e$ in $G$. If $f\in\mathcal{C}(G) $
then
\[
\delta(f)=\sum_{[P]\in\mathcal{P}}\sum_{[\sigma]\in\mathcal{E}_{2}({}^{o}%
M_{P})}d(\sigma)\int_{\mathfrak{a}_{P}^{\ast}}\Theta_{P,\sigma.i\nu}%
(f)\mu(\sigma,\nu)d\nu
\]
with $d(\sigma)$ the formal degree of $\sigma$.
\end{theorem}

Now for the explanations. The set $\mathcal{E}_{2}(^{o}M_{P})$ is the set of
equivalence classes of irreducible square integrable representations of
$^{o}M_{P},[\sigma]$ is the equivalence class of $\sigma$. $\Theta
_{P,\sigma.i\nu}$ is the character of the induced representation
$I_{P,\sigma,i\nu}$ initially defined for $f\in C_{c}^{\infty}(G)$ as
$\mathrm{tr}\left(  \pi_{P,\sigma,i\nu}(f)\right)  $ and shown by
Harish-Chandra to extend to a continuous functional on $\mathcal{C}(G)$.
$\mu(\sigma,\nu)$ is for each $\sigma$ a non-negative analytic function on
$\mathfrak{a}_{P}^{\ast}$ that is of polynomial growth in $\nu$. We should
note that the function $\mu$ depends on $[P]$ and the normalizations of the
Haar measures of all of the groups involved. Also, the parameter $[P]$ is the
conjugacy class of $A_{P}$ in Harish-Chandra \cite{Plancherel3} and
\cite{RRGII}.

This monumental achievement is one of the most important theorems of the
twentieth century. Harish-Chandra's steps leading to its proof led to the
deepest results on intertwining operators, square integrable representations,
regular singular differential equations, classification problems,... We will
need the following implication.

\begin{theorem}
\label{HCimplication} (c.f. \cite{RRGII}, Theorem 13.4.6 (1)) Let $P\in\lbrack
P]\in\mathcal{P}$. Let for $\sigma\in\mathcal{E}_{2}(^{o}M_{P})$,
$v,w\in\left(  I_{\sigma}\right)  _{K}$ and $\alpha$ in the (usual) Schwartz
space of $\mathfrak{a}_{P}^{\ast}$, $\mathcal{S(}\mathfrak{a}_{P}^{\ast}),$
\[
\varphi_{P,\sigma,\alpha,v,w}(g)=\int_{\mathfrak{a}_{P}^{\ast}}\left\langle
\pi_{P,\sigma,\iota\nu}(g)v,w\right\rangle \alpha(\nu)\mu(\sigma,\nu)d\nu.
\]
Then $\varphi_{P,\sigma,v,w}\in\mathcal{C}(G)$. Let $\mathcal{C}%
_{[P],[\sigma]}(G)$ denote the closure in $\mathcal{C}(G)$ of the span of
\[
\left\{  \varphi_{P,\sigma,\alpha,w}|v,w\in\left(  I_{\sigma}\right)
_{K},\alpha\in\mathcal{S(}\mathfrak{a}_{P}^{\ast})\right\}  .
\]
Then $\mathcal{C}(G)$ is the orthogonal (relative to the $L^{2}$--inner
product) direct sum
\[%
%TCIMACRO{\dbigoplus \nolimits_{\lbrack P]\in\mathcal{P},\omega\in
%\mathcal{E}_{2}(^{o}M_{P})}}%
%BeginExpansion
{\displaystyle\bigoplus\nolimits_{\lbrack P]\in\mathcal{P},\omega
\in\mathcal{E}_{2}(^{o}M_{P})}}
%EndExpansion
\mathcal{C}_{[P],[\omega]}(G).
\]

\end{theorem}

\section{A decomposition of $\mathcal{C}(N_{o}\backslash G;\chi)$}

\begin{lemma}
Let $\varphi,\psi\in\mathcal{C}(G)$ then
\[
\int_{N_{o}\backslash G}\int_{N_{o}}\int_{N_{o}}\left\vert \varphi(n_{1}%
g)\psi(n_{2}g)\right\vert dgdn_{1}dn_{2}<\infty.
\]

\end{lemma}

\begin{proof}
The proof of Proposition \ref{Fourier} used only the estimates on the absolute
value of an element of $\mathcal{C}(G)$ and is true for the trivial character
of $N_{o}$. Set
\[
\varphi_{1}(g)=\int_{N_{o}}\left\vert \varphi(ng)\right\vert dn.
\]
Then the proof of Proposition \ref{Fourier} shows that
\[
\varphi_{1}(g)\leq C_{d}a(g)^{\rho}(1+\left\Vert \log a(g)\right\Vert )^{-d}%
\]
for all $d>0$. Also (using the proof of Lemma \ref{Ltwo})
\[
\int_{G}\varphi_{1}(g)\left\vert \psi(g)\right\vert dg=\int_{N_{o}\backslash
G}\int_{N_{o}}\varphi_{1}(g)\left\vert \psi(ng)\right\vert dndg=
\]%
\[
\int_{N_{o}\backslash G}\varphi_{1}(g)\psi_{1}(g)dg<\infty
\]
with
\[
\psi_{1}(g)=\int_{N_{o}}\left\vert \psi(ng)\right\vert dg.
\]

\end{proof}

\begin{proposition}
Let $\chi$ be a character of $N_{o}$. If $\varphi\in\mathcal{C}_{[P],[\omega
]}(G)$ and $\psi\in\mathcal{C}_{[Q].[\eta]}(G)$ and $([P],[\sigma
])\neq([Q],[\eta])$ then
\[
\left\langle T_{\chi}\varphi,T_{\chi}\psi\right\rangle =0.
\]

\end{proposition}

\begin{proof}
We are looking at the integral
\[
\int_{N_{o}\backslash G}\int_{N_{o}}\chi(n_{1})^{-1}\varphi(n_{1}g)dn_{1}%
\int_{N_{o}}\chi(n_{2})\overline{\psi(n_{2}g)}dn_{2}dg.
\]
This integral converges absolutely by the previous lemma. We also note that it
can be written
\[
\int_{G}\varphi(g)\int_{N_{o}}\chi(n)\overline{\psi(ng)}dndg
\]
which also converges absolutely. Hence it can be rewritten as
\[
\int_{N_{o}}\chi(n)\int_{G}\varphi(g)\overline{\psi(ng)}dgdn=\int_{N_{o}}%
\chi(n)^{-1}\int_{G}\varphi(g)\overline{L(n)\psi(g)}dgdn.
\]
But $L(n)\psi\in\mathcal{C}_{[Q],[\eta]}(G)$, and $\left\langle \mathcal{C}%
_{[P],[\omega]}(G),\mathcal{C}_{[Q],[\eta]}(G)\right\rangle =0$.
\end{proof}

\begin{theorem}
\label{prelim-plancherel}Assume that $\chi$ is generic. The space
$\mathcal{C}(N_{o}\backslash G;\chi)$ is the completion of the orthogonal
direct sum
\[
\bigoplus_{\lbrack P]\in\mathcal{P}}\bigoplus_{[\omega]}T_{\chi}%
\mathcal{C}_{[P],[\omega]}(G).
\]

\end{theorem}

\begin{proof}
The Proposition combined with Theorem \ref{surjective} and the continuity of
$T_{\chi}$ imply the theorem.
\end{proof}

\begin{lemma}
If $\varphi\in\mathcal{C}(G)$ and if $P$ is a parabolic subgroup of $G$
containing $P_{o}$ then
\[
(T_{\chi}\varphi)^{P}(ma)=\int_{N_{o}}\chi(n_{o})\left(  L(n_{o}%
)\varphi\right)  ^{\bar{P}}(ma)dn_{o}.
\]

\end{lemma}

\begin{proof}
Let $\varphi\in\mathcal{C}(G)$. Then we calculate $\left(  \varphi_{\chi
}\right)  ^{P}(ma)$ for $m\in{}^{o}M_{P},a\in A_{P}$.
\[
\left(  \varphi_{\chi}\right)  ^{P}(ma)=a^{\rho_{P}}\int_{\bar{N}_{P}}%
\int_{N_{o}}\chi(n_{o})^{-1}\varphi(n_{o}\bar{n}ma)dn_{o}d\bar{n}.
\]
The integral converges absolutely so we may interchange the order of
integration. This implies that
\[
\left(  \varphi_{\chi}\right)  ^{P}(ma)=a^{\rho_{P}}\int_{N_{o}}\chi
(n_{o})^{-1}\int_{\bar{N}_{P}}\varphi(n_{o}\bar{n}ma)dn_{o}d\bar{n}=
\]
\[
a^{\rho_{P}}\int_{N_{o}}\chi(n_{o})\int_{\bar{N}_{P}}\varphi(n_{o}^{-1}\bar
{n}ma)dn_{o}d\bar{n}=\int_{N_{o}}\chi(n_{o})\left(  L(n_{o})\varphi\right)
^{\bar{P}}(ma)dn_{o}.
\]

\end{proof}

\section{\label{Analytic Cont}The holomorphic continuation of Jacquet
integrals}

The main purpose of this section is to describe our work on the analytic
continuation of Jacquet integrals. The proofs of the main results are
complicated. We will refer to the appropriate places in \cite{RRGII} Section
15.4. \ In addition the section includes some new estimates on Jacquet
integrals outside the range of convergence that are needed in the Whittaker
Plancherel theorem.

We will assume that $\chi$ is generic throughout this section. If $(\pi,H)$ is
a Hilbert representation of $G$ then the space of Whittaker vectors on
$H^{\infty}$ is
\[
Wh_{\chi}(H^{\infty})=\{\lambda\in(H^{\infty})^{\prime}|\lambda\circ
\pi(n)=\chi(n)\lambda,n\in N_{o}\}.
\]
Since a finitely generated $(Lie(G),K)$ module is finitely generated as a
$U(Lie(N_{o}))$ module (c.f. \cite{RRGI} 3.7.2 p.96) we have

\begin{lemma}
\label{Finitenes-Wh}(\cite{RRGII}, Lemma 15.4.3) If $(\pi,H)$ is admissible
and finitely generated then $\dim Wh_{\chi}(H^{\infty})<\infty$.
\end{lemma}

\begin{theorem}
\label{discretewhitt}Let $(\pi,H)$ be an irreducible square integrable
representation of $G$. If $w\in H_{K}$ then
\[
\lambda_{w}(v)=\int_{N_{o}}\chi(n)^{-1}\left\langle \pi(n)v,w\right\rangle dn
\]
is absolutely\ convergent for $v\in H^{\infty}$ and defines an element of
$Wh_{\chi}(H^{\infty})$. Furthermore, $Wh_{\chi}(H^{\infty})=\{\lambda
_{w}|w\in H_{K}\}$.
\end{theorem}

\begin{proof}
Let $w\in H_{K}$. If $v\in H^{\infty}$ then the function $T_{w}(v)(g)=$
$\left\langle \pi(g)v,w\right\rangle $ is in $\mathcal{C}(G)$. Furthermore,
$T_{w}:H^{\infty}\rightarrow\mathcal{C}(G)$ is continuous. Thus the first part
of our assertion follows from Proposition \ref{Fourier}.

We now prove the second assertion. Let $\lambda\in Wh_{\chi}(H^{\infty})$.
Let, for $\gamma\in\hat{K}$, $E_{\gamma}$ denote the orthogonal projection of
$H$ \ onto its $\gamma$--isotypic component, $H(\gamma)$. Then
\[
\lambda\circ E_{\gamma}(v)=\left\langle v,z_{\gamma}\right\rangle
\]
with $z_{\gamma}\in H(\gamma)$. \ Let $\left\Vert ...\right\Vert $ be a
standard norm.

First note that

1. The function $g\rightarrow\lambda(\pi(g)v)$ is in $\mathcal{C}%
(N_{o}\backslash G;\chi)$.

To prove this assertion we first note that Lemma 15.2.3 in \cite{RRGII} (the
Lemma should have the hypothesis that $\chi$ is generic) it is proved that
$\lambda$ is tame for every minimal parabolic subgroup $Q$. Also Theorem 5.5.4
combined with Proposition 5.1.2 in \cite{RRGI} implies that there exists
$\varepsilon>0$ and $C$ such if $a\in A$ and $\alpha(a)\geq1$ for all
$\alpha\in\Phi(Q,A_{Q})$ then if $v,w\in H_{K}$
\[
\left\vert \left\langle \pi(a)v,w\right\rangle \right\vert \leq C_{v,w}%
a^{-(1+\varepsilon)\rho_{Q}}.
\]
Theorem 15.2.5 in \cite{RRGII} now implies that here exists $\varepsilon>0$
and $p$, a continuous seminorm on $H^{\infty}$ such that if $a$ is as above
then
\[
\lambda(\pi(a)v)\leq p(v)a^{-\rho_{Q}-\varepsilon\rho_{Q}},v\in H^{\infty}.
\]
This inequality implies 1.

Consider
\[
I_{w}(\lambda)(v)=\int_{N_{o}\backslash G}\overline{\lambda(\pi(g)v)}%
\lambda_{w}(\pi(g)v)dg=
\]%
\[
\int_{G}\overline{\lambda(\pi(g)v)}\left\langle \pi(g)v,w\right\rangle dg.
\]
Which converges by 1. Note that
\[
\varphi\longmapsto\int_{G}\overline{\lambda(\pi(g)v)}\varphi(g)dg
\]
defines a continuous functional on $\mathcal{C}(G)$. Theorem \ref{K-Fourier}
implies that if we set
\[
\varphi_{\gamma}(g)=d(\gamma)\int_{K}\chi_{\gamma}(k)^{-1}\varphi(kg)dk
\]
(here $d(\gamma)$ is the dimension, and $\chi_{\gamma}$ is the character of
$\gamma$) then the series $\sum_{\gamma\in\hat{K}}\varphi_{\gamma}$ converges
to $\varphi$ in $\mathcal{C}(G)$. If $\varphi(g)=\left\langle \pi
(g)v,w\right\rangle $ then the Schur orthogonality relations (for $K$) imply
that
\[
\int_{G}\overline{\lambda(\pi(g)v)}\left\langle \pi(g)v,w\right\rangle
dg.=\sum_{\gamma\in\widehat{K}}\int_{G}\overline{\lambda(\pi(g)v)}\left\langle
E_{\gamma}\pi(g)v,w\right\rangle dg=
\]%
\[
\sum_{\gamma\in\widehat{K}}\int_{G}\overline{\lambda(E_{\gamma}\pi
(g)v)}\left\langle \pi(g)v,E_{\gamma}w\right\rangle dg=
\]%
\[
\sum_{\gamma\in\widehat{K}}\int_{G}\overline{\lambda(E_{\gamma}\pi
(g)v)}\left\langle E_{\gamma}\pi(g)v,w\right\rangle dg=
\]%
\[
\frac{1}{d(\pi)}\sum_{\gamma\in\widehat{K}}\left\Vert v\right\Vert
^{2}\left\langle z_{\gamma},w\right\rangle =\frac{1}{d(\pi)}\left\Vert
v\right\Vert ^{2}\overline{\lambda(w)}.\qquad\overset{}{(\ast)}%
\]
Here $d(\pi)$ is the formal degree of $\pi$ relative to our choice of Haar
measure on $G$. The formula $(\ast)$ implies that if $I_{w}(\lambda)(v)=0$ for
all $w$ for some $v\neq0$ then $\lambda=0$. Since $\dim Wh_{\chi}(H^{\infty
})<\infty$ the second assertion follows.
\end{proof}

\begin{corollary}
\label{whitproduct}(To the proof) If $\omega\in\mathcal{E}_{2}(G)$ then there
exists an inner product $\left(  ...,...\right)  _{\omega}$ on $Wh_{\chi
}(H_{\omega}^{\infty})$ such that if $\lambda,\mu\in Wh_{\chi}(H_{\omega
}^{\infty})\ $and $v,w\in H_{\omega}^{\infty}$ then
\[
\int_{N_{o}\backslash G}\lambda(\pi_{\omega}(g)v)\overline{\mu(\pi_{\omega
}(g)w)}dg=(\lambda,\mu)_{\omega}\left\langle v,w\right\rangle
\]
In particular, if $\left\Vert v\right\Vert =1$ then
\[
(\lambda,\mu)_{\omega}=\int_{N_{o}\backslash G}\lambda(\pi_{\omega
}(g)v)\overline{\mu(\pi_{\omega}(g)v)}dg
\]
independent of $v$. Furthermore if $w\in H_{\omega}^{\infty}$ then
\[
(\lambda,\lambda_{w})_{\omega}=\frac{1}{d(\omega)}\overline{\lambda(w)}.
\]

\end{corollary}

\begin{proof}
First note that if $\lambda,\mu\in Wh_{\chi}(H^{\infty})$ then
\[
\int_{N_{o}\backslash G}\lambda(\pi_{\omega}(g)v)\overline{\mu(\pi_{\omega
}(g)w)}dg
\]
defines an $G$--invariant Hermitian form on $H^{\infty}$. Thus
\[
\int_{N_{o}\backslash G}\lambda(\pi_{\omega}(g)v)\overline{\mu(\pi_{\omega
}(g)w)}dg=c(\lambda,\mu)\left\langle v,w\right\rangle
\]
with $c(\lambda,\mu)$ Hermitian. If $v\neq0$ then $\mathrm{Span}_{\mathbb{C}%
}(\pi_{\omega}(G)v)$ is dense in $H_{\omega}^{\infty}$ so if $\lambda\neq0$
then $\lambda(\pi_{\omega}(g)v)$ is not identically $0$. Hence, if
$\lambda\neq0,$ $c(\lambda,\lambda)>0$. Hence, $c(...,...)$ defines an inner
product on $Wh_{\chi}(H^{\infty})$ which we denote by$(\lambda,\mu)_{\omega}$.
Let $v\in H_{\omega}^{\infty}$ then
\[
\int_{N_{o}\backslash G}\lambda(\pi_{\omega}(g)v)\overline{\lambda_{w}%
(\pi_{\omega}(g)v)}dg=\int_{N_{o}\backslash G}\int_{N_{0}}\lambda(\pi_{\omega
}(g)v)\chi(n)\left\langle w,\pi_{\omega}(ng)v\right\rangle dndg=
\]%
\[
\int_{N_{o}\backslash G}\int_{N_{0}}\lambda(\pi_{\omega}(ng)v)\left\langle
w,\pi_{\omega}(ng)v\right\rangle dndg=\int_{G}\lambda(\pi_{\omega
}(g)v)\overline{\left\langle \pi_{\omega}(g)v,w\right\rangle }dg=
\]%
\[
\frac{1}{d(\omega)}\left\langle v,v\right\rangle \lambda(w)
\]
by formula $(\ast)$ in the proof of the preceding theorem.
\end{proof}

Let $P$ be a standard parabolic subgroup of $G$ with Langlands decomposition
$P=^{o}M_{P}A_{P}N_{P}$. If $(\sigma,H_{\sigma})$ is a Hilbert representation
of $^{o}M_{P}$ that is finitely generated and admissible, $\lambda\in
Wh_{\chi_{|^{o}M_{P}\cap N_{o}}}(H_{\sigma}^{\infty})$ and $f\in I_{\sigma
}^{\infty}$ and $\nu\in\left(  \mathfrak{a}_{P}^{\ast}\right)  _{\mathbb{C}}$
($\mathfrak{a}_{P}=Lie(A_{P})$) then we consider the integral
\[
J(P,\sigma,\nu)(\lambda)(f)=\int_{N_{P}}\chi(n)^{-1}\lambda(_{{}\bar{P}}%
f_{\nu}(n))dn.
\]
We set for $r\in\mathbb{R}$
\[
(\mathfrak{a}_{P}^{\ast})_{r}^{-}=\{\nu\in(\mathfrak{a}_{P}^{\ast
})_{\mathbb{C}}|\operatorname{Re}(\nu,\alpha)<r,\alpha\in\Phi(P,A_{P})\}.
\]
We have (\cite{RRGII} Lemma 15.6.5,p.398)

\begin{lemma}
Let $\sigma,\lambda$ be as above then there exists $c=c_{\sigma}$ such that if
$\nu\in(\mathfrak{a}_{P}^{\ast})_{c}^{-}$ the the integral $J(P,\sigma
,\nu)(\lambda)(f)$ converges absolutely for all $f\in I_{\sigma}^{\infty}$.
Furthermore, the map $\nu\longmapsto J(P,\sigma,\nu)(\lambda)(f)$ is
holomorphic on $(\mathfrak{a}_{P}^{\ast})_{c}^{-}$ for all $\lambda\in
Wh_{\chi_{|^{o}M_{P}\cap N_{o}}}(H_{\sigma}^{\infty})$ and all $f\in
I_{\sigma}^{\infty}$.
\end{lemma}

The next two results are the main results on the holomorphic continuation. The
first (\cite{RRGII} Theorem 15.6.5, p. 399) is

\begin{theorem}
Let $(\sigma,H_{\sigma})$ be an irreducible, admissible Hilbert representation
of $^{o}M_{P}$. Let $\lambda\in Wh_{\chi_{|^{o}M_{P}\cap N_{o}}}(H_{\sigma
}^{\infty})$. If $f\in I_{\sigma}^{\infty}$ then the map
\[
\nu\mapsto J(P,\sigma,\nu)(\lambda)(f)
\]
has a holomorphic continuation to $(\mathfrak{a}_{P}^{\ast})_{\mathbb{C}}$.
Furthermore, for all $\nu\in(\mathfrak{a}_{P}^{\ast})_{\mathbb{C}}$,
$J(P,\sigma,\nu)$ defines a linear bijection between $Wh_{\chi_{|^{o}M_{P}\cap
N_{o}}}(H_{\sigma}^{\infty})$ and $Wh_{\chi}(I_{\bar{P},\sigma,\nu}^{\infty})$.
\end{theorem}

And the second (\cite{RRGII}, Theorem 15.4.1 p. 381) whose proof and that of
the previous theorem comprise more than 20 complicated pages (pp. 382-405).

\begin{theorem}
\label{continuation}Assume that $(\sigma,H_{\sigma})$ is a square integrable
representation of $\ {}^{o}M_{P}$ then the constant $c_{\sigma}$ can be taken
to be $0$ and the map
\[
(\nu,f)\mapsto J(P,\sigma,\nu)(\lambda)(f)
\]
defines a continuous map on $(\mathfrak{a}_{P}^{\ast})_{0}^{-}\times
I_{\sigma}^{\infty}$ that is holomorphic om $\nu$ and linear in $f$. The map%
\[
\nu\mapsto J(P,\sigma,\nu)(\lambda)
\]
extends to a weakly holomorphic map from $\left(  \mathfrak{a}_{P}^{\ast
}\right)  _{\mathbb{C}}$\ to $\left(  I_{\sigma}^{\infty}\right)  ^{\prime}$.
Finally, for each $\nu\in\left(  \mathfrak{a}_{P}^{\ast}\right)  _{\mathbb{C}%
}$, $J(P,\sigma,\nu)$ defines a linear bijection between $Wh_{\chi_{|^{o}%
M_{P}\cap N_{o}}}(H_{\sigma}^{\infty})$ and $Wh_{\chi}(I_{\bar{P},\sigma,\nu
}^{\infty})$.
\end{theorem}

Our next task is to derive a tempered estimate, in $\nu$, on $J(P,\sigma
,i\nu)$. Assume that $G$ has compact center. Let $P=N_{P}A_{P}{}^{o}M_{P}$ be
a $P_{o}$ standard parabolic subgroup of $G$ with standard Langlands
decomposition. As usual, $\bar{P}=\bar{N}_{P}A_{P}{}^{o}M_{P}$. If $g\in G$
then $g=\bar{n}amk$ with $\bar{n}\in\bar{N}_{P},a\in A_{P},m\in{}^{o}%
M_{P},k\in K$ and we use the notation $n_{\bar{P}}(g)=\bar{n},a_{\bar{P}%
}(g)=a,m_{\bar{P}}(g)=m$ and $k_{\bar{P}}(g)=k.$ Than as is usually noted
$n_{\bar{P}}$ and $a_{\bar{P}}$ define smooth functions but only $m_{\bar{P}%
}(g)k_{\bar{P}}(g)$ defines a function.

Fix $(\sigma,H_{\sigma})$ an irreducible square integrable representation of
$^{o}M_{P}$. \ Let $\pi_{\nu}$denote $\pi_{\bar{P},\sigma,\nu}$ and $u_{\nu}$
denote $_{\bar{P}}u_{\nu}$ for $u\in I_{\sigma}^{\infty}$\ We now estimate
$J(P,\sigma,\nu)(\nu)$ for $v\in I_{\sigma}^{\infty},\lambda\in Wh_{\chi
_{|N_{o}\cap M_{P}}}(H_{\sigma}^{\infty})$ and $\operatorname{Re}(\nu
,\alpha)<0$ for all $\alpha\in\Phi(P,A)$. We have
\[
J(P,\sigma,\nu)(\lambda)(v)=\int_{N_{P}}\chi(n)^{-1}\lambda(v_{\nu}(n))dn=
\]%
\[
\int_{N_{P}}\chi(n)^{-1}a_{\bar{P}}(n)^{\nu-\rho_{P}}\lambda(\sigma(m_{\bar
{P}}(n))v(k_{\bar{P}}(n)))dn.
\]
Thus
\[
\left\vert J(P,\sigma,\nu)(\lambda)(v)\right\vert \leq\int_{N_{P}}a_{\bar{P}%
}(n)^{\operatorname{Re}\nu-\rho_{P}}\left\vert \lambda(\sigma(m_{\bar{P}%
}(n))v(k_{\bar{P}}(n)))\right\vert dn\leq
\]%
\[
\int_{N_{P}}a_{\bar{P}}(n)^{\operatorname{Re}\nu-\rho_{P}}\Xi(a(m_{\bar{P}%
}(n))dn\left\Vert \lambda\right\Vert q(v)
\]
by Corollary \ref{xi-estimate}, here $q$ is a continuous seminorm on
$I_{\sigma}^{\infty}$. We have proved

\begin{lemma}
Let $(\sigma,H_{\sigma})$ be an irreducible square integrable representation
of $^{o}M_{P},v\in I_{\sigma}^{\infty},\lambda\in Wh_{\chi_{|N_{o}\cap M_{P}}%
}(H_{\sigma}^{\infty})$ and $\operatorname{Re}(\nu,\alpha)<0$ for all
$\alpha\in\Phi(P,A)$. Then%
\[
\left\vert J(P,\sigma,\nu)(\lambda)(v)\right\vert \leq\int_{N_{P}}a_{\bar{P}%
}(n)^{\operatorname{Re}\nu-\rho_{P}}\left\vert \lambda(\sigma(m_{\bar{P}%
}(n))v(k_{\bar{P}}(n)))\right\vert dn
\]%
\[
\leq\left\Vert \lambda\right\Vert q(v)\int_{N_{P}}a(n)^{\operatorname{Re}%
\nu-\rho_{P}}dn.
\]
with $q$ a continuous seminorm on $I_{\sigma}^{\infty}$.
\end{lemma}

Let $F$ be the irreducible finite dimensional representation of $G$ such
that${}{}^{o}M_{P}$ acts trivially on $F^{\mathfrak{n}_{P}}$ and $A_{P}$ acts
by $a^{4\rho}$. (See the example in $10.2.1$ in \cite{RRGII}). We note that
$F=F^{\mathfrak{n}_{P}}\oplus\mathfrak{\bar{n}}_{P}F$ as an $M_{P}$--module.
Thus
\[
F/\mathfrak{\bar{n}}_{P}F\cong F^{\mathfrak{n}_{P}}%
\]
as an $M_{P}$--module. We have a surjective homomorphism
\[
I_{\bar{P},\sigma,\nu-4\rho}^{\infty}\otimes F\rightarrow I_{\bar{P}%
,\sigma,\nu}^{\infty}%
\]
this induces an injective map
\[
Wh_{\chi}(I_{\bar{P},\sigma,\nu}^{\infty})\rightarrow Wh_{\chi}(I_{\bar
{P},\sigma,\nu-4\rho}^{\infty}\otimes F).
\]
We also note that we have a map
\[
\Gamma:Wh_{\chi}(I_{\bar{P},\sigma,\nu-4\rho}^{\infty})\otimes F^{\ast
}\rightarrow Wh_{\chi}(I_{\bar{P},\sigma,\nu-4\rho}^{\infty}\otimes F)
\]
(see Theorem 15.5.7 p.393 \cite{RRGII}). This map is defined on the space
$\widetilde{Wh}_{\chi}(I_{\bar{P},\sigma,\nu-4\rho}^{\infty})\otimes F^{\ast}$
with
\[
\widetilde{Wh}_{\chi}(I_{\bar{P},\sigma,\nu-4\rho}^{\infty})=\{T\in(I_{\bar
{P},\sigma,\nu-4\rho}^{\infty})^{\prime}|(X-d\chi(X)^{k})T=0,X\in
\mathfrak{n}_{o}\}
\]
which is a $\mathfrak{g}=Lie(G)$--module and is given as follows. Let
$H\in\mathfrak{a}_{P}$ be such that $\alpha(H)=1$ for $\alpha$ the non-zero
restriction of a simple root to $\mathfrak{a}_{P}$. Then the eigenvalues of
$H$ are of the form $4\rho_{P}(H)-k$ \ with $k\geq0$ and $k\in\mathbb{Z}$. We
set $F_{k}^{\ast}$ equal to the $4\rho_{P}(H)-k$ eigenspace of $H$ acting on
$F^{\ast}$ and let $p_{k}$ be the projection of $F^{\ast}$onto $F_{k}^{\ast}$.
\ Then
\[
\Gamma=\sum L_{k}(I\otimes p_{k})
\]
with $L_{k}\in U(\mathfrak{g})$ depending only on $\chi$ and $k$ (see the
definition of $\Gamma_{k}$ in 15.5.7 in \cite{RRGII}) . Thus, if $\lambda
_{1},...,\lambda_{d}$ is a basis of $Wh_{\chi_{|N_{o}\cap M_{P}}}(H_{\sigma
}^{\infty})$ and if $f_{1},...,f_{r}$ is a basis of $F^{\ast}$ compatible with
the grade (i.e. $f_{j}\in F_{k(j)}^{\ast})$ then the elements
\[
L_{k(j)}\left(  J_{P,\sigma,\nu-4\rho_{P}}(\lambda_{i})\otimes f_{j}\right)
,1\leq i\leq d,1\leq j\leq r
\]
form a basis of $Wh_{\chi}(I_{\bar{P},\sigma,\nu-4\rho}^{\infty}\otimes F)$.
This basis can be written
\[
\sum_{p,q}a_{ij},_{pq}(\nu)J(P,\sigma,\nu-4\rho_{P})(\lambda_{p})\circ
d_{ij}\otimes f_{q}%
\]
with $a_{ij,pq}$ \ a polynomial in $\nu$ and $d_{ij}$ is a differential
operator on $I_{\sigma}^{\infty}$ corresponding to the action of
$U(\mathfrak{g)}$. Let $T=T_{1}\circ T_{2}$ with
\[
T_{2}(\varphi\otimes f)(g)=\varphi(g)\otimes gv
\]
and
\[
T_{1}(h)(g)=\left(  I\otimes Q\right)  (h(g))
\]
with $Q$ the natural surjection $F\rightarrow F/\mathfrak{\bar{n}}_{P}F$.
Thus
\[
T:I_{\bar{P},\sigma,\nu-4\rho}^{\infty}\otimes F\rightarrow I_{\bar{P}%
,\sigma_{\nu-4\rho}\otimes F_{|\bar{P}}}^{\infty}%
\]
(here $\sigma_{\nu}$ is the $\bar{P}$ representation on $H_{\sigma}^{\infty}$
with $\bar{N}_{P}$ acting trivially, $A_{P}$ acting by $a\longmapsto
a^{-\rho_{P}+\nu-4\rho_{P}}$) \
\[
T_{1}(h)(g)=\left(  I\otimes Q\right)  (h(g))
\]
with $Q$ the natural surjection $F\rightarrow F/\mathfrak{\bar{n}}_{P}F$. Note
that neither of these maps depend on $\nu$.

With this in mind we have our tempered estimate:

\begin{theorem}
\label{Tempered-Estimate}There exists $m>0$ and a continuous seminorm,
$\gamma_{1}$ on $I_{\sigma}^{\infty}$ such that if $u\in I_{\sigma}^{\infty}$
then
\[
\left\vert J(P,\sigma,i\nu)(\lambda)(u)\right\vert \leq\gamma_{1}%
(u)(1+\left\Vert \nu\right\Vert )^{m}\left\Vert \lambda\right\Vert
\]
for $\nu\in\mathfrak{a}^{\ast}$.
\end{theorem}

\begin{proof}
We note that we can choose a $K$--Fr\'{e}chet, summand $Z$ of the space
$I_{\sigma}^{\infty}\otimes F$ such then $T$ is a $K$--isomorphism of $Z$ to
$I_{\sigma}^{\infty}$. Let $u_{1},...,u_{r}$. be the dual basis to the basis
$f_{1},...,f_{r}$of $F$. If $z\in Z$ then
\[
z=\sum z_{i}\otimes u_{i}.
\]
Thus
\[
J(P,\sigma,i\nu)(\lambda)T(z)=\sum_{ij}b_{ij}(\nu,\lambda)J\left(
P,\sigma,i\nu-4\rho)\right)  (\lambda_{i})\circ d_{ij}z_{j}%
\]

Applying the above Lemma we have
\[
\left\vert J(P,\sigma,i\nu)(\lambda)(T(z))\right\vert \leq\left\Vert
\lambda\right\Vert \sum_{ij}\gamma(d_{ij}z_{j})\left(  1+\left\Vert
v\right\Vert \right)  ^{m}\left\Vert \lambda_{i}\right\Vert \int_{N_{P}%
}a(n)^{-5\rho_{P}}dn.
\]
This implies that there exists a continuous seminorm, $\gamma_{1}$, on
$I_{\sigma}^{\infty}$ such that
\[
\left\vert J(P,\sigma,i\nu)(\lambda)(T(z))\right\vert \leq\left\Vert
\lambda\right\Vert \gamma_{1}(Tz)\left(  1+\left\Vert \nu\right\Vert \right)
^{m}.
\]

\end{proof}

\begin{corollary}
Let $\alpha\in\mathcal{S}(\mathfrak{a}^{\ast})$ then if $\lambda\in
Wh_{\chi_{|M_{P}\cap N_{o}}}(H_{\sigma}^{\infty}),v,w\in I_{\sigma}^{\infty}$
then
\[
\int_{\mathfrak{a}^{\ast}}J(P,\sigma,i\nu)(\lambda)(\pi_{\bar{P},\sigma,i\nu
}(g)v)\alpha(\nu)d\nu
\]
converges absolutely and defines an element of $C^{\infty}(N_{o}\backslash
G;\chi).$
\end{corollary}

We will see that if $\alpha\in\mathcal{S}(\mathfrak{a}^{\ast})\mu
(\sigma,i\cdot)$ then the integral in the Corollary defines an element of
$\mathcal{C}(N_{o}\backslash G;\chi)$.

\begin{theorem}
\label{J-estimate}Let $\omega\subset\mathfrak{a}_{P}^{\ast}$ be a compact set
and let $u\in(I_{\sigma}^{\infty})_{K}$then there exist $r>0$ , C$_{u}$, and
$d$ such that if then
\[
\left\vert J(P,\sigma,i\nu)(\lambda)(\pi_{\bar{P},\sigma,i\nu}%
(ak)u)\right\vert \leq C_{u}\left\Vert \lambda\right\Vert a^{\rho_{o}%
}(1+\left\Vert \log a\right\Vert )^{d}%
\]
for $\nu\in\mathfrak{\omega}$ and $a\in A_{o},k\in K$.
\end{theorem}

\begin{proof}
If $\nu$ is fixed then the estimate
\[
\left\vert J(P,\sigma,i\nu)(\lambda)(\pi_{\bar{P},\sigma,i\nu}%
(ak)u)\right\vert \leq C_{u,\nu}\left\Vert \lambda\right\Vert a^{\rho_{o}%
}(1+\left\Vert \log a\right\Vert )^{d}%
\]
follows from the tameness (see 15.2.1 \cite{RRGII}) of the elements of
$Wh_{\chi}(I_{\bar{P}_{o},\sigma,i\nu}^{\infty})$. Indeed, for each minimal
parabolic $Q$ in the Weyl group orbit of $P_{o}$ one gets the estimate
\[
\left\vert J(P,\sigma,i\nu)(\lambda)(\pi_{\bar{P},\sigma,i\nu}%
(ak)u)\right\vert \leq C_{\nu,Q}(u)\left\Vert \lambda\right\Vert a^{-\rho_{Q}%
}(1+\left\Vert \log a\right\Vert )^{d}%
\]
for $a\in Cl(A_{Q}^{+})$ using the methods of 15.2.5 (here we are using the
fact that $\pi_{\bar{P},\sigma.i\nu}$ is tempered, see 5.1, especially the
beginning of 5.1.1 in \cite{RRGI}) and note that
\[
a^{-\rho_{Q}}\leq a^{\rho_{o}}%
\]
if $a\in Cl(A_{Q}^{+})$. It is not hard to see that the constants involved in
the definition of tameness in 15.2.1 of \cite{RRGII} can be taken constant in
$\nu$ in compacta. One now proves the theorem by establishing the dependence
on parameters of the asymptotic parameters using the method of 12.4.8 ,12.4.9
and 12.4.10 on the argument in 15.2.4 used to prove Theorem 15.2.5 all in
\cite{RRGII}. Here the key aspect of the argument is to prove the estimates in
the end of 12.4.9 and 12.4.10 using the same methodology involving systems
ordinary differential equations depending on parameters as in 12.4.8.
\end{proof}

\section{Whittaker cusp forms and the Harish-Chandra discrete series}

\begin{theorem}
\label{Inner-product-Psi}Let $P$ be a standard parabolic subgroup of $G$. Let
$(\sigma,H_{\sigma})\in\lbrack\sigma]\in\mathcal{E}_{2}({}^{o}M_{P})$. For
$\lambda\in Wh_{\chi_{|^{o}M_{P}\cap N_{o}}}(H_{\sigma}^{\infty})$, and
$\ \alpha\in C_{c}^{\infty}\mathcal{(}\mathfrak{a}_{P}^{\ast})$ set for
$v\in(I_{\sigma}^{\infty})_{K}$
\[
\Psi(P,\alpha,\sigma,\lambda,v)(g)=\int_{\mathfrak{a}_{P}^{\ast}}%
J(P,\sigma,i\nu)(\lambda)(\pi_{\overline{P},\sigma,i\nu}(g)v)\alpha(\nu)d\nu.
\]
If $f\in\mathcal{C}(N_{o}\backslash G,\chi)$ then
\[
\int_{N_{o}\backslash G}\overline{\Psi(P,\alpha,\sigma,\lambda,v)(g)}f(g)dg=
\]
\[
\int_{N_{o}\cap M_{P}\backslash{}^{o}M_{P}\times K\times A}(\int
_{\mathfrak{a}_{P}^{\ast}}a^{-i\nu}\overline{\alpha(\nu)\lambda\left(
\sigma(m)v(k)\right)  }(R(k)f)^{P}(ma)d\nu)dmdkda.
\]

\end{theorem}

\begin{proof}
Theorem \ref{J-estimate} implies that the left hand side of the equation
converges and defines a continuous functional on $\mathcal{C}(N_{o}\backslash
G,\chi)$. Also Theorem \ref{15.3.Replace} implies that the right hand side of
the equation defines a continuous functional in $f$. Thus both sides of this
equation are continuous in $f$ so it is enough to prove the equation for
$\left\vert f\right\vert \in C_{c}(N_{o}\backslash G)$. Set $\rho=\rho_{P}$.
If $z\in\mathbb{C}$ set
\[
\varphi(z,g)=\int_{\mathfrak{a}_{P}^{\ast}}J(P,\sigma,i\nu-z\rho)(\lambda
)(\pi_{\overline{P},\sigma,i\nu-z\rho}(g)v)\alpha(\nu)d\nu
\]
which is entire in $z.$Also
\[
\int_{N_{o}\backslash G}\overline{\varphi(z,g)}f(g)dg
\]
is entire in $\bar{z}$. Thus
\[
\int_{N_{o}\backslash G}\overline{\Psi(P,\alpha,\sigma,\lambda,v)(g)}%
f(g)dg=\lim_{\varepsilon\rightarrow+0}\int_{N_{o}\backslash G}\overline
{\varphi(\varepsilon,g)}f(g)dg.
\]
If $\varepsilon>0$ then
\[
\int_{N_{o}\backslash G}\overline{J(P,\sigma,i\nu-\varepsilon\rho
)(\lambda)(\pi_{\overline{P},\sigma,i\nu-\varepsilon\rho}(g)v)}f(g)dg=
\]%
\[
\int_{N_{o}\backslash G}\overline{\lambda\left(  \int_{N_{P}}\chi
(n)^{-1}v_{i\nu-\varepsilon\rho}(ng)dn\right)  }f(g)dg.
\]
Here, the inner integral is convergent by Lemma 39. We note that $N_{0}%
=N_{P}(N_{o}\cap M_{P})$ and the expression $n_{o}=nn^{\ast}$ with $n_{P}\in
N_{P}$ and $n^{\ast}\in N_{o}\cap M_{P}$ is unique. Now,
\[
\int_{N_{o}\backslash G}\overline{\lambda\left(  \int_{N_{P}}\chi
(n)^{-1}v_{i\nu-\varepsilon\rho}(ng)dn\right)  }f(g)dg=
\]%
\[
\int_{N_{o}\backslash G}\overline{\lambda\left(  \int_{N_{P}}v_{i\nu
-\varepsilon\rho}(ng)\overline{f(ng)}dn\right)  }dg=\int_{N_{o}\cap
M_{P}\backslash G}\overline{\lambda\left(  \int v_{i\nu-\varepsilon\rho
}(g)\right)  }f(g)dg.
\]
We now include the integral over $\alpha$ and have
\[
\int_{\bar{N}_{P}\times N_{o}\cap M_{P}\backslash{}^{o}M_{P}\times A_{P}\times
K\times\mathfrak{a}_{P}^{\ast}}\overline{\alpha(\nu)}a^{2\rho_{P}}%
\overline{\lambda(v_{i\nu-\varepsilon\rho}(\bar{n}mak))}f(\bar{n}mak)d\bar
{n}dmdadkd\nu=
\]%
\[
\int_{\bar{N}_{P}\times N_{o}\cap M_{P}\backslash{}^{o}M_{P}\times A_{P}\times
K\times\mathfrak{a}_{P}^{\ast}}\overline{\alpha(\nu)}a^{2\rho}a^{\rho_{P}%
-i\nu-\varepsilon\rho}\overline{\lambda(\sigma(m)v(k))}\times
\]%
\[
\left(  R(k)f\right)  (\bar{n}ma)d\bar{n}dmdadkd\nu=
\]%
\[
\int_{N_{o}\cap M_{P}\backslash{}^{o}M_{P}\times A\times K\times
\mathfrak{a}_{P}^{\ast}}a^{-\varepsilon\rho}\int_{\mathfrak{a}^{\ast}%
}\overline{\alpha(\nu)}a^{-i\nu}d\nu\overline{\lambda(\sigma(m)v(k))}\left(
R(k)f\right)  ^{P}(ma)dmdadk.
\]
Lemma \ref{compact-estimate} implies that there exists $C>0$ such that if
$a\in A_{P}$ and if $m\in{}^{o}M_{P}$ and if $\left(  R(k)f\right)
^{P}(ma)\neq0$ then $a^{\rho}\geq C>0$. Thus $a^{-\varepsilon\rho}\leq
C^{-\varepsilon}$ so we can apply dominated convergence the limit as
$\varepsilon\rightarrow+0$ to deduce the theorem.
\end{proof}

\begin{lemma}
Let $\nu_{o}\in\mathfrak{a}_{P}^{\ast}$ and $\alpha_{j}\in C_{c}^{\infty
}(\mathfrak{a}_{P}^{\ast})$ for $j=1,2,...$be non-negative valued functions satisfying

a) supp $\alpha_{j+1}\subset$ supp $\alpha_{j}$ and $\cap_{j\geq1}$supp
$\alpha_{j}=\{\nu_{o}\}.$

b) $\int_{\mathfrak{a}_{P}^{\ast}}\alpha_{j}(\nu)d\nu=1$.

In the notation of the previous theorem set
\[
T_{j}(f)=\int_{N_{o}\backslash G}\overline{\Psi(P,\alpha_{j},\sigma
,\lambda,\nu)(g)}f(g)dg
\]
for $f\in\mathcal{C}(N_{o}\backslash G,\chi)$. Then
\[
\lim_{j\rightarrow\infty}T_{j}(f)=\int_{N_{o}\backslash G}\overline
{J(P,\sigma,i\nu_{o})(\lambda)(\pi_{\overline{P},\sigma,i\nu_{o}}%
(g)v)}f(g)dg.
\]

\end{lemma}

\begin{proof}
Let $\omega$ be the support of $\alpha_{1}$ then Theorem \ref{J-estimate}
implies that there exists a constants $C$,$d$ depending on $\omega$ and $v$
such that such that
\[
\left\vert J(P,\sigma,i\nu_{o})(\lambda)(\pi_{\overline{P},\sigma,i\nu_{o}%
}(ak)v)\right\vert \leq Ca^{\rho}(1+\log\left\Vert a\right\Vert )^{d}.
\]
We write
\[
T_{j}(f)-\int_{N_{o}\backslash G}\overline{J(P,\sigma,i\nu_{o})(\lambda
)(\pi_{\overline{P},\sigma,i\nu_{o}}(g)v)}f(g)dg=
\]%
\[
\int_{N_{o}\backslash G}\int_{\mathfrak{a}_{P}^{\ast}}(J(P,\sigma
,i\nu)(\lambda)(\pi_{\overline{P},\sigma,i\nu}(g)v)-J(P,\sigma,i\nu
_{o})(\lambda)(\pi_{\overline{P},\sigma,i\nu_{o}}(g)v))\alpha_{j}(\nu)d\nu
f(g)dg
\]
To prove the limit formula we note that if $\varepsilon>0$ there exists a
compact subset $U\subset N_{o}\backslash G$ such that
\[
\left\vert \int_{N_{o}\backslash G-U}\int_{\mathfrak{a}_{P}^{\ast}}%
(J(P,\sigma,i\nu)(\lambda)(\pi_{\overline{P},\sigma,i\nu}(g)v)-J(P,\sigma
,i\nu_{o})(\lambda)(\pi_{\overline{P},\sigma,i\nu_{o}}(g)v))\alpha_{j}%
(\nu)d\nu f(g)dg\right\vert <\frac{\varepsilon}{2}%
\]
for all $j$. Indeed for each $r$~$>0$ there exists $B$ such that if $\nu\in$
supp $a_{1}$ then
\[
\left\vert J(P,\sigma,i\nu)(\pi_{\overline{P},\sigma,i\nu}(g)v)-J(P,\sigma
,i\nu_{o})(\lambda)(\pi_{\overline{P},\sigma,i\nu_{o}}(g)v))\right\vert
\]%
\[
\leq Ba_{P_{o}}(g)^{\rho}(1+\log\left\Vert a_{P_{o}}(g)\right\Vert )^{d}%
\]
Thus for each $r$ there exists $B_{r}$ such that
\[
\left\vert J(P,\sigma,i\nu)(\lambda)(\pi_{\overline{P},\sigma,i\nu
}(g)v)-J(P,\sigma,i\nu_{o})(\lambda)(\pi_{\overline{P},\sigma,i\nu_{o}%
}(g)v)\right\vert \left\vert f(g)\right\vert
\]%
\[
\leq B_{r}a(g)^{2\rho}(1+\log\left\Vert a(g)\right\Vert )^{d-r}.
\]
Let $r$ be large enough that
\[
\int_{N_{o}\backslash G}a(g)^{2\rho}(1+\log\left\Vert a(g)\right\Vert
)^{d-r}dg<\infty.
\]
If $\delta>0$ is given then there exists $U_{\delta}\subset N_{o}\backslash G
$ compact such that
\[
\int_{N_{o}\backslash G-U_{\delta}}a(g)^{2\rho}(1+\log\left\Vert
a(g)\right\Vert )^{d-r}dg<\delta.
\]
Thus
\[
\left\vert \int_{N_{o}\backslash G-U_{\delta}}\int_{\mathfrak{a}_{P}^{\ast}%
}(J(P,\sigma,i\nu)(\lambda)(\pi_{\overline{P},\sigma,i\nu}(g)v)-J(P,\sigma
,i\nu_{o})(\lambda)(\pi_{\overline{P},\sigma,i\nu_{o}}(g)v))\alpha_{j}%
(\nu)d\nu f(g)dg\right\vert \leq
\]%
\[
\int_{N_{o}\backslash G-U_{\delta}}\int_{\mathfrak{a}_{P}^{\ast}}\left\vert
J(P,\sigma,i\nu)(\lambda)(\pi_{\overline{P},\sigma,i\nu}(g)v)-J(P,\sigma
,i\nu_{o})(\pi_{\overline{P},\sigma,i\nu_{o}}(g)v))\right\vert \alpha_{j}%
(\nu)d\nu\left\vert f(g)\right\vert dg\leq
\]%
\[
B_{r}\delta\int_{\mathfrak{a}_{P}^{\ast}}\alpha_{j}(\nu)d\nu=B_{r}\delta.
\]
Let $\delta$ be such that $B_{r}\delta<\frac{\varepsilon}{2}$ and
$U=U_{\delta}$. To complete the proof of the lemma we must prove that there
exists $j_{o}$ such that if $j\geq j_{o}$ then
\[
\left\vert \int_{U}\int_{\mathfrak{a}_{P}^{\ast}}(J(P,\sigma,i\nu
)(\pi_{\overline{P},\sigma,i\nu}(g)v)-J(P,\sigma,i\nu_{o})(\pi_{\overline
{P},\sigma,i\nu_{o}}(g)v))\alpha_{j}(\nu)d\nu f(g)dg\right\vert <\frac
{\varepsilon}{2}.
\]
Since $U\times$ supp $\alpha_{1}$ is compact and
\[
g,\nu\mapsto J(P,\sigma,i\nu)(\pi_{\overline{P},\sigma,i\nu}(g)v)-J(P,\sigma
,i\nu_{o})(\pi_{\overline{P},\sigma,i\nu_{o}}(g)v))\alpha_{j}(\nu)f(g)
\]
is continuous a standard compactness argument proves that given $\mu>0$ there
exist $V$ a neighborhood of $\nu_{o}$ in $\mathfrak{a}_{P}^{\ast}$ such that
\[
\left\vert (J(P,\sigma,i\nu)(\lambda)(\pi_{\overline{P},\sigma,i\nu
}(g)v)-J(P,\sigma,i\nu_{o})(\lambda)(\pi_{\overline{P},\sigma,i\nu_{o}%
}(g)v))f(g)\right\vert <\mu
\]
if $\left(  g,v\right)  \in U\times V.$ Thus if the support of $\alpha_{j_{o}%
}\subset V$ and if $j\geq j_{o}$ then
\[
\left\vert \int_{U}\int_{\mathfrak{a}_{P}^{\ast}}(J(P,\sigma,i\nu
)(\lambda)(\pi_{\overline{P},\sigma,i\nu}(g)v)-J(P,\sigma,i\nu_{o}%
)(\lambda)(\pi_{\overline{P},\sigma,i\nu_{o}}(g)v))\alpha_{j}(\nu)d\nu
f(g)dg\right\vert
\]%
\[
\leq\mu\int_{\mathfrak{a}_{P}^{\ast}}\alpha_{j}(\nu)d\nu.
\]
So, (say) take $\mu=\frac{\varepsilon}{4}.$
\end{proof}

\begin{theorem}
If $(\pi,H)$ is an irreducible unitary subrepresentation of $L^{2}%
(N_{o}\backslash G,\chi)$ then there exists an irreducible square integrable
representation $(\sigma,V)$ and $\lambda\in Wh_{\chi}(V^{\infty})$ such that
\[
H^{\infty}=\{g\mapsto\lambda(\sigma(g)v)|v\in V^{\infty}\}.
\]

\end{theorem}

\begin{proof}
We have seen in \cite{RRGII} Lemma 15.1.1 p. 365 that the support of
$L^{2}(N_{o}\backslash G,\chi)$ as an abstract representation is contained in
the set of tempered representations. Thus it is enough to show that if
$(\sigma,V)$ is an irreducible tempered representation such that there exists
a unitary intertwining operator $L:V\rightarrow L^{2}(N_{o}\backslash G,\chi)
$ then $\sigma$ must be square integrable on $G$. Assume the contrary. We will
show that this leads to a contradiction. We note that if we set for $v\in
V^{\infty}$, $\lambda(v)=L(v)(e)$ ($e$ the identity element of $G$) then
$\lambda\in Wh_{\chi}(V^{\infty})$ and $L(v)(g)=$ $\lambda(\pi(g)v)$ defines
an element of $^{o}\mathcal{C}(N_{o}\backslash G;\chi)$ (Theorem
\ref{Discrete+Cuspform}). Since $(\sigma,V)$ is tempered there exists
$(P\,A_{P}) $ a $P_{o}$ standard parabolic pair with $P\neq G,$ $(\mu.H_{\mu
})$ a square integrable irreducible representation of $\,{}^{o}M_{P}$ and
$\nu_{o}\in\mathfrak{a}_{P}^{\ast}$ such that $(\sigma,V) $ is equivalent with
a direct summand of $I_{\bar{P},\sigma,i\nu_{o}}$ (c.f, \cite{RRGI}
Proposition 5.2.5) thus Theorem \ref{continuation} implies that there exists
$u\in I_{\bar{P},\mu,i\nu_{o}}^{\infty}$ and $\tau\in Wh_{\chi_{|M_{P}\cap
N_{o}}}(H_{\mu}^{\infty})$ such that
\[
L(v)(g)=J(P,\sigma,i\nu)(\tau)(\pi_{\bar{P},\mu,i\nu_{o}}(g)u),g\in G\text{.}%
\]
We have in the notation of the previous lemma
\[
\lim_{j\rightarrow\infty}\int_{N_{o}\backslash G}\overline{\Psi(P,\alpha
_{j},\mu,u,\tau)(g)}L(u)(g)dg=
\]%
\[
\int_{N_{o}\backslash G}\overline{J(P,\sigma,i\nu_{o})(\tau)(\pi_{\bar
{P},\sigma,i\nu_{o}}(g)u)}L(u)(g)dg=\left\Vert L(u)\right\Vert ^{2}.
\]
On the other hand by the preceding theorem
\[
\int_{N_{o}\backslash G}\overline{\Psi(P,\alpha_{j},\mu,v,\tau)(g)}L(u)(g)dg=
\]%
\[
\int_{N_{o}\cap M_{P}\backslash{}^{o}M_{P}\times K\times A}(\int
_{\mathfrak{a}_{P}^{\ast}}a^{-i\nu}\overline{\alpha_{j}(\nu)\lambda\left(
\sigma(m)v(k)\right)  }(R(k)L(u))^{P}(ma)d\nu)dmdkda.
\]
Which is $0$ because $L(u)\in{}^{o}\mathcal{C}(N_{o}\backslash G;\chi)$ . This
is the desired contradiction completing the proof of the theorem.
\end{proof}

\section{\label{Firststeps}First steps on the continuous spectrum}

Beuzart-Plessis' Proposition B.3.1 in \cite{raphael} has as an immediate consequence

\begin{proposition}
\label{fundamental}Let $P_{o}$ \ be a minimal parabolic subgroup of $G$ and
$N_{o}$ (as usual) its unipotent radical
\[
\int_{\lbrack N_{o},N_{o}]}a_{\bar{P}_{o}}(n)^{-\rho_{P_{o}}}(1+\log\left\Vert
n\right\Vert )^{r}dn<\infty
\]
for \ all $r$.
\end{proposition}

Note that if the domain of integration is $N_{o}$ rather than $[N_{o},N_{o}]$
then
\[
\int_{N_{o}}a_{\bar{P}_{o}}(n)^{-\rho_{P_{o}}}dn=\infty
\]
The result in \cite{raphael} is actually much stronger. We have stated the
result in this form because it is all that we need and we have an elementary
proof of this result for $GL(n,\mathbb{R})$ and $GL(n,\mathbb{C})$ and real
reductive groups of split rank 1 and several groups of rank 2 is given in
\cite{tame}. Since this result will be the basis of our determination of the
continuous spectrum the use of this form of the proposition leads to an
elementary proof for the important special case of $GL(n)$.

We will be using notation from the material preceding Lemma 1 in section 2 in
particular $\left\Vert ...\right\Vert $ is the standard norm defined therein.
Also, $v$ will denote a unit vector in $\wedge^{\dim N_{o}}Lie(\bar{N}_{o})$.

\begin{corollary}
\label{polyest}Let $Lie(N_{o})=Lie[N_{o},N_{o}]\oplus Z$ with $Ad(A_{o})Z=Z$.
There exists $s\in\mathbb{R}$ and $\ $for each $\omega\subset G$ a compact
subset a constant $C_{\omega}$ such that if $X\in Z$ and if $g\in\omega$ such
that if $d$ is given there exists $C_{d}$such that
\[
\int_{\lbrack N_{o},N_{o}]}a_{\bar{P}}(n\exp Xg)^{-\rho_{o}}(1+\log\left\Vert
n\right\Vert )^{d}dn\leq C_{\omega}C_{d}(1+\left\Vert X\right\Vert )^{s}.
\]

\end{corollary}

\begin{proof}
If $v$ is as as above and if $g\in G$ then
\[
\left\Vert g^{-1}v\right\Vert =a_{\bar{P}_{o}}(g)^{2\rho_{o}}.
\]
Thus if $n\in\lbrack N_{o},N_{o}],g\in\omega,X\in Z$ then
\[
a_{\bar{P}_{o}}(n\exp Xg)^{2\rho_{o}}=\left\Vert g^{-1}\exp(-X)n^{-1}%
v\right\Vert .
\]
The inequality
\[
\left\Vert n^{-1}v\right\Vert =\left\Vert \exp Xgg^{-1}\exp(-X)n^{-1}%
v\right\Vert \leq\left\Vert \exp Xg\right\Vert \left\Vert g^{-1}\exp
(-X)n^{-1}v\right\Vert
\]
implies that%
\[
\left\Vert g^{-1}\exp(-X)n^{-1}v\right\Vert \geq\left\Vert \exp Xg\right\Vert
^{-1}\left\Vert n^{-1}v\right\Vert
\]
so
\[
a_{\bar{P}_{o}}(n\exp Xg)^{-\rho_{o}}\leq\left\Vert \exp Xg\right\Vert
^{\frac{1}{2}}a_{\bar{P}_{o}}(n)^{-\rho_{o}}\leq\left\Vert \exp X\right\Vert
^{\frac{1}{2}}\left\Vert g\right\Vert ^{\frac{1}{2}}a_{\bar{P}_{o}}%
(n)^{-\rho_{o}}.
\]
Noting that $\left\Vert \exp X\right\Vert ^{2}$ is a polynomial in $X$ the
Corollary follows.
\end{proof}

We note that if $P$ is a standard parabolic subgroup then

\begin{lemma}
\label{simpleL}There exist $k$ and $C>0$ such that $\left\Vert ^{o}m_{\bar{P}%
}(n)\right\Vert \leq C\left\Vert n\right\Vert ^{k}$ for $n\in N_{P}$.
\end{lemma}

\begin{proof}
$n=\bar{n}_{\bar{P}}(n)a_{\bar{P}}(n){}^{o}m_{\bar{P}}(n)k_{\bar{P}}(n)$ so we
have
\[
\left\Vert n\right\Vert =\left\Vert \bar{n}_{\bar{P}}(n)a_{\bar{P}}(n){}%
^{o}m_{\bar{P}}(n)\right\Vert \geq\left\Vert a_{\bar{P}}(n){}^{o}m_{\bar{P}%
}(n)\right\Vert \geq\left\Vert a_{\bar{P}}(n)\right\Vert ^{-1}\left\Vert
^{o}m_{\bar{P}}(n)\right\Vert
\]
thus
\[
\left\Vert a_{\bar{P}}(n)\right\Vert \left\Vert n\right\Vert \geq\left\Vert
^{o}m_{\bar{P}}(n)\right\Vert .
\]
Note that
\[
\left\Vert a_{\bar{P}}(n)\right\Vert \leq C_{1}\left\Vert n\right\Vert ^{r}%
\]
for some $C_{1}>0$ and $r>0$. Take $k=r+1$ and $C=C_{1}$.
\end{proof}

\begin{lemma}
Let $\left(  \pi,H_{\pi}\right)  $ be an admissible Hilbert representation of
$G$ and let $\chi$ be a generic character of $N_{o}$. Then there exists a
finite subset $S_{\chi,\pi}=\{x_{1},x_{2},...,x_{r}\}\subset U(Lie(G))$ such
that if $H_{\chi,\pi}$ is the Hilbert space completion of $(H_{\pi}^{\infty
})_{K}$ with respect to the inner product
\[
\left\langle v,w\right\rangle _{\chi,\pi}=\sum_{i=1}^{r}\left\langle
x_{i}v,x_{i}w\right\rangle _{\pi}+\left\langle v,w\right\rangle _{\pi}%
\]
and if $\lambda\in Wh_{\chi}(H_{\pi}^{\infty})$ then $\lambda$ extends to a
continuous functional on $H_{\chi,\pi}$. Furthermore, $H_{\chi,\pi}$ is
$\pi(G)$ invariant and $(\pi,H_{\chi,\pi})$ is a Hilbert representation of $G$.
\end{lemma}

\begin{proof}
If $\mu$ is a continuous functional on $H_{\pi}^{\infty}$ then by the
definition of the Fr\'{e}chet space structure on $H^{\infty}$ there exist
$u_{i}\in U(\mathfrak{g}),$ $i=1,...,d$ and a constant $C$ such that
\[
|\mu(v)|\leq C\sum_{i=1}^{d}\left\Vert u_{i}v\right\Vert \leq\sqrt{d}%
C\sum_{i=1}^{d}\left\langle u_{i}v,u_{i}v\right\rangle
\]
for $v\in H^{\infty}$. Let $\lambda_{1},...,\lambda_{m}$ be a basis of
$Wh_{\chi}(H_{\pi}^{\infty})$. Then for each $i$ choose a finite set
$S_{i}\subset U(Lie(G))$ such that
\[
\left\vert \lambda_{i}(v)\right\vert \leq C_{i}\sum_{x\in S_{i}}\left\langle
xv,xw\right\rangle
\]
for $v\in H^{\infty}$. Take $S_{\chi,\pi}=\cup S_{i}$.
\end{proof}

Let $\chi$ be a generic character of $N_{o}$ and let $P$ be a standard
parabolic subgroup of $G$ and\ let $\left(  \sigma,H_{\sigma}\right)  $ be an
irreducible square integrable representation of {}$^{o}M_{\bar{P}}$. We
consider the integral if $u\in I_{\sigma}^{\infty}$
\[
j_{\sigma,\mu}(u)=\int_{N_{P}}\chi(n)^{-1}u_{\mu}(n)dn.
\]
Here and in the rest of this section if $P$ and $\sigma$ are understood then
we write $u_{\mu}$ for $_{\bar{P}}u_{\sigma,\mu}.$

We will now apply the above lemmas with $G$ replaced by $^{o}M_{P}$ and $\pi$
replaced by $\sigma$ a square integrable representation of $^{o}M_{P}$. To
simplify notation we will use the notation $\chi^{\ast}$ for $\chi_{|N_{o}\cap
M_{P}}$.

\begin{lemma}
\label{key-estimate}Let $\left\langle ...,...\right\rangle _{\chi^{\ast
},\sigma}$ be as in the previous lemma for $\sigma$. There exist constants
$a,b,C>0$ such that such that if $\mu\in\left(  \mathfrak{a}_{P}\right)
_{\mathbb{C}}^{\ast}$ and $\operatorname{Re}(\mu,\alpha)<-t(\rho_{P},\alpha)$
with $t>0$ for for $\alpha\in\Phi(P,A_{P})$ then%
\[
\left\Vert u_{\mu}(n)\right\Vert _{\chi^{\ast},\sigma}\leq C\left\Vert
n\right\Vert ^{-a(1+t)}\left\Vert n\right\Vert ^{b}%
\]
in particular this implies that there exists $c_{\sigma}$ such that the
integral
\[
\int_{N_{P}}\left\Vert u_{\mu}(n)\right\Vert _{\chi^{\ast},\sigma}dn
\]
converges absolutely and uniformly in compacta of the set
\[
\{\mu|\operatorname{Re}(\mu,\alpha)<-c_{\sigma}(\rho_{P},\alpha),\alpha\in
\Phi(P,A_{P})\}
\]
for $u\in I_{\sigma}^{\infty}$.
\end{lemma}

\begin{proof}
We have for $n\in N_{P}$
\[
u_{\mu}(n)=a_{\bar{P}}(n)^{-\rho+\mu}\sigma({}^{o}m_{\bar{P}}(n))u(k_{\bar{P}%
}(n)),
\]
thus
\[
\left\Vert u_{\mu}(n)\right\Vert _{\chi^{\ast},\sigma}\leq a_{\bar{P}%
}(n)^{-\rho+\operatorname{Re}\mu}\left\Vert \sigma({}^{o}m_{\bar{P}%
}(n))\right\Vert _{\chi^{\ast}.\sigma}\left\Vert u(k_{\bar{P}}(n))\right\Vert
_{\chi^{\ast},\sigma}.
\]
Thus
\[
\left\Vert u_{\mu}(n)\right\Vert _{\chi^{\ast},\sigma}\leq C_{2}a_{\bar{P}%
}(n)^{-\rho+\operatorname{Re}\mu}\left\Vert \sigma({}^{o}m_{\bar{P}%
}(n))\right\Vert _{\chi^{\ast}.\sigma}\leq C_{2}a_{\bar{P}}(n)^{-\rho
+\operatorname{Re}\mu}\left\Vert ^{o}m_{\bar{P}}(n))\right\Vert ^{s}%
\]%
\[
\leq C_{2}C_{3}a_{\bar{P}}(n)^{-\rho+\operatorname{Re}\mu}\left\Vert
n\right\Vert ^{sr}.
\]
by Lemma \ref{simpleL} and the fact that there exist $C$ and $c$ such that
$\left\Vert \sigma({}m))\right\Vert _{\chi^{\ast}.\sigma}\leq C\left\Vert
m\right\Vert ^{c}$. Harish-Chandra's Lemma (see Lemma \ref{HCLemma}) implies
that
\[
1+\rho(\log a_{\bar{P}}(n))\geq C_{4}(1+\log\left\Vert n\right\Vert ).
\]
This implies that there exists $C_{5},r$ such that
\[
a_{\bar{P}}(n)^{\rho}\geq C_{5}\left\Vert n\right\Vert ^{p}.
\]
Thus, if $\operatorname{Re}(\mu,\alpha)<-t(\rho_{P},\alpha)$ then applying the
first part of Lemma \ref{HCinequality}
\[
a_{\bar{P}}(n)^{-\rho+\operatorname{Re}\mu}\leq a_{\bar{P}}(n)^{-(1+t)\rho
}\leq C_{5}^{-(1+t)}\left\Vert n\right\Vert ^{-p(1+t)}.
\]
Hence
\[
\left\Vert u_{\mu}(n)\right\Vert _{\chi^{\ast},\sigma}\leq C_{6}\left\Vert
n\right\Vert ^{-p(1+t)}\left\Vert n\right\Vert ^{sr}.
\]
Since there exists $q$ such that
\[
\int_{N_{P}}\left\Vert n\right\Vert ^{-q}dn<\infty\text{.}%
\]
The lemma follows.
\end{proof}

In light of this lemma we see that $j_{\sigma,\mu}(u)$ is holomorphic in $\mu$
with values in $H_{\chi^{\ast},\sigma}$ on the set
\[
\{\mu\in\mathfrak{a}_{\mathbb{C}}^{\ast}|\operatorname{Re}(\mu,\alpha
)<-c_{\sigma}(\rho_{P},\alpha)\mathrm{\ for\ }\alpha\in\Phi(P,A_{P})\}.
\]

Define $(H_{\sigma}^{\infty})_{\chi^{\ast}}$ to be the closure of the space
\[
\{(\sigma(n)-\chi(n))u|u\in H_{\sigma}^{\infty},n\in N_{o}\cap M_{P}\}.
\]
in $H_{\sigma}^{\infty}$. Also set $\left(  H_{\chi^{\ast},\sigma}\right)
_{\chi^{\ast}}$ equal to the closure of
\[
\{(\sigma(n)-\chi(n))u|u\in H_{\chi^{\ast},\sigma},n\in N_{o}\cap M_{P}\}.
\]
in $H_{1}$. Then $(H_{\sigma}^{\infty})_{\chi^{\ast}}\subset\left(
H_{\chi^{\ast},\sigma}\right)  _{\chi^{\ast}}$ so we have a natural continuous
map $\iota:H_{\sigma}^{\infty}/(H_{\sigma}^{\infty})_{\chi^{\ast}}\rightarrow
H_{\chi^{\ast},\sigma}/\left(  H_{\chi^{\ast},\sigma}\right)  _{\chi^{\ast}}$.

\begin{lemma}
The space $H_{\sigma}^{\infty}/(H_{\sigma}^{\infty})_{\chi^{\ast}}$ is finite
dimensional and $\iota$ is bijective.
\end{lemma}

\begin{proof}
We note that $H_{\sigma}^{\infty}/(H_{\sigma}^{\infty})_{\chi^{\ast}}$ is a
Fr\'{e}chet space and $H_{1}/\left(  H_{1}\right)  _{\chi^{\ast}}$ is a
Hilbert space so in both cases the continuous duals separate the points. By
the definition of $H_{\chi^{\ast},\sigma}$ the inclusion of $H_{\sigma
}^{\infty}$ into $H_{\chi^{\ast},\sigma}$ is continuous and its image is
dense. Also, if $\lambda\in\left(  H_{\chi^{\ast},\sigma}/\left(
H_{\chi^{\ast},\sigma}\right)  _{\chi^{\ast}}\right)  ^{\prime}$ then
$\lambda$ pulls back to $\check{\lambda}$ on $H_{\chi^{\ast},\sigma}$ as an
element of $Wh_{\chi^{\ast}}(H_{\chi^{\ast},\sigma})$. By the definition of
$H_{\chi^{\ast},\sigma}$, the element $\check{\lambda}_{|H_{\sigma}^{\infty}}$
\ is in $Wh_{\chi^{\ast}}(H_{\sigma}^{\infty})$. The result now follows from
the definition of $H_{\chi^{\ast},\sigma}$ and the finite dimensionality of
the space of Whittaker vectors.
\end{proof}

We will identify $\left(  H_{\sigma}^{\infty}/(H_{\sigma}^{\infty}%
)_{\chi^{\ast}}\right)  ^{\prime}$ with $Wh_{\chi^{\ast}}(H_{\sigma}^{\infty
})$ and thus write $\lambda$ \ for $\check{\lambda}$.

Let $p$ be the canonical projection of $H_{\chi^{\ast},\sigma}$ onto
$H_{\chi^{\ast},\sigma}/\left(  H_{\chi^{\ast},\sigma}\right)  _{\chi^{\ast}}$
and let $\tau$ be the inverse map to $\iota$. We define for $(\mu
,\alpha)<-c_{\sigma}(\rho,\alpha),\alpha\in\Phi(P,A_{P})$,
\[
\tilde{j}_{\sigma,\mu}(u)=\tau(p(j_{\sigma,\mu}(u)).
\]
Then for each $u\in H_{\sigma}^{\infty}$ this defines a holomorphic map in
$\mu$ of
\[
\{\mu\in\left(  \mathfrak{a}_{P}\right)  _{\mathbb{C}}^{\ast}|(\mu
,\alpha)<-c_{\sigma}(\rho,\alpha),\alpha\in\Phi(P,A_{P})\}
\]
to $H_{\sigma}^{\infty}/(H_{\sigma}^{\infty})_{\chi^{\ast}}$.

\begin{proposition}
\label{holomorphy}If $u\in H_{\sigma}^{\infty}$ then the map $\mu\mapsto
\tilde{j}_{\sigma,\mu}(u)$ initially defined on
\[
\{\mu\in\left(  \mathfrak{a}_{P}\right)  _{\mathbb{C}}^{\ast}%
|\operatorname{Re}(\mu,\alpha)<-c_{\sigma}(\rho,\alpha),\alpha\in\Phi
(P,A_{P})\}
\]
extends to a holomorphic map of $\left(  \mathfrak{a}_{P}\right)
_{\mathbb{C}}^{\ast}$ to $H_{\sigma}^{\infty}/(H_{\sigma}^{\infty}%
)_{\chi^{\ast}}$. Furthermore, the map
\[
\left(  \mathfrak{a}_{P}\right)  _{\mathbb{C}}^{\ast}\times H_{\sigma}%
^{\infty}\rightarrow H_{\sigma}^{\infty}/(H_{\sigma}^{\infty})_{\chi^{\ast}}%
\]
given by
\[
\mu,u\mapsto\tilde{j}_{\sigma,\mu}(u)
\]
is continuos and holomorphic in $\mu$.
\end{proposition}

\begin{proof}
If $\mu\in\left(  \mathfrak{a}_{P}\right)  _{\mathbb{C}}^{\ast}$ ,
$\operatorname{Re}(\mu,\alpha)<-c_{\sigma}(\rho,\alpha),\alpha\in\Phi
(P,A_{P})$, and if $\lambda\in Wh_{\chi}(H_{1})$ then $\lambda(j_{\sigma,\mu
}(u))=J(P,\sigma,\mu)(\lambda)(u).$ The observation that $\lambda_{|\left(
H_{1}\right)  _{\chi^{\ast}}}=0$ combined with the above lemmas now imply that
$\lambda(\tilde{j}_{\sigma,\mu}(u))=J(P,\sigma,\mu)(\lambda)(u)$. Let
$w_{1},...,w_{r}$ be a basis of $H_{\sigma}^{\infty}/(H_{\sigma}^{\infty
})_{\chi^{\ast}}$ and let $\lambda_{1},...,\lambda_{r}$ be the dual basis in
$Wh_{\chi^{\ast}}(H_{\sigma}^{\infty})$ then
\[
\tilde{j}_{\sigma,\mu}(u)=\sum_{i=1}^{r}J(P,\sigma,\mu)(\lambda_{i})(u)w_{i}.
\]
This formula implements the holomorphic continuation and proves the indicted
properties of it.
\end{proof}

\begin{theorem}
\label{key-Theorem}Let $\alpha\in\mathcal{S(}\mathfrak{a}_{P}^{\ast})$,
$\beta(\nu)=\mu(\sigma,\nu)\alpha(\nu)$ (recall that $\mu(\sigma,\nu)$ is
the\ Harish-Chandra Plancherel density) and $v,w\in\left(  I_{\sigma}^{\infty
}\right)  _{K}$ then
\[
\int_{N_{o}}\chi(n_{o})^{-1}\int_{\mathfrak{a}_{P}^{\ast}}\left\langle
\pi_{i\nu}(n_{o}g)v,w\right\rangle \beta(\nu)d\nu dn_{o}%
\]%
\[
=\int_{\mathfrak{a}_{P}^{\ast}}J(P,\sigma,i\nu)(\lambda_{\tilde{j}%
_{\sigma,i\nu}(w)})(\pi_{i\nu}(g)v))\beta(\nu)d\nu.
\]

\end{theorem}

\begin{proof}
Fix $g\in G$. We are computing
\[
\int_{N_{o}}\chi(n_{o})^{-1}\int_{\mathfrak{a}_{P}^{\ast}}\int_{N_{P}%
}\left\langle v_{i\nu}(nn_{o}g),w_{i\nu}(n)\right\rangle dn\beta(\nu)d\nu
dn_{o}.
\]
We will do the calculation indirectly. Corollary \ref{polyest} and Corollary
\ref{N-extimate} in the Appendix imply that there exist $d$ and $C_{g,v,w}$
such that for all $z\in\mathbb{C}$ $\ $with $\operatorname{Re}z\geq0$ if $X\in
Z$ ($Z$ as in Corollary \ref{polyest}) then
\[
\left\vert \int_{\lbrack N_{o},N_{o}]}\int_{\mathfrak{a}_{P}^{\ast}}%
\int_{N_{P}}\left\langle v_{i\nu-z\rho_{P}}(nn_{o}\exp Xg),w_{i\nu-\bar{z}%
\rho_{P}}(n)\right\rangle dn\beta(\nu)d\nu dn_{o}\right\vert
\]%
\[
\leq\left\vert \int_{\mathfrak{a}_{P}^{\ast}}\beta(\nu)d\nu\right\vert
C_{g,v,w}^{1+\operatorname{Re}z}(1+\left\Vert X\right\Vert )^{d}.
\]
We fix $v,w,g$ and $\alpha$ and choose $C>0$ such that $C_{g,v,w}%
^{1+s}\left\vert \int_{\mathfrak{a}_{P}^{\ast}}\beta(\nu)d\nu\right\vert \leq
C^{1+s}$. Set for $\operatorname{Re}z\geq0$ and $X\in Z$
\[
\phi_{z}(X)=\int_{[N_{o},N_{o}]}\int_{\mathfrak{a}_{P}^{\ast}}\int_{N_{P}%
}\left\langle v_{i\nu-z\rho_{P}}(nn_{o}\exp Xg),w_{i\nu-\bar{z}\rho_{P}%
}(n)\right\rangle dn\beta(\nu)d\nu dn_{o}%
\]
Here $\phi_{z}(X)$ is holomorphic in $z$ for $\operatorname{Re}z>0$ and
continuous for $\operatorname{Re}z\geq0$. \ In the course of the argument we
will use the following assertion

I) There exists a constant $r_{\sigma}$ such that if $\operatorname{Re}%
z>r_{\sigma} $ then%
\[
\int_{N_{o}}\int_{N_{P}}\left\vert \left\langle v_{i\nu-z\rho_{P}}%
(nn_{o}g),w_{i\nu-\bar{z}\rho_{P}}(n)\right\rangle \right\vert dn<\infty.
\]
Thus Fubini's theorem implies that $\phi_{z}\in L^{1}(Z)$ for
$\operatorname{Re}z>r_{\sigma}.$

We prove this starting with a standard argument. Since the integrand is
non-negative the convergence can be tested in either order so we look at
\[
\int_{N_{P}}\int_{N_{o}}\left\vert \left\langle v_{i\nu-z\rho_{P}}%
(nn_{o}g),w_{i\nu-\bar{z}\rho_{P}}(n)\right\rangle \right\vert dn_{o}dn
\]%
\[
=\int_{N_{P}}\int_{N_{o}}\left\vert \left\langle v_{i\nu-z\rho_{P}}%
(n_{o}g),w_{i\nu-\bar{z}\rho_{P}}(n)\right\rangle \right\vert dn_{o}dn.
\]
Writing $n_{o}=n_{M}n_{1}$ with $n_{M}\in M_{P}\cap N_{o}$ and $n_{1}\in
N_{P}$ we can normalize the measures such that the last integral is equal to%
\[
\int_{N_{P}}\int_{N_{P}}\int_{N_{M}}\left\vert \left\langle v_{i\nu-z\rho_{P}%
}(n_{M}n_{1}g),w_{i\nu-\bar{z}\rho_{P}}(n)\right\rangle \right\vert
dn_{M}dn_{1}dn
\]%
\[
=\int_{N_{P}}\int_{N_{P}}\int_{N_{M}}\left\vert \left\langle \sigma
(n_{M})v_{i\nu-z\rho_{P}}(n_{1}g),w_{i\nu-\bar{z}\rho_{P}}(n)\right\rangle
\right\vert dn_{M}dn_{1}dn.
\]
Theorem 15.2.4 in \cite{RRGII} implies that there exist with $q_{1}$ and
$q_{2}$ continuous seminorms on $H_{\sigma}^{\infty}$ \ such that
\[
\overset{}{(\ast)}\int_{N_{M}}\left\vert \left\langle \sigma(n_{M}%
)x,y\right\rangle \right\vert dn_{M}\leq q_{1}(x)q_{2}(y).
\]
Indeed, Theorem 15.2.4 in \cite{RRGII} asserts that if $\Lambda$ is defined as
in 4.3.5 in \cite{RRGI} for $\left(  H_{\sigma}\right)  _{K}$ then there exist
continuous semi-norms $\phi_{1},\phi_{2}$ on $H_{\sigma}^{\infty}$ and $d$
such that if $x,y\in H_{\sigma}^{\infty}$ and $a\in{}^{o}M_{P}\cap A_{o}^{+}$
(here $\mathfrak{a}_{o}^{+}$ is as in Section 2 and $A_{o}^{+}=\exp
\mathfrak{a}_{o}^{+}$) then
\[
\left\vert \left\langle \sigma(a)x,y\right\rangle \right\vert \leq
a^{-\Lambda}(1+\log\left\Vert a\right\Vert )^{d}\phi_{1}(x)\phi_{2}(x).
\]
Noting that $\sigma$ is square integrable Theorem 5.5.4 \cite{RRGI} which says
that a square integrable representation of a real reductive group with compact
center is rapidly decreasing as defined 5.1.1 \cite{RRGI} and this implies
that there exists $\varepsilon>0$ such that if $h\in{}^{o}\mathfrak{m\cap
}\mathfrak{a}_{o}^{+}$ then%
\[
\Lambda(h)\geq(1+\varepsilon)\rho_{o}(h).
\]
The upshot of this is that if $a\in{}^{o}M_{P}\cap A_{o}^{+}$ then
\[
\left\vert \left\langle \sigma(a)x,y\right\rangle \right\vert \leq
a^{-(1+\varepsilon)\rho}(1+\log\left\Vert a\right\Vert )^{d}\phi_{1}%
(x)\phi_{2}(x).
\]
Now if $m\in$ $^{o}M_{P}$ then $m=k_{1}ak_{2}$ with $k_{i}\in M_{P}\cap K$ and
$a\in\overline{{}^{o}M_{P}\cap A_{o}^{+}}$ thus%
\[
\left\vert \left\langle \sigma(m)x,y\right\rangle \right\vert =\left\vert
\left\langle \sigma(a)\sigma(k_{2})x,\sigma(k_{1}^{-1})y\right\rangle
\right\vert \leq a^{-(1+\varepsilon)\rho}(1+\log\left\Vert a\right\Vert
)^{d}\phi_{1}(\sigma(k_{1}^{-1})x)\phi_{2}(\sigma(k_{2})x).
\]
This implies that if $q_{i}(x)=\mathrm{\sup}\{\phi_{i}(\sigma(k)x)|k\in
M_{P}\cap K$ we have
\[
\left\vert \left\langle \sigma(m)x,y\right\rangle \right\vert \leq\left\Vert
m\right\Vert ^{-(1+\varepsilon)}(1+\log\left\Vert m\right\Vert )^{d}%
q_{1}(x)q_{2}(x).
\]
Since%
\[
\int_{N_{M}}\left\Vert m\right\Vert ^{-(1+\varepsilon)}(1+\log\left\Vert
m\right\Vert )^{d}dm<\infty
\]
this implies $(\ast)$ above. Thus the integral that we are studying satisfies%
\[
\leq\int_{N_{P}}\int_{N_{P}}a_{\bar{P}}(n_{1}g)^{-(1+\operatorname{Re}%
z)\rho_{P}}a_{\bar{P}}(n)^{-(1+\operatorname{Re}z)\rho_{P}}\times
\]%
\[
q_{1}(\sigma(m_{\bar{P}}(n_{1}g))v(k(n_{1}g)))q_{2}(\sigma(m_{\bar{P}%
}(n))w(k(n)))dn_{1}dn
\]%
\[
\leq\int_{N_{P}}\int_{N_{P}}a_{\bar{P}}(n_{1}g)^{-(1+\operatorname{Re}%
z)\rho_{P}}a_{\bar{P}}(n)^{-(1+\operatorname{Re}z)\rho_{P}}\times
\]%
\[
\left\Vert m_{\bar{P}}(n_{1}g)\right\Vert ^{s}\left\Vert m_{\bar{P}%
}(n))\right\Vert ^{t}\max_{k_{1},k_{2}\in K}q_{3}(v(k_{1}))q_{4}%
(w(k_{2}))dn_{1}dn.
\]
Here we have used the fact that since $H_{\sigma}^{\infty}$ is of moderate
growth (see \cite{RRGII} 11.5.1 p.84) there exist $q_{3},q_{4}$ continuous
seminorms on $H_{\sigma}^{\infty}$ and $s,t$ such that%
\[
q_{1}(\sigma(m)x)\leq\left\Vert m\right\Vert ^{s}q_{3}(x),q_{2}(\sigma
(m)x)\leq\left\Vert m\right\Vert ^{t}q_{4}(x).
\]
Applying Proposition \ref{key-estimate} the assertion in I) now follows.

We will now begin the proof of the theorem.

If $\psi\in\mathcal{S}(Z)$ then%
\[
\left\vert \phi_{z}(X)\psi(X)\right\vert \leq C^{1+\operatorname{Re}%
z}(1+\left\Vert X\right\Vert )^{d}\left\vert \psi(X)\right\vert .
\]
with $C$ (given above) independent of $\operatorname{Re}z\geq0$ and $\psi$
thus in particular $\phi_{z}\psi\in L^{1}(Z)$ if $\operatorname{Re}z\geq0$.
This allows us to define the tempered distributions
\[
T_{z}(\psi)=\int_{Z}\phi_{z}(X)\psi(X)dX.
\]
for $\operatorname{Re}z\geq0$. Dominated convergence implies that $T_{z}$ is
weakly continuous for $\operatorname{Re}z\geq0$ and weakly holomorphic for
$\operatorname{Re}z>0$. In particular,
\[
\lim_{\varepsilon\rightarrow0+}T_{\varepsilon}=T_{0}%
\]
weakly in $\mathcal{S}(Z)$. Recalling that if $\mathcal{F}$ denotes the
Fourier transform on $\mathcal{S}(Z)$, which we will write as
\[
\mathcal{F}(\psi)(Y)=\int_{Z}e^{iB(\theta Y,X)}\psi(X)dX),
\]
(here $dX$ is normalized so that $\mathcal{F}^{-1}(\psi)(X)=\mathcal{F}%
(\psi)(-X)$) and using the usual formula for the Fourier transform of a
tempered distribution as $\mathcal{F}(T)=T\circ\mathcal{F}$. The continuity of
the Fourier transform implies that
\[
\lim_{\varepsilon\rightarrow0+}\mathcal{F}(T_{\varepsilon})=\mathcal{F}%
(T_{0})
\]
We also note that Fubini's theorem combined with I) above (in this proof) and
the fact that $\mathcal{C}(G)_{|N_{o}}\subset L^{1}(N_{o})$ implies that if
$\phi_{z}\in L^{1}(Z)$ for $\operatorname{Re}z=0$ or $\operatorname{Re}%
z>r_{\sigma}$ then
\[
\mathcal{F}(T_{z})(\psi)=\int_{Z}\mathcal{F}(\phi_{z})(X)\psi(X)dX\text{.}%
\]
and $\mathcal{F}(\phi_{z})$ is a continuous bounded function on $Z$ in this set.

Define for $Y\in Z$ and $X\in\mathfrak{n}_{o}$
\[
\chi_{Y}(\exp(X))=e^{-iB(\theta Y,X)}.
\]
If $\operatorname{Re}z\geq0$, $Y\in Z$ then consider,%
\[
h_{z}(Y)=\int_{N_{o}}\chi_{Y}(n_{o})^{-1}\int_{\mathfrak{a}_{P}^{\ast}}%
\int_{N_{P}}\left\langle v_{i\nu-z\rho_{P}}(nn_{o}g),w_{i\nu-\bar{z}\rho_{P}%
}(n)\right\rangle dn\beta(\nu)d\nu dn_{o}%
\]
then $h_{0}$ is defined and continuous for $z=0$ since since $\mathcal{C}%
(G)_{|N_{o}}\in N_{o}$ also if $\operatorname{Re}z>r_{\sigma}$ then I) above
implies that $h_{z}$ is the inverse Fourier transform of $\phi_{z}$ so is
continuous on $Z$ and holomorphic for $\operatorname{Re}z>r_{\sigma}$ and
$h_{z}(Y)$ holomorphic in $z$ for $\operatorname{Re}z>r_{\sigma}$ also if
$z=0$ or $\operatorname{Re}z>r_{\sigma}$ then
\[
\mathcal{F}(T_{z})(\psi)=\int_{Z}h_{z}(Y)\psi(Y)dY.
\]
Let $c_{\sigma}$ be as in Proposition \ref{key-estimate}. Let $U$ be the set
of all $Y\in Z$ such that $\chi_{Y}$ is a generic character, the $U$ is open
and dense in $Z$. We fix $Y\in U$ and write $\chi=\chi_{Y}$ and will now
compute
\[
h_{z}(Y)=\int_{N_{o}}\chi(n_{o})^{-1}\int_{\mathfrak{a}_{P}^{\ast}}\int
_{N_{P}}\left\langle v_{i\nu-z\rho_{P}}(nn_{o}g),w_{i\nu-z\rho_{P}%
}(n)\right\rangle dn\beta(\nu)d\nu dn_{o}%
\]
under the assumption that $\operatorname{Re}z>\max\{r_{\sigma},c_{\sigma}\}$
where this total integral is absolutely convergent. We may thus integrate in
any order. So, fix $z$ with $\operatorname{Re}z>\max\{r_{\sigma},c_{\sigma}\}$
and $Y\in U$. Set $\chi=\chi_{Y}$. We will use the notation $N_{M}=N_{o}\cap
M_{P}$ and note that $N_{P}N_{M}=N_{P_{o}}$ consider
\[
h_{z}(Y)=\int_{N_{P}\times N_{M}}\chi(n_{P}n_{M})^{-1}\times
\]%
\[
\int_{\mathfrak{a}_{P}^{\ast}}\int_{N_{P}}\left\langle \sigma(n_{M}%
)v_{i\nu-z\rho_{P}}(n_{M}^{-1}nn_{P}n_{M}g),w_{i\nu-\bar{z}\rho_{P}%
}(n)\right\rangle dn\beta(\nu)d\nu dn_{P}dn_{M}.
\]
We first integrate over $N_{P}$ and use the substitution $n_{P}\rightarrow
n^{-1}n_{P}$ to find that
\[
h_{z}(Y)=\int_{N_{P}\times N_{M}}\chi(n_{P}n_{M})^{-1}\times
\]%
\[
\int_{\mathfrak{a}_{P}^{\ast}}\int_{N_{P}}\chi(n)\left\langle \sigma
(n_{M})v_{i\nu-z\rho_{P}}(n_{M}^{-1}n_{P}n_{M}g),w_{i\nu-\bar{z}\rho_{P}%
}(n)\right\rangle dn\beta(\nu)d\nu dn_{P}dn_{M},
\]%
\[
=\int_{N_{P}\times N_{M}}\chi(n_{P}n_{M})^{-1}\times
\]%
\[
\int_{\mathfrak{a}_{P}^{\ast}}\int_{N_{P}}\chi(n)\left\langle \sigma
(n_{M})v_{i\nu-z\rho_{P}}(n_{P}g),w_{i\nu-\bar{z}\rho_{P}}(n)\right\rangle
dn\beta(\nu)d\nu dn_{P}dn_{M}%
\]%
\[
=\int_{\mathfrak{a}_{P}^{\ast}}\lambda_{j_{\sigma,i\nu-\bar{z}\rho}%
(w)}(j_{\sigma,i\nu-z\rho}(\pi_{i\nu}(g)v))\beta(\nu)d\nu
\]%
\[
=\int_{\mathfrak{a}_{P}^{\ast}}\int_{\mathfrak{a}_{P}^{\ast}}\overline
{\lambda_{j_{\sigma,i\nu-z\rho}(v)}(j_{\sigma,i\nu-z\rho}(\pi_{i\nu-\bar
{z}\rho}(g)w))}\beta(\nu)d\nu
\]
This implies that the value only depends on the image of $j_{\sigma,i\nu
-z\rho}(v)$ (resp. $j_{\sigma,i\nu-z\rho}(w)$) in the finite dimensional
space
\[
H_{\sigma}^{\infty}/\left(  H_{\sigma}^{\infty}\right)  _{\chi_{|N_{o}\cap
M_{P}}}=H_{\chi_{|N_{o}\cap M_{P}},\sigma}/\left(  H_{\chi_{|N_{o}\cap M_{P}%
},\sigma}\right)  _{\chi_{|N_{o}\cap M_{P}}}%
\]
If $\operatorname{Re}z>\max\{r_{\sigma},c_{\sigma}\},$ using the notation
$J_{\chi_{Y}}(P,\sigma,\nu)$ and $\tilde{j}_{\sigma,\nu}^{Y}$ to take into
account the dependence on $Y$ and noting that both are holomorphic in $\nu$
and continuous in $Y$ (see Proposition \ref{Continuity} in Appendix B) we
have
\[
h_{z}(Y)=\int_{\mathfrak{a}_{P}^{\ast}}J_{\chi_{Y}}(P,\sigma,i\nu
-z\rho)(\lambda_{\tilde{j}_{\sigma,i\nu-z\rho}^{Y}(w)})(\pi_{i\nu}%
(g)v))\beta(\nu)d\nu
\]
if $Y\in U$. \ This gives a holomorphic continuation of $h_{z}$ to
$\operatorname{Re}z>0$. Let $\psi$ have compact support in $U$ then we have
the family of distributions on $U$ (see Theorem\ref{J-estimate})
\[
S_{z}(\psi)=\int_{U}\int_{\mathfrak{a}_{P}^{\ast}}J_{\chi_{Y}}(P,\sigma
,i\nu-z\rho)(\lambda_{\tilde{j}_{\sigma,i\nu-z\rho}^{Y}(w)})(\pi_{i\nu
}(g)v))\beta(\nu)d\nu\psi(Y)dY
\]
continuous weakly for $\operatorname{Re}z\geq0$ and weakly holomorphic on
$\operatorname{Re}z>0$. Also since the original distributions, $T_{z}$, have
the same holomorphy and continuity propertied and we have proved that
$\mathcal{F}\left(  T_{z}\right)  _{|U}=S_{z|U}$ for $\operatorname{Re}%
z>\max\{r_{\sigma},c_{\sigma}\}$ The holomorphy for $\operatorname{Re}z>0$ and
continuity for $\operatorname{Re}z\geq0$ imply that the equation is true for
$\operatorname{Re}z\geq0$. If $\psi\in C_{c}^{\infty}(U)$
\[
\int_{U}h_{0}(Y)\psi(Y)dY=\mathcal{F}(T_{0})(\psi)=\lim_{\varepsilon
->0+}\mathcal{F}(T_{\varepsilon})(\psi)=S_{0}(\psi).
\]
Thus, if $Y\in U$ then%
\[
h_{0}(Y)=\int_{\mathfrak{a}_{P}^{\ast}}J_{\chi_{Y}}(P,\sigma,i\nu
)(\lambda_{\tilde{j}_{\sigma,i\nu}^{Y}(w)})(\pi_{i\nu}(g)v))\beta(\nu)d\nu.
\]
Which is the statement of the theorem.
\end{proof}

In the proof of Proposition \ref{holomorphy} Let $\lambda_{1},...,\lambda_{r}
$ be a basis of $Wh_{\chi_{|M_{P}\cap N_{o}}}(H_{\sigma}^{\infty})$ and let
$w_{1},...,w_{r}\in H_{\sigma}^{\infty} $ project to $H_{\sigma}^{\infty
}/\left(  H_{\sigma}^{\infty}\right)  _{\chi_{|N_{o}\cap M_{P}}}$ and satisfy
$\lambda_{i}(w_{j})=\delta_{ij},i,j=1,...,r$. If $w\in I_{\sigma}^{\infty}$
then
\[
\tilde{j}_{\sigma,i\nu}(w)\equiv J(P,\sigma,i\nu)(\lambda_{i})(w)w_{i}%
\quad\operatorname{mod}\left(  H_{\sigma}^{\infty}\right)  _{\chi_{|N_{o}\cap
M_{P}}}.
\]
Hence
\[
\lambda_{\tilde{j}_{\sigma,i\nu}(w)}=\sum\overline{J_{\chi}(P,\sigma
,i\nu)(\lambda_{i})(w)}\lambda_{w_{_{i}}}.
\]
We also note that if $\lambda_{1},...,\lambda_{r}$ is an orthonormal basis
relative to the inner product in Corollary \ref{whitproduct} then
$\lambda_{w_{i}}=\lambda_{i}$. We have proved

\begin{corollary}
\label{basic-Formula} If $\alpha\in\mathcal{S(}\mathfrak{a}_{P}^{\ast})$ and
$v,w\in\left(  I_{\sigma}^{\infty}\right)  _{K}$ setting
\[
\psi(g)=\psi(\alpha,v,w)(g)=\int_{\mathfrak{a}_{P}^{\ast}}\left\langle
\pi_{i\nu}(g)v,w\right\rangle \alpha(\nu)\mu(\sigma,i\nu)d\nu
\]
then if $\lambda_{1}^{\sigma},...,\lambda_{r_{\sigma}}^{\sigma}$ is an
orthonormal basis of $Wh_{\chi_{|M_{P}\cap N_{o}}}(H_{\sigma}^{\infty})$ then
(recalling that
\[
\psi_{\chi}(g)=\int_{N_{o}}\chi(n)^{-1}\psi(ng)dn)
\]%
\[
\psi_{\chi}(g)=\sum_{i=1}^{d_{\sigma}}\int_{\mathfrak{a}_{P}^{\ast}}%
\overline{J_{\chi}(P,\sigma,i\nu)(\lambda_{i}^{\sigma})(w)}J_{\chi}%
(P,\sigma,i\nu)(\lambda_{i}^{\sigma})(\pi_{\bar{P},\sigma,i\nu}(g)v)\alpha
(\nu)\mu(\sigma,i\nu)d\nu.
\]

\end{corollary}

\section{ \label{Firstform}The Whittaker Plancherel Theorem first form}

We assume that $G$ has compact center and that $\chi$ is generic.

If $F\subset\hat{K}$ is a finite set then define $\mathcal{C}(G)_{F}$ to be
the set of elements of $\mathcal{C}(G)$ such that
\[
\sum_{\gamma\in F}d(\gamma)\int_{K}f(gk)\chi_{\gamma}(k^{-1})dk=f(g),g\in G
\]
and let $_{F}\mathcal{C}(G)$ be those elements of $\mathcal{C}(G)$ such that
\[
\sum_{\gamma\in F}d(\gamma)\int_{K}f(kg)\chi_{\gamma}(k)dk=f(g),g\in G
\]
Let $P$ be a standard cuspidal parabolic subgroup and let $\sigma
\in\mathcal{E}_{2}\left(  ^{o}M_{P}\right)  $. Let $E_{\gamma}$ be the
orthogonal projection onto $I_{\sigma}(\gamma)$ (the $\gamma$--isotypic
component). We set $\pi_{\nu}=\pi_{\bar{P},\sigma,\nu}$. If $F\subset\hat{K}$
\ then we write $E_{F}=\sum_{\gamma\in F}E_{\gamma}$. If $f\in\mathcal{C}%
(G)_{F}$ then if $f\in{}_{H}\mathcal{C}(G)\cap\mathcal{C}(G)_{F}$ then
\[
\pi_{\nu}(f)=E_{H}\pi_{\nu}(f)E_{F}%
\]
Thus if $\left\{  v_{i}^{\sigma}\right\}  $ is an orthonormal basis of
$I_{\sigma}$ such that $v_{i}^{\sigma}\in I_{\sigma}(\gamma_{i})$ and if
$S_{\sigma}(F)=\{i|\gamma_{i}\in F\}$ then $\left\{  v_{i}^{\sigma}\right\}
_{i\in S_{\sigma}}$ is an orthonormal basis of $I_{\sigma}(F)=\sum_{\gamma\in
F}I_{\sigma}(\gamma)$. Before we state our key result we need a bit more
notation. If we need to indicate the dependence of $J(P,\sigma,\nu)$ on $\chi$
we will write $J_{\chi}(P,\sigma,\nu)$. If $f\in\mathcal{C}(G)$ then set
\[
T_{P,\sigma}(f)(g)=\int_{\mathfrak{a}_{P}^{\ast}}\mathrm{tr}(\pi_{\bar
{P},\sigma,i\nu}(L_{g^{-1}}f))\mu(\sigma,i\nu)d\nu.
\]

\begin{theorem}
\label{whittakerI}Let $P$ be a standard cuspidal parabolic subgroup and let
$\sigma\in\mathcal{E}_{2}({}^{o}M_{P})$. Let $\{\lambda_{1}^{\sigma
},...,\lambda_{d_{\sigma}}^{\sigma}\}$ be an orthonormal basis of
$Wh(H_{\sigma}^{\infty})$ relative to the inner product given in Corollary 37.
Let $F$ be a finite subset of $\hat{K}$, let $f\in\mathcal{C}(G)_{F}$ and let
$\{v_{i}^{\sigma}\}$ be as above an orthonormal basis of $\left(  I_{\sigma
}\right)  _{K}$ compatible with the isotypic decomposition then
\[
T_{P,\sigma}(f)_{\chi}(g)=\sum_{i=1}^{d_{\sigma}}\sum_{l\in S_{\sigma}(F)}%
\int_{\mathfrak{a}_{P}^{\ast}}J_{\chi^{-1}}(P,\sigma,i\nu)(\lambda_{i}%
^{\sigma})(\pi_{\bar{P},\sigma,i\nu}(f)v_{l})\times
\]%
\[
\overline{J_{\chi^{-1}}(P,\sigma,i\nu)(\lambda_{i}^{\sigma})(\pi_{\bar
{P},\sigma i\nu}(g)v_{l})}\mu(\sigma,i\nu)d\nu.
\]

\end{theorem}

\begin{proof}
Both sides of the equation are continuous in $f\in\mathcal{C}(G)_{F}.$ Since
the left $K$--finite elements of $\mathcal{C}(G)_{F}$ are dense in
$\mathcal{C}(G)_{F}$ it is enough to prove the equation for $f\in{}%
_{H}\mathcal{C}(G)\cap\mathcal{C}(G)_{F}$ any finite $H\subset\hat{K}$. By
definition
\[
T_{P,\sigma}(f)(g)=\sum_{l\in S_{\sigma}(F)}\int_{\mathfrak{a}_{P}^{\ast}%
}\left\langle \pi_{\bar{P},\sigma,i\nu}(g^{-1})\pi_{\bar{P},\sigma,i\nu
}(f)v_{l},v_{l}\right\rangle \mu(\sigma,i\nu)d\nu
\]
note that this is a finite sum and is equal to $\psi$ is as in Corollary
\ref{basic-Formula}. In the calculation below we will use the notation
$\alpha_{ij}(\nu)=\left\langle v_{j},\pi_{\bar{P},\sigma,i\nu}(f)v_{l}%
\right\rangle .$
\[
\sum_{l\in S_{\sigma}(F),j\in S_{\sigma}(H)}\int_{\mathfrak{a}_{P}^{\ast}%
}\left\langle \pi_{\bar{P},\sigma,i\nu}(g^{-1})v_{j},v_{l}\right\rangle
\left\langle \pi_{\bar{P},\sigma,i\nu}(f)v_{l},v_{j}\right\rangle \mu
(\sigma,i\nu)d\nu
\]%
\[
=\sum_{l\in S_{\sigma}(F),j\in S_{\sigma}(H)}\int_{\mathfrak{a}_{P}^{\ast}%
}\left\langle v_{j},\pi_{\bar{P},\sigma,i\nu}(g)v_{l}\right\rangle
\left\langle \pi_{\bar{P},\sigma,i\nu}(f)v_{l},v_{j}\right\rangle \mu
(\sigma,i\nu)d\nu
\]%
\[
=\sum_{l\in S_{\sigma}(F),j\in S_{\sigma}(H)}\overline{\psi(\alpha_{ij}%
,v_{l},v_{j})(g)}.
\]
Which implies that
\[
T_{P,\sigma}(f)_{\chi}(g)=\sum_{l\in S_{\sigma}(F),j\in S_{\sigma}%
(H)}\overline{\psi(\alpha_{ij},v_{l},v_{j})_{\chi^{-1}}(g)}%
\]%
\[
=\sum_{i=1}^{d_{\sigma}}\sum_{j\in S_{\sigma}(F)}\sum_{l\in S_{\sigma}(F)}%
\int_{\mathfrak{a}_{P}^{\ast}}J_{\chi^{-1}}(P,\sigma,i\nu)(\lambda_{i}%
^{\sigma})(v_{j})\times
\]%
\[
\overline{J_{\chi^{-1}}(P,\sigma,i\nu)(\lambda_{i}^{\sigma})(\pi_{i\nu
}(g)v_{l})}\left\langle \pi_{\bar{P},\sigma,i\nu}(f)v_{l},v_{j}\right\rangle
\mu(\sigma,i\nu)d\nu
\]%
\[
=\sum_{i=1}^{d_{\sigma}}\sum_{l\in S_{\sigma}(F)}\int_{\mathfrak{a}_{P}^{\ast
}}J_{\chi^{-1}}(P,\sigma,i\nu)(\lambda_{i}^{\sigma})(\pi_{\bar{P},\sigma,i\nu
}(f)v_{l})\overline{J_{\chi^{-1}}(P,\sigma,i\nu)(\lambda_{i}^{\sigma}%
)(\pi_{i\nu}(g)v_{l})}\mu(\sigma,i\nu)d\nu.
\]

\end{proof}

Let $P$ be a $P_{o}$ standard parabolic subgroup of $G$ with standard
Langlands decomposition $M_{P}A_{P}N_{P}$, $\sigma$ an irreducible square
integrable representation of $M_{p}$ and and $\left\{  \lambda_{1}^{\sigma
},...,\lambda_{r_{\sigma}}^{\sigma}\right\}  $ an orthonormal basis of
$Wh_{\chi|_{M_{p}\cap N_{o}}}$. If $f\in\mathcal{C}(G)$ is right $K$--finite
with isotopic components in $F$\ we define
\[
W_{P,\sigma}(f)(g)=\sum_{i=1}^{d_{\sigma}}\sum_{l\in S_{\sigma}(F)}%
\int_{\mathfrak{a}_{P}^{\ast}}J_{\chi^{-1}}(P,\sigma,i\nu)(\lambda_{i}%
^{\sigma})(\pi_{\bar{P},\sigma,i\nu}(f)v_{l})\times
\]%
\[
\overline{J_{\chi^{-1}}(P,\sigma,i\nu)(\lambda_{i}^{\sigma})(\pi_{i\nu
}(g)v_{l})}\mu(\sigma,i\nu)d\nu.
\]
The following is the distributional version of the Whittaker Plancherel Theorem

\begin{corollary}
\label{Whitt2}If $f$ is a right $K$--finite element of $\mathcal{C}(G)$ then
\[
f_{\chi}=\sum_{[P]\in\mathcal{P}}\sum_{[\sigma]\in\mathcal{E}_{2}({}^{o}%
M_{P})}d(\sigma)W_{P,\sigma}(f).
\]
Note that this is a finite sum since an $^{o}M_{P}\cap K$--type can appear in
only a finite number of elements of $\mathcal{E}_{2}({}^{o}M_{P})$.
\end{corollary}

\begin{proof}
Harish-Chandra's Plancherel Theorem (Theorem \ref{plancherel}) implies that in
the notation of the previous theorem
\[
f=\sum_{[P]\in\mathcal{P}}\sum_{[\sigma]\in\mathcal{E}_{2}({}^{o}M_{P}%
)}d(\sigma)T_{P,\sigma}(f).
\]
Thus
\[
f_{\chi}=\sum_{[P]\in\mathcal{P}}\sum_{[\sigma]\in\mathcal{E}_{2}({}^{o}%
M_{P})}d(\sigma)T_{P,\sigma}(f)_{\chi}=
\]
\[
\sum_{\lbrack P]\in\mathcal{P}}\sum_{[\sigma]\in\mathcal{E}_{2}({}^{o}M_{P}%
)}d(\sigma)W_{P,\sigma}(f).
\]
by the previous theorem.
\end{proof}

\section{\label{Final}Whittaker Plancherel Theorem}

In this section we will use Theorem \ref{Whitt2} to derive the spectral
decomposition of a right $K$--finite element of $\mathcal{C}(N_{o}\backslash
G;\chi)$ in the case when $G$ has compact center and $\chi$ is generic. The
key lemma in the proof is

\begin{lemma}
Assume $\chi$ is generic. Let $P$ be a standard cuspidal parabolic subgroup
and let $\sigma\in\omega\in\mathcal{E}_{2}({}^{o}M_{P})$. Let $\psi\in
C^{\infty}(N_{o}\backslash G;\chi)$ be such that $\left\vert \psi\right\vert
\in C_{c}(N_{o}\backslash G)$ and let $\varphi\in C_{c}^{\infty}(N_{o})$ is
such that
\[
\int_{N_{o}}\chi(n)^{-1}\varphi(n)dn=1.
\]
Set $f(nak)=\varphi(n)\psi(ak)$ for $n\in N_{o},a\in A_{o},k\in K$. If
$\lambda\in Wh_{\chi^{-1}|_{N_{o}\cap M_{P}}}(H_{\sigma}^{\infty})$ and $v\in
I_{\sigma}^{\infty}$ then
\[
J_{\chi^{-1}}(P,\sigma,i\nu)(\lambda)(\pi_{\bar{P},\sigma,i\nu}(f)v)=\int
_{N_{o}\backslash G}J_{\chi^{-1}}(P,\sigma,i\nu)(\lambda)(\pi_{\bar{P}%
,\sigma,i\nu}(g)v)\psi(g)dg
\]

\end{lemma}

\begin{proof}
Note that $\nu\longmapsto J_{\chi^{-1}}(P,\sigma,\nu)(\lambda)(\pi_{\bar
{P},\sigma,\nu}(f)v)$ is holomorphic on $\left(  \mathfrak{a}_{P}^{\ast
}\right)  _{\mathbb{C}}$. Thus if we prove that
\[
J_{\chi^{-1}}(P,\sigma,\nu)(\pi_{\bar{P},\sigma,\nu}(f)v)=\int_{N_{o}%
\backslash G}J_{\chi^{-1}}(P,\sigma,\nu)(\pi_{\bar{P},\sigma,\nu}%
(g)v)\psi(g)dg
\]
if $\operatorname{Re}(\nu,\alpha)<0$ for $\alpha\in\Phi(P,A_{P})$ then the
formula will be true for all $\nu$ since both sides of the equation are
holomorphic in $\nu$. We calculate (setting $w_{\nu}={}_{\bar{P}}w_{\nu}$ for
$w\in I_{\sigma}^{\infty})$
\[
J_{\chi^{-1}}(P,\sigma,\nu)(\pi_{\bar{P},\sigma,\nu}(f)v)=\int_{N_{P}}%
\chi(n)\lambda(\left(  \pi_{\bar{P},\sigma,\nu}(f)v\right)  _{\nu}(n))dn
\]%
\[
=\int_{N_{P}}\chi(n)\lambda(\int_{G}f(g)v_{\nu}(ng)dg)dn.
\]
Note that $N_{o}=N_{M}N_{P}$ with $N_{M}=N_{o}\cap M_{P}\ $and the map
$N_{M}\times N_{P}\rightarrow N_{o}$ given by multiplication is a
diffeomorphism. Theorem \ref{discretewhitt} says that there exists
$z\in\left(  H_{\sigma}\right)  _{K}$such that if $u\in H_{\sigma}^{\infty}$
then
\[
\lambda(u)=\lambda_{z}(u)=\int_{N_{M}}\left\langle \sigma(n_{M}%
)u,z\right\rangle \chi(n_{M})dn_{M}.
\]
Thus
\[
\int_{N_{P}}\chi(n)\lambda(\int_{G}f(g)v_{\nu}(ng))dg)dn=
\]%
\[
\int_{N_{P}}\chi(n)\int_{N_{M}}\chi(n_{M})\left\langle \sigma(n_{M})(\int
_{G}f(g)v_{\nu}(ng)dg),z\right\rangle dn_{M}dn
\]%
\[
=\int_{N_{P}}\chi(n)\int_{N_{M}}\chi(n_{M})\times
\]%
\[
\left\langle \sigma(n_{M})(\int_{N_{o}\times A_{o}\times K}f(n_{o}%
a_{o}k)v_{\nu}(nn_{o}a_{o}k)a_{o}^{-2\rho}dn_{o}da_{o}dk),z\right\rangle
dn_{M}dn
\]%
\[
=\int_{N_{P}}\int_{N_{M}}\chi(nn_{M})\times
\]%
\[
\left\langle \int_{N_{o}\times A_{o}\times K}f(n_{o}a_{o}k)v_{\nu}(n_{M}%
nn_{o}a_{o}k)a_{o}^{-2\rho}dn_{o}da_{o}dk),z\right\rangle dn_{M}dn.
\]
Note that under the condition on $\nu$ and $\psi$ this integral converges
absolutely so it can be calculated in any order. So we are calculating
\[
\int_{N_{o}}\chi(n_{1})\left\langle \int_{N_{o}\times A_{o}\times K}%
\varphi(n_{o})\psi(a_{o}k)v_{\nu}(n_{1}n_{o}a_{o}k),z\right\rangle
a_{o}^{-2\rho}dn_{o}da_{o}dk)dn_{1}%
\]%
\[
=\int_{N_{o}}\chi(n_{1}n_{o}^{-1})\left\langle \int_{N_{o}\times A_{o}\times
K}\varphi(n_{o})\psi(a_{o}k)v_{\nu}(n_{1}a_{o}k)),z\right\rangle a_{o}%
^{-2\rho}dn_{o}da_{o}dk)dn_{1}%
\]%
\[
=\int_{N_{o}}\chi(n_{o})^{-1}\varphi(n_{o})dn_{o}\times
\]%
\[
\int_{N_{o}\times A_{o}\times K}\chi(n_{1})\left\langle \psi(a_{o}k)v_{\nu
}(n_{1}a_{o}k),z\right\rangle a_{o}^{-2\rho}da_{o}dkdn_{1}%
\]%
\[
=\int_{N_{M}\times N_{P}}\chi(n_{P})\psi(a_{o}k)\left\langle \sigma
(n_{M})v_{\nu}(n_{P}a_{o}k)),z\right\rangle dn_{M}dn_{P}a^{-2\rho}da_{o}dk
\]%
\[
=\int_{A_{o}\times K}\int_{N_{P}}\chi(n_{P})\psi(a_{o}k)\lambda(v_{\nu}%
(n_{P}a_{o}k))dn_{P}a^{-2\rho}da_{o}dk=
\]%
\[
=\int_{A_{o}\times K}J_{\chi^{-1}}(P,\sigma,i\nu)(\lambda)(\pi_{\bar{P}%
,\sigma.\nu}(a_{o}k)v)\psi(a_{o}k)a_{o}^{-2\rho}da_{o}dk
\]%
\[
=\int_{N_{o}\backslash G}J_{\chi^{-1}}(P,\sigma,i\nu)(\lambda)(\pi_{\bar
{P},\sigma.\nu}(g)v)\psi(g)dg.
\]

\end{proof}

Let $P$ be a standard cuspidal parabolic subgroup and let $[\sigma
]\in\mathcal{E}_{2}({}^{o}M_{P})$. If $\psi\in\mathcal{C}(N_{o}\backslash
G;\chi)$ and $\lambda\in Wh_{\chi^{-1}|_{N_{o}\cap M_{P}}}(H_{\sigma}^{\infty
}),v\in I_{\sigma}^{\infty}$ then set
\[
\mathcal{W}_{P,\sigma}^{\chi}(\psi)(\lambda,v,\nu)=\int_{N_{o}\backslash
G}J_{\chi^{-1}}(P,\sigma,i\nu)(\lambda)(\pi_{\bar{P},\sigma.i\nu}%
(g)v)\psi(g)dg.
\]

Before we can state the Whittaker Plancherel Theorem we need more notation. If
$P$ is a standard Parabolic subgroup of $G$ and $\sigma\in\mathcal{E}_{2}%
({}^{o}M_{P})$ choose $\lambda_{1}^{\sigma},...,\lambda_{r_{\sigma}}^{\sigma}$
an orthonormal basis of $Wh_{\chi^{-1}|_{N_{o}\cap M_{P}}}(H_{\sigma}^{\infty
})$ and $\left\{  v_{i}^{\sigma}\right\}  $ an orthonormal basis of $\left(
I_{\sigma}^{\infty}\right)  _{K}$ that respects the $K$--isotypic decomposition.

\begin{theorem}
Assume that $\chi$ is generic. If $\psi$ is a right $K$--finite element of
$\mathcal{C}(N_{o}\backslash G;\chi)\ $then
\[
\psi(g)=\sum_{[P]\in\mathcal{P}}\sum_{[\sigma]\in\mathcal{E}_{2}({}^{o}M_{P}%
)}d(\sigma)\sum_{i=1}^{r_{\sigma}}\sum_{l}\int_{\mathfrak{a}_{P}^{\ast}%
}\mathcal{W}_{P,\sigma}^{\chi}(\psi)(\lambda_{i}^{\sigma},v_{l}^{\sigma}%
,\nu))
\]%
\[
\overline{J_{\chi^{-1}}(P,\sigma,i\nu)(\lambda_{i}^{\sigma})(\pi_{\bar
{P},\sigma,i\nu}(g)v_{l})}\mu(\sigma,i\nu)d\nu.
\]

\end{theorem}

\begin{proof}
First we assume that $\left\vert \psi\right\vert \in C_{c}(N_{o}\backslash G)$
and let $f$ be as in the previous lemma. Then, since, $f_{\chi}=\psi$ Theorem
\ref{Whitt2} implies that
\[
\psi=f_{\chi}=\sum_{[P]\in\mathcal{P}}\sum_{[\sigma]\in\mathcal{E}_{2}({}%
^{o}M_{P})}d(\sigma)W_{P,\sigma}(f).
\]
Now
\[
W_{P,\sigma}(f)(g)=\sum_{i=1}^{d_{\sigma}}\sum_{l}d(\sigma)\int_{\mathfrak{a}%
_{P}^{\ast}}J_{\chi^{-1}}(P,\sigma,i\nu)(\lambda_{i}^{\sigma})(\pi_{\bar
{P},\sigma,i\nu}(f)v_{l}^{\sigma})\times
\]%
\[
\overline{J_{\chi^{-1}}(P,\sigma,i\nu)(\lambda_{i}^{\sigma})(\pi_{\bar
{P},\sigma,i\nu}(g)v_{l}^{\sigma})}\mu(\sigma,i\nu)d\nu.
\]
But the formula is a finite sum depending only on the $K$--types that occur in
$\mathrm{Span}R_{K}\psi$ and the terms in the sum are continuous on
$\mathcal{C}(N_{o}\backslash G;\chi)_{K}$ so the formula is true on all of
$\mathcal{C}(N_{o}\backslash G;\chi)_{K}$. The theorem now follows from the
definition of $\mathcal{W}_{P,\sigma}^{\chi}(\psi)(\lambda,v,\nu)$.
\end{proof}

\begin{corollary}
\label{12-version}Assume the $\chi$ is generic. If $\psi_{1},\psi_{2}\in$
$\mathcal{C}(N_{o}\backslash G;\chi)$ are $K$--finite elements then%
\[
\int_{N_{o}\backslash G}\psi_{1}(g)\overline{\psi_{2}(g)}dg=
\]%
\[
\sum_{\lbrack P]\in\mathcal{P}}\sum_{[\sigma]\in\mathcal{E}_{2}({}^{o}M_{P}%
)}d(\sigma)\sum_{i=1}^{r_{\sigma}}\sum_{l}\int_{\mathfrak{a}_{P}^{\ast}%
}\mathcal{W}_{P,\sigma}^{\chi}(\psi_{1})(\lambda_{i}^{\sigma},v_{l}^{\sigma
},\nu))\overline{\mathcal{W}_{P,\sigma}^{\chi}(\psi_{2})(\lambda_{i}^{\sigma
},v_{l}^{\sigma},\nu)}\mu(\sigma,i\nu)d\nu.
\]

\end{corollary}

\begin{proof}
The preceding theorem implies that since all sums are finite
\[
\int_{N_{o}\backslash G}\psi_{1}(g)\overline{\psi_{2}(g)}dg=
\]%
\[
\psi(g)=\sum_{[P]\in\mathcal{P}}\sum_{[\sigma]\in\mathcal{E}_{2}({}^{o}M_{P}%
)}d(\sigma)\sum_{i=1}^{r_{\sigma}}\sum_{l}\int_{N_{o}\backslash G}%
\int_{\mathfrak{a}_{P}^{\ast}}\mathcal{W}_{P,\sigma}^{\chi}(\psi_{1}%
)(\lambda_{i}^{\sigma},v_{l}^{\sigma},\nu))\times
\]%
\[
\overline{J_{\chi^{-1}}(P,\sigma,i\nu)(\lambda_{i}^{\sigma})(\pi_{\bar
{P},\sigma,i\nu}(g)v_{l})}\mu(\sigma,i\nu)d\nu\overline{\psi_{2}(g)}dg.
\]
As in the proof of the above theorem, since both sides of the formula that we
are proving are continuous in $\psi_{1}$ and $\psi_{2}$ we may assume that
$\left\vert \psi_{1}\right\vert \in C_{c}(N_{o}\backslash G)$, so we can
interchange the integrals and get%
\[
\sum_{\lbrack P]\in\mathcal{P}}\sum_{[\sigma]\in\mathcal{E}_{2}({}^{o}M_{P}%
)}d(\sigma)\sum_{i=1}^{r_{\sigma}}\sum_{l}\int_{\mathfrak{a}_{P}^{\ast}%
}\mathcal{W}_{P,\sigma}^{\chi}(\psi_{1})(\lambda_{i}^{\sigma},v_{l}^{\sigma
},\nu))\times
\]%
\[
\int_{N_{o}\backslash G}\overline{J_{\chi^{-1}}(P,\sigma,i\nu)(\lambda
_{i}^{\sigma})(\pi_{\bar{P},\sigma,i\nu}(g)v_{l}\psi_{2}(g)}dg\mu(\sigma
,i\nu)d\nu.
\]
Which is our desired formula.
\end{proof}

\appendix
%dummy comment inserted by tex2lyx to ensure that this paragraph is not empty

\section{Tempered estimates}

In this section we collect some estimates that will be needed in the proofs of
the main results. The first result recalls a technique of \cite{Cowling-et-al}%
. This result was also used by Sun in \cite{Sun} which contains versions of
the first few results of this appendix. We include full details of the
argument in \cite{Cowling-et-al} and our version of the results in \cite{Sun}
since we need the independence of parameters in Corollary
\ref{Main-Tempered-Estimate}. The rest of the appendix contains some technical
results needed in the body of the paper. Let $G$ be a real reductive group and
let $(R,L^{2}(G))$ denote the right regular representation of $G$ on
$L^{2}(G)$

\begin{theorem}
Let $f,h\subset L^{2}(G)$ be continuous and be respectively in the
$K$--isotypic components for $\gamma,\mu\in\hat{K}$ relative to $R$. Then
\[
\left\vert \left\langle R(g)f,h\right\rangle \right\vert \leq d(\gamma
)d(\mu)\Xi(g)\left\Vert f\right\Vert \left\Vert g\right\Vert
\]
here the norm is the $L^{2}$--norm
\end{theorem}

\begin{proof}
If $p\in C(G)$ then set
\[
\tilde{p}(g)=\max_{k\in K}\left\vert p(gk)\right\vert .
\]

1. If in addition $\int_{K}p(gk)\chi_{\gamma}(k^{-1})dk=d(\gamma)p(g)$ for all
$g\in G$ then
\[
\left\vert \tilde{p}(g)\right\vert \leq d(\gamma)\left(  \int_{K}\left\vert
p(gk)\right\vert ^{2}dk\right)  ^{\frac{1}{2}}.
\]

Indeed,
\[
p(gk)=d(\gamma)\int_{K}p(gku)\overline{\chi_{\gamma}(u)}du.
\]
so
\[
\left\vert p(gk)\right\vert \leq d(\gamma)\left(  \int_{K}\left\vert
p(gku)\right\vert ^{2}du\right)  ^{\frac{1}{2}}\left(  \int_{K}\left\vert
\chi_{\gamma}(u)\right\vert ^{2}du\right)  ^{\frac{1}{2}.}.
\]
This clearly implies the assertion.

2. If $p\in L^{2}(G)$ is as in 1. then
\[
\left\Vert \tilde{p}\right\Vert \leq d(\gamma)\left\Vert p\right\Vert .
\]

To see this integrate the squares of the two sides of the inequality in 1.

3. If $f,g$ are as above then
\[
\left\vert \left\langle R(g)f,h\right\rangle \right\vert \leq\left\langle
R(g)\tilde{f},\tilde{h}\right\rangle .
\]
This is obvious.

We are left with showing that if $\alpha,\beta\in C(G/K)\cap L^{2}(G)$ are
positive functions then
\[
\left\langle R(g)\alpha,\beta\right\rangle \leq\Xi(g)\left\Vert \alpha
\right\Vert \left\Vert \beta\right\Vert .
\]
Let $P_{o}$ be a minimal parabolic subgroup of $G$ then if $\delta$ is the
modular function of $P_{o}$ and if $u\in L^{1}(G)$ then
\[
\int_{G}u(x)dx=\int_{K}\int_{P_{o}}\delta(p)^{-1}u(pk)dpdk
\]
(here $dp$ is an appropriate choice of right invariant measure on $P_{o}$).
Thus we have
\[
\left\langle R(g)\alpha,\beta\right\rangle =\int_{K}\int_{P_{o}}\delta
(p)^{-1}\alpha(pkg)\beta(p)dpdk
\]
\[
\leq\int_{K}\left(  \int_{P_{o}}\delta(p)^{-1}\alpha(pkg)^{2}dp\right)
^{\frac{1}{2}}dk\left\Vert \beta\right\Vert .
\]
If $x\in G$ then write $x=p(x)k(x)$ then
\[
\int_{K}\left(  \int_{P_{o}}\delta(p)^{-1}\alpha(pkg)^{2}dp\right)  ^{\frac
{1}{2}}dk=\int_{K}\left(  \int_{P_{o}}\delta(p)^{-1}\alpha(pp(kg)k(kg))^{2}%
dp\right)  ^{\frac{1}{2}}dk
\]
\[
=\int_{K}\delta(p(kg))^{\frac{1}{2}}dk\left(  \int_{P_{o}}\delta(p)^{-1}%
\alpha(p)^{2}dp\right)  ^{\frac{1}{2}}=\Xi(g)\left\Vert \alpha\right\Vert .
\]

\end{proof}

\begin{theorem}
(compare \cite{Sun}) \label{Sun}Let $(\pi,H)$ be an irreducible square
integrable representation of $G$ then there exists a continuous semi-norm, $q
$, on $H^{\infty}$such that
\[
\left\vert \left\langle \pi(g)v,w\right\rangle \right\vert \leq\Xi(g)q(v)q(w)
\]
for all $g\in G$.
\end{theorem}

\begin{proof}
Let $u\in H$ be a unit vector. If $T(z)(g)=\sqrt{d(\pi)}\left\langle
\pi(g)z,u\right\rangle $ then the Schur orthogonality relations imply that $T
$ is an isometry hence
\[
\left\vert \left\langle \pi(g)v,w\right\rangle \right\vert =\left\vert
\left\langle R(g)T(v),T(w)\right\rangle \right\vert
\]
\[
\leq d(\gamma)d(\tau)\left\Vert T(v)\right\Vert \left\Vert T(w)\right\Vert
\Xi(g)=d(\gamma)d(\tau)\left\Vert v\right\Vert \left\Vert w\right\Vert \Xi(g)
\]
Thus if $v,w\in H^{\infty}$ and $v=\sum_{\gamma\in\hat{K}}v_{\gamma}$ and
$w=\sum_{\gamma\in\hat{K}}w_{\gamma}$ then
\[
\left\vert \left\langle \pi(g)v,w\right\rangle \right\vert =\left\vert
\sum_{\gamma,\tau}\left\langle \pi(g)v_{\gamma},w_{\tau}\right\rangle
\right\vert \leq\sum_{\gamma,\tau}\left\vert \left\langle \pi(g)v_{\gamma
},w_{\tau}\right\rangle \right\vert
\]
\[
\leq\Xi(g)\sum_{\gamma,\tau\in\hat{K}}d(\gamma)d(\tau)\left\Vert v_{\gamma
}\right\Vert \left\Vert w_{\tau}\right\Vert .
\]
We note that if $C_{K}$ is the Casimir operator of $K$ \ and if $\lambda
_{\gamma}$is the eigenvalue of $C_{K}$on corresponding to $\gamma$ and if
$v\in H^{\infty}$ then
\[
\left\Vert d\pi(\left(  I+C_{K}\right)  ^{r})v\right\Vert ^{2}=\sum
(1+\lambda_{\gamma})^{r}\left\Vert v_{\gamma}\right\Vert ^{2}.
\]
This implies that
\[
\left\Vert v_{\gamma}\right\Vert \leq\left(  \lambda_{\gamma}+1\right)
^{-\frac{r}{2}}\left\Vert (\left(  I+C_{K}\right)  ^{r})v\right\Vert
\]
so
\[
\sum_{\gamma\in\hat{K}}d(\gamma)\left\Vert v_{\gamma}\right\Vert
\leq\left\Vert d\pi(\left(  I+C_{K}\right)  ^{r})v\right\Vert \sum_{\gamma
\in\hat{K}}d(\gamma)(1+\lambda_{\gamma})^{-\frac{r}{2}}.
\]
If $r$ is sufficiently large then the series converges. So define
\[
q(v)=\sum_{\gamma\in\hat{K}}d(\gamma)\left\Vert v_{\gamma}\right\Vert .
\]

\end{proof}

\begin{corollary}
\label{Main-Tempered-Estimate} Let $P$ be a standard parabolic subgroup, if
$\sigma$ is a square integrable representation of $^{o}M_{P}$ and if $v,w$ are
continuous maps of $K$ to $H_{\sigma}^{\infty}$ that are elements of
$\mathrm{Ind}_{M_{P}\cap K}^{K}(\sigma_{|_{M_{P}\cap K}})$ \ then
\[
\left\vert \left\langle \pi_{\bar{P},\sigma,i\nu}(g)v,w\right\rangle
\right\vert \leq\Xi(g)r(v)r(w)
\]
with $r(v)=\max_{k\in K}q_{M_{P}}(v(k))$. Note that $r$ depends on $\sigma$
and not on $\nu$.
\end{corollary}

\begin{proof}%
\[
\left\langle \pi_{i\nu}(g)v,w\right\rangle =\int_{K}\left\langle v_{i\nu
}(kg),w(k)\right\rangle dk
\]
\[
=\int_{K}a_{\bar{P}}(kg)^{i\nu-\rho}\left\langle \sigma
(m(kg))v(k(kg)),w(k)\right\rangle dk.
\]
So
\[
\left\vert \left\langle \pi_{i\nu}(g)v,w\right\rangle \right\vert \leq\int
_{K}a_{\bar{P}}(kg)^{-\rho_{P}}|\left\langle \sigma
(m(kg))v(k(kg)),w(k)\right\rangle |dk
\]
\[
\leq\int_{K}a_{\bar{P}}(kg)^{-\rho_{P}}\Xi_{^{o}M_{P}}%
(m(kg))q(v(k(gk))q(w(k))dk
\]
\[
\leq\int_{K}a_{\bar{P}}(kg)^{-\rho_{P}}\Xi_{^{o}M_{P}}(m(kg))dkr(v)r(w).
\]
We note that
\[
\int_{K}a_{\bar{P}}(kg)^{-\rho_{P}}\Xi_{^{o}M_{P}}(m(kg))dk=\Xi(g).
\]
Indeed
\[
\int_{K}a_{\bar{P}}(kg)^{-\rho_{P}}\Xi_{^{o}M_{P}}(m(kg))dk
\]
\[
=\int_{K}a_{\bar{P}}(kg)^{-\rho_{P}}\int_{K\cap M}a_{M\cap\bar{P}_{o}}%
(k_{1}m(kg))^{-\rho_{M\cap\bar{P}_{o}}}dk_{1}dk.
\]
Reversing the order of integration and noting that $k_{1}m(kg)=m(k_{1}kg)$
and
\[
a_{\bar{P}}(g)^{-\rho_{\bar{P}}}a_{\bar{P}_{o}\cap\cap{}^{o}M}(m(g))^{-\rho
_{M\cap\bar{P}_{o}}}=a_{\bar{P}_{o}}(g)^{-\rho_{o}}%
\]
yields
\[
\int_{K}a_{\bar{P}}(kg)^{-\rho_{P}}a_{^{o}M\cap\bar{P}_{o}}(m(kg))^{-\rho
_{M\cap\bar{P}_{o}}}dk=\Xi(g).
\]
The corollary follows.
\end{proof}

Let $P$ be a standard parabolic subgroup and let $\bar{P}$ be the standard
opposite parabolic subgroup to $P$. Let $\bar{N}$ be the unipotent radical of
$\bar{P}$ and let $\bar{P}={}^{o}M_{P}A_{P}\bar{N}_{P}$ be its $K$--standard
Langlands decomposition. write for $g\in G,$
\[
g=\bar{n}(g)a_{\bar{P}}(g)m_{\bar{P}}(g)k_{\bar{P}}(g)
\]
$\bar{n}(g)$ $\in\bar{N}_{P},a_{\bar{P}}(g)\in A_{P}$, $m_{\bar{P}}(g)\in
{}^{o}M_{P}$, $k_{\bar{P}}(g)\in K$ with the usual ambiguity in $m_{\bar{P}}$
and $k_{\bar{P}.}$.

\begin{lemma}
\label{Simple-Estimate}Let $\omega\subset G$ be a compact subset then there
exists a positive constant $C_{\omega}$ such that if $g\in\omega$ then
\[
C_{\omega}^{-1}a_{\bar{P}}(x)^{\rho_{P}}\leq a_{\bar{P}}(xg)^{\rho_{P}}\leq
C_{\omega}a_{\bar{P}}(x)^{\rho_{P}}.
\]

\end{lemma}

\begin{proof}
Let $r=\dim\mathfrak{\bar{n}}_{P}$ consider the $G$ module $V=\wedge
^{r}\mathfrak{g}$. Let $\left\Vert ...\right\Vert _{V}$ denote the norm on $V$
and the operator norm on $End(V)$ corresponding to extension of the inner
product $\left\langle X,Y\right\rangle =-B(X,\theta Y)$, $X,Y\in\mathfrak{g} $
to $\wedge^{r}\mathfrak{g}$. If $\xi$ is a unit vector in $\wedge
^{r}\mathfrak{\bar{n}}_{P}$ and $g\in G$ then
\[
g^{-1}\xi=a_{\bar{P}}(g)^{2\rho_{P}}k_{\bar{P}}(g)^{-1}\xi.
\]
Thus
\[
\left\Vert g^{-1}\xi\right\Vert _{V}=a_{\bar{P}}(g)^{2\rho_{P}}.
\]
Thus
\[
a_{\bar{P}}(xg)^{2\rho_{P}}=\left\Vert g^{-1}x^{-1}\xi\right\Vert
\leq\left\Vert g\right\Vert _{V}\left\Vert x^{-1}\xi\right\Vert _{V}%
=\left\Vert g\right\Vert _{V}a_{\bar{P}}(x)^{2\rho_{P}},
\]
Also
\[
\left\Vert x^{-1}\xi\right\Vert _{V}=\left\Vert gg^{-1}x^{-1}v\right\Vert
_{V}\leq\left\Vert g\right\Vert _{V}\left\Vert g^{-1}x^{=1}\xi\right\Vert _{V}%
\]
so
\[
a_{\bar{P}}(xg)^{2\rho_{P}}\geq\left\Vert g\right\Vert _{V}^{-1}a_{\bar{P}%
}(x)^{2\rho_{P}}.
\]
Let $C_{\omega}=\max_{g\in\omega}\left\Vert g\right\Vert _{V}^{\frac{1}{2}}$.
\end{proof}

\begin{proposition}
\label{First N-estimate}Notation as in the previous corollary. If $v,w\in
I_{\sigma}^{\infty}$, if $\omega$ is a compact subset of $G$ and if
$g\in\omega,n_{o}\in N_{o}$ and $\operatorname{Re}z\geq0$ then there exist
constants $C_{\omega}^{\prime},C_{\omega}^{\prime\prime}$ such that
\[
\left\vert \int_{N_{P}}\left\langle v_{i\nu-z\rho_{P}}(nn_{o}g),w_{i\nu
-\bar{z}\rho_{P}}(n)\right\rangle dn\right\vert \leq\left(  C_{\omega}%
^{\prime}\right)  ^{1+\operatorname{Re}z}\Xi(n_{o}g)r(v)r(w)\leq.\left(
C_{\omega}^{\prime\prime}\right)  ^{1+\operatorname{Re}z}\Xi(n_{o})r(v)r(w).
\]
Here $v_{\nu}=v_{\bar{P},\sigma,\nu}.$
\end{proposition}

\begin{proof}
We have
\[
\int_{N_{P}}\left\langle v_{i\nu-z\rho_{P}}(nn_{o}g),w_{i\nu-\bar{z}\rho
}(n)\right\rangle dn
\]%
\[
=\int_{N_{P}}a_{\bar{P}}(nn_{o}g)^{i\nu-(1+z)\rho_{P}}a_{\bar{P}}%
(n)^{-i\nu-(1+z)\rho_{P}}\left\langle \sigma(m_{\bar{P}}(nn_{o}g))v(k(nn_{o}%
g)),\sigma(m_{\bar{P}}(n))w(k(n))\right\rangle dn
\]
thus taking absolute values we have
\[
\left\vert \int_{N_{P}}\left\langle v_{i\nu-z\rho_{P}}(nn_{o}g),w_{i\nu
-\bar{z}\rho}(n)\right\rangle dn\right\vert
\]%
\[
\leq\int_{N_{P}}a_{\bar{P}}(nn_{o}g)^{-(1+\operatorname{Re}z)\rho_{P}}%
a_{\bar{P}}(n)^{-(1+\operatorname{Re}z)\rho_{P}}\left\langle \sigma(m_{\bar
{P}}(nn_{o}g))v(k(nn_{o}g)),\sigma(m_{\bar{P}}(n))w(k(n))\right\rangle dn.
\]
The previous lemma implies that this expression is
\[
\leq C_{\omega}^{2+m}\int_{N_{P}}a_{\bar{P}}(nn_{o})^{-(1+\operatorname{Re}%
z)\rho_{P}}a_{\bar{P}}(n)^{-(1+\operatorname{Re}z)\rho_{P}}\left\vert
\left\langle \sigma(m_{\bar{P}}(nn_{o}g))v(k(nn_{o}g)),\sigma(m_{\bar{P}%
}(n))w(k(n))\right\rangle \right\vert dn
\]%
\[
\leq C_{\omega}^{2+m}\int_{N_{P}}a_{\bar{P}}(nn_{o})^{-\rho_{P}}a_{\bar{P}%
}(n)^{-\rho_{P}}\left\vert \left\langle \sigma(m_{\bar{P}}(nn_{o}%
g))v(k(nn_{o}g)),\sigma(m_{\bar{P}}(n))w(k(n))\right\rangle \right\vert dn
\]
since $\alpha_{\bar{P}}(n)^{\rho_{P}}\geq1$ for $n\in N_{o}$. Now applying the
lemma again this expression is
\[
\leq C_{\omega}^{3+m}\int_{N_{P}}a_{\bar{P}}(nn_{o}g)^{-\rho_{P}}a_{\bar{P}%
}(n)^{-\rho_{P}}\left\vert \left\langle \sigma(m_{\bar{P}}(nn_{o}%
g))v(k(nn_{o}g)),\sigma(m_{\bar{P}}(n))w(k(n))\right\rangle \right\vert dn.
\]
Also Theorem \ref{Sun} implies
\[
\left\vert \left\langle \sigma(m_{\bar{P}}(nn_{o}g))v(k(nn_{o}g)),\sigma
(m_{\bar{P}}(n))w(k(n))\right\rangle \right\vert \leq\Xi_{M_{P}}(m_{P}%
(n)^{-1}m_{P}(nn_{o}g))r(v)r(w).
\]
To complete the proof we need to show that
\[
\int_{N_{P}}a_{\bar{P}}(ng)^{-\rho_{P}}a_{\bar{P}}(n)^{-\rho_{P}}\Xi_{M_{P}%
}(m_{\bar{P}}(n)^{-1}m_{\bar{P}}(ng))dn=\Xi(g).
\]
Indeed, note then
\[
\Xi(g)=\left\langle \pi_{\bar{P}_{o},1,0}(g)1,1\right\rangle .
\]
Set $\sigma=\pi_{\bar{P}_{o},1,0}$ then induction in stages implies that if
$v\in I_{\sigma}^{\infty}$ corresponds to $1$ then using the notation
\[
v_{0}=_{{}\bar{P}}v_{\sigma,0}%
\]
we have
\[
\Xi(g)=\left\langle \pi_{\bar{P}_{o}\cap M_{p},\sigma,0}(g)v,v\right\rangle
=\int_{N_{P}}\left\langle v_{0}(ng),v_{0}(n)\right\rangle dn
\]%
\[
=\int_{N_{P}}a_{\bar{P}}(ng)^{-\rho_{P}}a_{\bar{P}}(n)^{-\rho_{P}}\left\langle
\sigma(m_{\bar{P}}(ng)1,\sigma(m_{\bar{P}}(n))1\right\rangle dn
\]%
\[
=\int_{N_{P}}a_{\bar{P}}(ng)^{-\rho_{P}}a_{\bar{P}}(n)^{-\rho_{P}}\left\langle
\sigma(m_{\bar{P}}(n))^{-1}\sigma(m_{\bar{P}}(ng)1,1\right\rangle dn
\]%
\[
=\int_{N_{P}}a_{\bar{P}}(ng)^{-\rho_{P}}a_{\bar{P}}(n)^{-\rho_{P}}\Xi_{M_{P}%
}(m_{\bar{P}}(n)^{-1}m_{\bar{P}}(ng))dn
\]

\end{proof}

Observing that if $v_{o}$ is a unit vector in $\wedge^{n}Lie(\bar{N}_{o}),$
with $n=\dim\bar{N}_{o}$ , and if $g\in G$
\[
a_{\bar{P}_{o}}(g)^{\rho_{o}}\leq\left\Vert g^{-1}v_{o}\right\Vert
\leq\left\Vert g^{-1}\right\Vert =\left\Vert g\right\Vert .
\]
Using inequality 7. in section 3 we have

\begin{lemma}
$\Xi(g)\leq C_{2}a_{\bar{P}_{o}}(g)^{-\rho_{o}}(1+\log\left\Vert g\right\Vert
)^{d.}$.
\end{lemma}

This leads to the following variant of the above proposition.

\begin{corollary}
\label{N-extimate}Notation as in the previous corollary. If $v,w\in I_{\sigma
}^{\infty}$, if $\omega$ is a compact subset of $G$ and if $g\in\omega
,n_{o}\in N_{o}$ and $\operatorname{Re}z\,\geq0$ then
\[
\left\vert \int_{N_{P}}\left\langle v_{i\nu-z\rho_{P}}(nn_{o}g),w_{i\nu
-\bar{z}\rho_{P}}(n)\right\rangle dn\right\vert \leq.C_{2}\left(  C_{\omega
}^{\prime}\right)  ^{1+\operatorname{Re}z}a_{\bar{P}_{o}}(n_{o})^{-\rho_{o}%
}(1+\log\left\Vert n_{o}\right\Vert )^{d.}r(v)r(w).
\]

\end{corollary}

\section{The condition of regularity}

Let the notation be as in the proof of Theorem \ref{key-Theorem} we use the
continuity of $J_{\chi_{Y}}(P,\sigma,\nu)(\lambda)$ in $Y$ for $Y\in U$ (that
is $\chi_{Y}$ if generic). Here the character $\chi_{Y}$ on $N_{o}$ is given
by%
\[
\chi_{Y}(\exp X)=e^{-iB(\theta Y,X)}.
\]
One can prove this directly by running through the proof of the holomorphic
continuation of the Jacquet integral in sections 15.5,15.6,15.7 of
\cite{RRGII} keeping track of the dependence in $Y$. In this appendix we will
give a proof using an idea hinted at by the referee for split groups and real
rank one groups. To prove this we will need to have a better idea of the
implications of the condition that $\chi_{Y}$ is generic.

The definition that we gave in Section 4 of generic was that if $\alpha$ is a
simple root in $\Phi(P_{o},A_{o})$ (the roots of $A_{o}$ acting on
$\mathfrak{n}_{o}$) then $d\chi_{|\left(  \mathfrak{n}_{o}\right)  _{\alpha}%
}\neq0$. Let $M_{o}$ act on the space $Z$ via $Ad_{|M_{o}}$ $Z$ in the proof
of Theorem \ref{key-Theorem} (the sum of the simple root spaces). The
following result is well known and due to Kostant and Rallis but is hard to
directly reference the result (with a proof) over $\mathbb{R}$ we include a
simple (elementary) proof of it based on the representation theory of
$SL(2,\mathbb{R})$.

\begin{lemma}
Let $Z$ and $U$ be as in the proof of Theorem \ref{key-Theorem}. Then the
orbits of $M_{o}$ on $U$ are open in $Z$.
\end{lemma}

Recall that if $\{x,y,h\}$ is a basis of a Lie algebra, $\mathfrak{s}$, over
$\mathbb{R}$ with commutation relations
\[
\lbrack x,y]=h,[h,x]=2x,[h,y]=-2y
\]
and if $V$ is a finite dimensional $\mathfrak{s}$--module then $h$ acts
semi--simply on $V$ with integral eigenvalues. Thus $h$ induces a grade on $V
$, that is, $V=\oplus_{k}V_{k}$ with $V_{k}$ the $k$--eigenspace. The other
fact we will use is that%

\[
xV_{k}=V_{k+2},k\in\mathbb{Z}.
\]

Let $Z$ and $U\subset Z\subset\mathfrak{n}_{o}$ be as in the proof of Theorem
\ref{key-Theorem}. If $x\in U$ then (by definition) $x=\sum_{\alpha\in\Delta
}x_{\alpha}$ with $x_{\alpha}\in\mathfrak{g}_{\alpha}$ and $x_{\alpha}\neq0$
for $\alpha\in\Delta$. We note that if $\alpha,\beta\in\Delta$ then
\[
\lbrack\theta(\mathfrak{g}_{\beta}),\mathfrak{g}_{\alpha}]=0,\alpha\neq\beta.
\]
Also,
\[
\lbrack x_{\alpha},\theta x_{\alpha}]\in\mathfrak{a}_{o}.
\]
So%
\[
\lbrack x,\theta(\sum c_{\alpha}x_{\alpha})]=\sum c_{\alpha}[x_{\alpha},\theta
x_{\alpha}].
\]
If $\alpha\in\Delta$ let $h_{a}\in\mathfrak{a}_{o}$ be defined by
\[
B(h_{\alpha},u)=\alpha(u),u\in\mathfrak{a}_{o}.
\]
Noting that%
\[
\beta([x,\theta(\sum c_{\alpha}x_{\alpha})])=\sum c_{\alpha}B([x_{\alpha
},\theta x_{\alpha}],h_{\beta})
\]%
\[
=\sum c_{\alpha}B(x_{\alpha},[h_{\beta,},\theta x_{\alpha}])=-\sum c_{\alpha
}B(x,\theta x_{\alpha})(\alpha,\beta),
\]
There exist $c_{\alpha}\in\mathbb{R}$ such that if $y=\theta(\sum c_{\alpha
}x_{\alpha})$ then $\beta([x.y])=2$ for all $\beta\in\Delta$. This implies
that $x,y,h=[x,y]$ is a basis of a Lie algebra $\mathfrak{s}$ as above. We are
ready to prove the Lemma.

\begin{proof}
(of the Lemma) Let $x\in U$ and let $h,y,\mathfrak{s}$ be as above.
$\mathfrak{g}$ \ is an $\mathfrak{s}$--module under $ad$. Let $\mathfrak{g}%
=\oplus\mathfrak{g}_{k}$ be the corresponding grade. Since $h$ is a regular
element in $\mathfrak{a}_{o}$, we have $\mathfrak{g}_{0}=Lie(M_{o}),$ also
$\mathfrak{g}_{2}=Z$. Thus $ad(Lie(M_{o}))x=Z$. This implies that $Ad(M_{o})x$
is open.
\end{proof}

We return to the situation of the proof of Theorem \ref{key-Theorem}. Let $P$
be a standard parabolic subgroup of $G$ and let $\sigma$ be a square
integrable representation of ${}^{o}M$. Let $U_{M}$ be the set of regular
elements of $Z\cap Lie(M).$

\begin{lemma}
\label{Locallsigma}Fix a connected component of $U_{M}$, $U^{\prime}$ then
there exists $d$ such that $\dim Wh_{\chi_{x}}(H_{\sigma}^{\infty})=d$ for
$x\in U^{\prime}$. If $y\in U^{\prime}$ then there exists $U_{y}^{\prime}$ an
open neighborhood of $y$ in $U^{\prime}$ and $\lambda_{1}....,\lambda_{d},$
weakly continuous functions
\[
\lambda_{i}:U_{y}^{\prime}\rightarrow\left(  H_{\sigma}^{\infty}\right)
^{\prime},
\]
such that for each $x\in U_{y}^{\prime}$, $\lambda_{i}(x)\in Wh_{\chi_{x}%
}(H_{\sigma}^{\infty})$ and $\lambda_{1}(x),...,\lambda_{d}(x)$ is a basis of
$Wh_{\chi_{x}}(H_{\sigma}^{\infty})$ for each $x\in U_{y}^{\prime}$.
\end{lemma}

\begin{proof}
Let $x_{o}\in U^{\prime}$ and let $\mu_{1},...,\mu_{d}$ be a basis of
$Wh_{\chi_{x_{o}}}(H_{\sigma}^{\infty})$. If $x\in U^{\prime}$ then
$x=Ad(\theta(m))x_{o}$ for some $m\in M_{o}\cap M$. Consider, $\xi_{i}=\mu
_{i}\circ\sigma(m)$. Then $\xi_{i}\in Wh_{\chi_{x}}(H_{\sigma}^{\infty})$ and
are clearly linearly independent. Since we can do the same starting with $x$
and using $m^{-1}$ we see that $\dim Wh_{\chi_{x_{o}}}(H_{\sigma}^{\infty})=d$
all $x\in U^{\prime}.$ Let $y\in U^{\prime}$ and let $W$ be a submanifold of
$M_{o}\cap M$ such that $Ad(W)y$ is an open neighborhood, $U_{y}^{\prime},$of
$y$ in $U^{\prime}$ and the map
\[
W\rightarrow Ad(W)y
\]
given by $w\mapsto Ad(w)y$ is a diffeomorphism. Let $\psi$ denote the inverse
diffeomophism. Let $\mu_{1},...,\mu_{d}$ be a basis of $Wh_{\chi_{y}%
}(H_{\sigma}^{\infty})$ and define $\lambda_{i}(x)=\mu_{i}\circ\sigma
(\psi(x))$ for $x\in U_{y}^{\prime}.$
\end{proof}

The main result of this appendix is

\begin{proposition}
\label{Continuity}Notation as above. Let $U^{\prime}$ be a connected component
of $U$ and let $\lambda:U^{\prime}\rightarrow\left(  H_{\sigma}^{\infty
}\right)  ^{\prime}$ be weakly continuous and such that $\lambda(x)\in
Wh_{\chi_{x}}(H_{\sigma}^{\infty})$ for $x\in U^{\prime}$. Then the map%
\[
x,\nu\mapsto J_{\chi_{x}}(P,\sigma,\nu)(\lambda(x))
\]
is strongly continuous from $U^{\prime}\times(\mathfrak{a}_{P})_{\mathbb{C}%
}^{\ast}$ to $(I_{\bar{P},\sigma}^{\infty})^{\prime}$ and holomorphic in $\nu$.
\end{proposition}

\begin{proof}
We note that if $m\in M_{o}$ and $m=a{}^{o}m$ with $a\in A_{P}$ and$\ {}%
^{o}m\in{}^{o}M_{P}$ then%
\[
J_{\chi_{x}}(P,\sigma,\nu)(\lambda(x))\circ\pi_{\bar{P},\sigma,\nu
}(m))=a^{-2\rho_{P}+\nu}J_{\chi_{Ad(\theta m)y}}(P,\sigma,\nu)(\lambda
(x)\circ\sigma(m)).
\]
Now argue as in the proof of Lemma \ref{Locallsigma} using cross sections of
the orbit map of $M_{o}$ onto $U^{\prime}$.
\end{proof}

\end{document}